\documentclass[10pt,a4paper,pointlessnumbers]{scrartcl}

\usepackage{lmodern}
\usepackage[T1]{fontenc}

\usepackage{amssymb}
\usepackage{amsmath}
\usepackage{amsfonts}
\usepackage{bbm}	
\usepackage{amsthm}
\usepackage{mathrsfs}
\usepackage{xcolor}
 \usepackage[pdftex,hidelinks,
     bookmarks         = true,
     bookmarksnumbered = true,
     pdfstartview      = FitH,
     pdfpagelayout     = SinglePage,
     urlcolor          = blue,
     linkcolor          = blue,
     citecolor=teal,
     colorlinks,
 	pdftoolbar = true
     ]{hyperref}

\usepackage[margin=2.2cm]{geometry}
\usepackage[all,cmtip]{xy}
\usepackage[utf8]{inputenc}
\usepackage{graphicx}
\usepackage{varwidth}
\usepackage{overpic}

\makeatletter
\DeclareRobustCommand{\em}{%
	\@nomath\em \if b\expandafter\@car\f@series\@nil
	\normalfont \else \slshape \fi}
\makeatother
\usepackage{upgreek}	
\usepackage{rotating}
\usepackage{tikz}
\usetikzlibrary{matrix,arrows,decorations.pathmorphing,shapes.geometric}
\usepackage{tikz-cd}
\usepackage{needspace}
\newcommand{\spaceplease}{\needspace{5\baselineskip}}

\usetikzlibrary{decorations.markings}

\usetikzlibrary{backgrounds}

\newcommand{\SkAlg}{\catf{SkAlg}}
\newcommand{\Sk}{\catf{Sk}}
\usepackage{todonotes}

\newcommand{\cA}{\cat{A}}
\newcommand{\cO}{\mathcal{O}}

\newcommand{\SURF}{\catf{SURF}}

\newcommand{\moduli}{\mathfrak{MF}}

\usepackage{tikz-cd}

\usetikzlibrary{decorations.markings}
\usetikzlibrary{backgrounds}

\usepackage{etex}

\tikzstyle{tikzfig}=[baseline=-0.25em,scale=0.5]

\pgfkeys{/tikz/tikzit fill/.initial=0}
\pgfkeys{/tikz/tikzit draw/.initial=0}
\pgfkeys{/tikz/tikzit shape/.initial=0}
\pgfkeys{/tikz/tikzit category/.initial=0}

\pgfdeclarelayer{edgelayer}
\pgfdeclarelayer{nodelayer}
\pgfsetlayers{background,edgelayer,nodelayer,main}

\tikzstyle{none}=[inner sep=0mm]

\tikzstyle{every loop}=[]

\usepackage{tikzit}

\tikzstyle{black dot}=[fill=black, draw=black, shape=circle, minimum size=3pt, inner sep=0pt]
\tikzstyle{black dot small}=[fill=black, draw=black, shape=circle, minimum size=2pt, inner sep=0pt]
\tikzstyle{fblack dot}=[fill=black, draw=red, shape=circle, minimum size=2pt, inner sep=0pt]
\tikzstyle{wbox}=[fill=white, draw=black, shape=rectangle, minimum height=0.5cm, minimum width=0.01cm]
\tikzstyle{bbox}=[fill=white, draw=blue, shape=rectangle, minimum height=0.5cm, minimum width=0.01cm]
\tikzstyle{rbox}=[fill=white, draw=red, shape=rectangle, minimum height=0.5cm, minimum width=0.01cm]
\tikzstyle{bwbox}=[draw=blue, shape=rectangle, minimum width=2cm, minimum height=0.5cm]
\tikzstyle{bbwbox}=[draw=blue, shape=rectangle, minimum width=1cm, minimum height=1cm]
\tikzstyle{big white circle}=[fill=white, draw=black, shape=circle, minimum width=0.75cm]
\tikzstyle{white dot big}=[fill=white, draw=black, shape=circle, inner sep=1pt]
\tikzstyle{white dot}=[fill=white, draw=black, shape=circle, minimum size=3pt, inner sep=0pt]
\tikzstyle{flat box}=[fill=white, draw=black, shape=rectangle, minimum width=1.3cm, minimum height=0.5cm,fill=morphismcolor]
\tikzstyle{square}=[fill=white, draw=black, shape=rectangle]
\tikzstyle{flat box 2}=[fill=white, draw=black, shape=rectangle, minimum height=0.5cm, minimum width=0.01cm,fill=morphismcolor]
\tikzstyle{bigbox}=[fill=white, draw=black, shape=rectangle, minimum height=0.5cm, minimum width=0.8cm,fill=morphismcolor]
\tikzstyle{over }=[front]
\tikzstyle{theta}=[fill=blue, draw=blue, shape=ellipse, minimum height=6pt, minimum width=6pt, inner sep=0pt]
\tikzstyle{thetabig}=[fill=blue, draw=blue, shape=ellipse, minimum width=1cm, minimum height=0.01cm]
\tikzstyle{thetainv}=[fill=blue, draw=red, shape=ellipse, minimum height=6pt, minimum width=6pt, inner sep=0pt]
\tikzstyle{thetabinv}=[fill=blue, draw=red, shape=ellipse, minimum width=1cm, minimum height=0.01cm]
\tikzstyle{bigdisk}=[draw=black, shape=circle, minimum width=3cm]
\tikzstyle{wdisk}=[shape=circle, minimum width=0.48cm,fill=white]
\tikzstyle{bigdisk2}=[draw=black, fill=lightgray, shape=circle, minimum width=3cm]
\tikzstyle{little disk}=[fill=white, draw=black, shape=circle, minimum width=0.5cm]

\tikzstyle{mid arrow}=[-, postaction={on each segment={mid arrow}}]
\tikzstyle{end arrow}=[->]
\tikzstyle{mover}=[-, link]
\tikzstyle{string}=[-, draw=blue,postaction={on each segment={mid arrow}}]
\tikzstyle{stringd}=[-, dotted,draw=blue,postaction={on each segment={mid arrow}}]
\tikzstyle{red}=[-, dotted,draw=red]
\tikzstyle{mydots}=[-,dotted]
\tikzstyle{open}=[-, line width=1pt,draw=blue,postaction={on each segment={mid arrow}}]
\tikzstyle{thick}=[-,line width=1pt]
\tikzstyle{rarrow}=[->,draw=red]
\tikzstyle{red mid arrow}=[-, draw={rgb,255: red,214; green,42; blue,51}, postaction={on each segment={mid arrow}}, line width=1pt]
\tikzstyle{RED}=[-, draw={rgb,255: red,214; green,42; blue,51}]
\tikzstyle{darrow}=[->,dotted]
\tikzstyle{blue}=[-, draw=blue]
\tikzstyle{blue mid arrow}=[-, draw={rgb,255: red,23; green,37; blue,167}, postaction={on each segment={mid arrow}}, line width=1pt]
\tikzstyle{bluedashed}=[-,dotted, draw=blue]
\tikzstyle{over}=[-, link]
\tikzstyle{mover}=[-, link]
\tikzstyle{mapsto}=[{|->}]
\usetikzlibrary{decorations.pathreplacing,decorations.markings}
\tikzset{
	on each segment/.style={
		decorate,
		decoration={
			show path construction,
			moveto code={},
			lineto code={
				\path [#1]
				(\tikzinputsegmentfirst) -- (\tikzinputsegmentlast);
			},
			curveto code={
				\path [#1] (\tikzinputsegmentfirst)
				.. controls
				(\tikzinputsegmentsupporta) and (\tikzinputsegmentsupportb)
				..
				(\tikzinputsegmentlast);
			},
			closepath code={
				\path [#1]
				(\tikzinputsegmentfirst) -- (\tikzinputsegmentlast);
			},
		},
	},
	mid arrow/.style={postaction={decorate,decoration={
				markings,
				mark=at position .7 with {\arrow[#1]{stealth}}
	}}},
}
\tikzset{%
	link/.style    = { white, double = black, line width = 1.8pt,
		double distance = 0.4pt },
	channel/.style = { white, double = black, line width = 0.8pt,
		double distance = 0.8pt },
}

\tikzstyle{tikzfig}=[baseline=-0.25em,scale=0.5]

\pgfkeys{/tikz/tikzit fill/.initial=0}
\pgfkeys{/tikz/tikzit draw/.initial=0}
\pgfkeys{/tikz/tikzit shape/.initial=0}
\pgfkeys{/tikz/tikzit category/.initial=0}

\pgfdeclarelayer{edgelayer}
\pgfdeclarelayer{nodelayer}
\pgfsetlayers{background,edgelayer,nodelayer,main}

\tikzstyle{none}=[inner sep=0mm]

\tikzstyle{every loop}=[]

\usepackage{mathtools}
\mathtoolsset{showonlyrefs}
\usepackage{epstopdf}

\newtheoremstyle{mytheorem}
{\topsep}
{\topsep}
{\slshape}
{0pt}
{\bfseries}
{.}
{ }
{\thmname{#1}\thmnumber{ #2}\thmnote{ {\normalfont\slshape(#3)}}}

\newtheoremstyle{mydefinition}
{\topsep}
{\topsep}
{\normalfont}
{0pt}
{\bfseries}
{.}
{ }
{\thmname{#1}\thmnumber{ #2}\thmnote{ {\normalfont\slshape(#3)}}}

\theoremstyle{mytheorem}
\newtheorem{theorem}{Theorem}[section]

\makeatletter
\newtheorem*{rep@theorem}{\rep@title}
\newcommand{\newreptheorem}[2]{%
	\newenvironment{rep#1}[1]{%
		\def\rep@title{#2 \ref{##1}}%
		\begin{rep@theorem}}%
		{\end{rep@theorem}}}
\makeatother

\newreptheorem{theorem}{Theorem}
\newreptheorem{corollary}{Corollary}

\newtheorem{lemma}[theorem]{Lemma}
\newtheorem{proposition}[theorem]{Proposition}
\newtheorem{corollary}[theorem]{Corollary}
\theoremstyle{mydefinition}\makeatletter
\newtheorem*{repx@theorem}{\repx@title}
\newcommand{\newrepxtheorem}[2]{%
	\newenvironment{repx#1}[1]{%
		\def\repx@title{#2 \ref{##1}}%
		\begin{repx@theorem}}%
		{\end{repx@theorem}}}
\makeatother
\newtheorem{definition}[theorem]{Definition}\newrepxtheorem{definition}{Definition}
\newtheorem{exx}[theorem]{Example}
\newenvironment{example}
{\pushQED{\qed}\exx}
{\popQED\endexx}
\newtheorem{remm}[theorem]{Remark}
\newenvironment{remark}
{\pushQED{\qed}\remm}
{\popQED\endremm}
\numberwithin{equation}{section}
\usepackage{enumitem}

\newenvironment{pnum}{\begin{enumerate}[topsep=2pt,parsep=2pt,partopsep=2pt,itemsep=0pt,label={(\roman{*})}]}{\end{enumerate}}

\DeclareMathSymbol{\Phiit}{\mathalpha}{letters}{"08}\let\Phi\undefined\newcommand{\Phi}{\Phiit}
\DeclareMathSymbol{\Psiit}{\mathalpha}{letters}{"09}\let\Psi\undefined\newcommand{\Psi}{\Psiit}
\DeclareMathSymbol{\Sigmait}{\mathalpha}{letters}{"06}\let\Sigma\undefined\newcommand{\Sigma}{\Sigmait}
\DeclareMathSymbol{\Xiit}{\mathalpha}{letters}{"04}
\DeclareMathSymbol{\Lambdait}{\mathalpha}{letters}{"03}\let\Lambda\undefined\newcommand{\Lambda}{\Lambdait}
\DeclareMathSymbol{\Piit}{\mathalpha}{letters}{"05}\let\Pi\undefined\newcommand{\Pi}{\Piit}
\DeclareMathSymbol{\Gammait}{\mathalpha}{letters}{"00}\let\Gamma\undefined\newcommand{\Gamma}{\Gammait}
\DeclareMathSymbol{\Omegait}{\mathalpha}{letters}{"0A}\let\Omega\undefined\newcommand{\Omega}{\Omegait}
\DeclareMathSymbol{\Upsilonit}{\mathalpha}{letters}{"07}\let\Upsilon\undefined\newcommand{\Upsilon}{\Upilonit}
\DeclareMathSymbol{\Thetait}{\mathalpha}{letters}{"02}\let\Theta\undefined\newcommand{\Theta}{\Thetait}

\def\Hom{\mathrm{Hom}}

\def\End{\catf{End}}
\def\id{\mathrm{id}}
\def\SL{\operatorname{SL}}

\def\k{\mathbf{k}}
\newcommand{\Vmod}{V\! \text{-}\catf{mod}}

\let\to\undefined\newcommand{\to}{\longrightarrow}
\let\mapsto\undefined\newcommand{\mapsto}{\longmapsto}
\newcommand{\catf}[1]{\mathsf{#1}}

\newcommand{\Map}{\catf{Map}}
\newcommand{\Diff}{\catf{Diff}}
\newcommand{\MapA}{\catf{Map}_{\! \cat{A}}}

\def\op{\mathrm{op}}

\newcommand{\framed}{\catf{f}E_2}

\newcommand{\ra}[1]{\xrightarrow{\   #1    \ }}

\def\Grpd{\catf{Grpd}}

\newcommand{\Lexf}{\catf{Lex}^\catf{f}}
\newcommand{\Rexf}{\catf{Rex}^\catf{f}}
\newcommand{\Rex}{\catf{Rex}}

\def\Cat{\catf{Cat}}

\newcommand{\ExtSurf}{\catf{Ext}^\circ\catf{(Surf)}}

\newcommand{\Gr}{\catf{Gr}}

\newcommand{\FA}{\mathfrak{F}_{\! \cat{A}}}

\newcommand{\hocolimsub}[1]{\underset{#1}{\operatorname{hocolim}}\,}

\newcommand{\smc}{\cat{S}}

\newcommand{\Fin}{\catf{Lex}^\catf{f}}

\newcommand{\vect}{\catf{vect}}

\newcommand{\spr}[1]{\left\langle #1\right\rangle}
\newcommand{\cat}[1]{\mathcal{#1}}

\newcommand{\Ext}{\catf{Ext}^\circ}

\newcommand{\Hbdy}{\catf{Hbdy}}

\newcommand{\Surf}{\catf{Surf}}

\newcommand{\Graphs}{\catf{Graphs}}

\newcommand{\Legs}{\catf{Legs}}
\newcommand{\Add}{\catf{Add}}

\newcommand{\RForests}{\catf{RForests}}
\newcommand{\Forests}{\catf{Forests}}

\newcommand{\Aut}{\catf{Aut}}

\newcommand{\SurfA}{\Surf_{\! \cat{A}}}

\newcommand{\SurfB}{\Surf_{\! \cat{B}}}
\newcommand{\SurfAB}{\Surf_{\! \cat{A},\cat{B}}}
\newcommand{\PhiA}{\Phi_{\! \cat{A}}}

\newcommand{\ha}[1]{\xhookrightarrow{ \ #1 \ }}

\newcommand{\MF}{\catf{MF}}
\newcommand{\opint}{[0,1]}

\definecolor{Blue}  {rgb} {0.282352,0.239215,0.803921}
\definecolor{Green} {rgb} {0.133333,0.545098,0.133333}
\definecolor{Red}   {rgb} {0.803921,0.000000,0.000000}
\definecolor{Violet}{rgb} {0.580392,0.000000,0.827450}

\usepackage{dingbat}

\renewcommand{\today}{\ifcase \month \or January\or February\or March\or %
	April\or May \or June\or July\or August\or September\or October\or November\or %
	December\fi {} \number  \year} 

\makeatletter
\newcommand{\monthyeardate}{%
	\DTMenglishmonthname{\@dtm@month}, \@dtm@year
}
\makeatother

\setkomafont{caption}{\small\slshape}

\setkomafont{section}{\rmfamily\large}
\setkomafont{subsection}{\rmfamily}
\setkomafont{subsubsection}{\rmfamily}
\setkomafont{subparagraph}{\normalfont\scshape}

\newtheorem*{theorem*}{Theorem}
\newtheorem*{corollary*}{Corollary}

\makeatletter
\renewcommand\section{\@startsection {section}{1}{\z@}%
	{-3.5ex \@plus -1ex \@minus -.2ex}%
	{2.3ex \@plus.2ex}%
	{\normalfont\scshape\centering}}
\makeatother

\usepackage{titlesec}
\titleformat{\subsection}[runin]
{\normalfont\bfseries}
{\thesubsection}
{0.5em}
{}
[.]

\usepackage{titletoc}
\dottedcontents{section}[1em]{\rmfamily}{1.5em}{0.2cm}
\dottedcontents{subsection}[5em]{\rmfamily}{1.9em}{0.2cm}

\begin{document}

			\begin{center}	\textbf{\large{A Classification of Modular Functors via Factorization Homology}}\\	\vspace{1cm}	{\large Adrien Brochier $^{a}$} \ {\large and} \ \ {\large Lukas Woike $^{b}$}\\ 	\vspace{5mm}{\slshape $^a$ Institut de Mathématiques de Jussieu-Paris Rive Gauche\\ UMR 7586\\ Université Paris Cité\\ Sorbonne Université\\
					CNRS\\ F-75013 Paris \\ France}	\\[7pt]	{\slshape $^b$ Université Bourgogne Europe\\ CNRS\\ IMB UMR 5584\\ F-21000 Dijon\\ France  }
\end{center}	\vspace{0.3cm}	
\begin{abstract}\noindent Modular functors are traditionally defined as systems of projective representations of mapping class groups of surfaces that are compatible with gluing. They can formally be described as modular algebras over certain extensions of the modular surface operad, with the values of the algebra lying in a suitable symmetric monoidal $(2,1)$-category $\mathcal{S}$ of linear categories. In this paper, we prove that modular functors in $\mathcal{S}$ are equivalent to self-dual balanced braided algebras $\mathcal{A}$ in $\mathcal{S}$ (a categorification of the notion of a commutative Frobenius algebra) for which a condition formulated in terms of factorization homology with coefficients in $\mathcal{A}$ is satisfied; we call such $\mathcal{A}$ connected. The equivalence in one direction is afforded by genus zero restriction. Our construction of the inverse equivalence is entirely topological and can be thought of as a far reaching generalization of the construction of modular functors from skein theory. In order to verify the connectedness condition in practice, we prove that it can be reduced to a single condition in genus one. Moreover, we show that cofactorizability of $\mathcal{A}$, a condition known to be satisfied for modular categories, is sufficient. Therefore, we recover in particular Lyubashenko's construction of a modular functor from a (not necessarily semisimple) modular category and show that it is determined by its genus zero part. Additionally, we exhibit modular functors that do not come from modular categories and outline applications to the theory of vertex operator algebras.	\end{abstract}

	\tableofcontents
	\normalsize

	\spaceplease
	\section{Introduction and summary}
	A \emph{modular functor}, in its simplest form, is a system of projective representations
	of mapping class groups of surfaces on	 vector spaces, the so-called \emph{spaces of conformal blocks}~\cite{Segal,ms89,turaev,tillmann,baki}.
	These are
	subject to compatibility conditions with respect to the gluing of surfaces.  
The notion of a modular functor grew out of attempts to axiomatically capture (aspects of)
the mathematical structure
 underlying conformal field theory and has, over the last three decades, developed into one of the key concepts
 at the cross-roads of low-dimensional topology, representation theory and mathematical physics. 
	More specifically, the reasons for the interest in modular functors include the following:
	
	\begin{itemize}
		\item Modular functors provide a rich source for highly 
		non-trivial and explicitly computable mapping class group representations. 
		A landmark result is the asymptotic faithfulness 
		of the mapping class group representations built from certain quantum groups 
		established by Andersen~\cite{andersenfaithful}
		and Freedman-Walker-Wang~\cite{fww} by different methods. 
		
		\item Modular functors are closely related to three-dimensional topological
		 field theory. In fact, the Reshe\-tikhin-Turaev construction \cite{rt1,rt2}
		provides in particular a class of modular functors. The input datum for this construction is a \emph{semisimple 	modular category} (also called \emph{modular fusion category}), a certain type of finitely semisimple braided tensor category 
		introduced by Moore-Seiberg~\cite{ms89,mooreseiberg} and Turaev~\cite{turaevmod,turaev}
		that can 
		be seen as a categorical version of \emph{modular data}  
		(however, the notion of a modular category is strictly finer \cite{mignardschauenburg}).
		Modular categories can for instance be obtained by taking modules over suitable quantum groups \cite{turaev,baki,egno,kassel} or vertex operator algebras~\cite{Huang2008}. 
		But
		 modular functors exist under far more general assumptions than three-dimensional topological field theories. Most importantly, they can also be built from possibly \emph{non-semisimple} modular categories as demonstrated  by Lyubashenko \cite{lyubacmp,lyu} while the construction of the three-dimensional topological field theory requires semisimplicity by \cite{BDSPV15} (see however \cite{gai} for a partially defined topological field theory using modified traces). Modular functors also provide a topological perspective on various homological invariants associated with these non-semisimple categories~\cite{svea2,dmf,mwdiff}.   
		
		\item From the perspective of mathematical physics, modular functors 
		form one of the key mathematical structures governing two-dimensional
		conformal field theory. For example, the famous \emph{Verlinde formula}~\cite{verlinde,mose88,cardy,witten,turaev}, i.e.\ the diagonalization of the fusion rules through the so-called \emph{$S$-matrix}, is ultimately a statement about the orbit of the fusion multiplication on the
		space of conformal blocks of the torus  under the action of the mapping class group $\SL(2,\mathbb{Z})$ of the torus. Another reason for the interest in modular functors is the fact that they provide a direct and conceptual access to the so-called \emph{correlators} of a conformal field theory~\cite{frs1,frs2,frs3,frs4,jfcs}, certain elements in the spaces of conformal blocks invariant under the mapping class group actions	 and gluing. 
		
		\end{itemize}

	The main result of this article is a classification of 
	modular functors.	To motivate our result, recall the classical theorem that two-dimensional topological field theories in the sense of Atiyah~\cite{atiyah} are classified by commutative Frobenius algebras, i.e.\ finite-dimensional commutative algebras equipped with a non-degenerate invariant pairing~\cite{abrams,kocktft}.
	 In particular, every ordinary two-dimensional topological field theory yields by evaluation on the circle and genus zero surfaces
	 a commutative Frobenius algebra. 
	 In the higher categorical framework, two-dimensional topological field theories and, more generally, modular functors
	 with values in a symmetric monoidal $(2,1)$-category $\smc$, yield by evaluation on the circle and genus zero surfaces a \emph{cyclic} algebra over the so-called \emph{framed $E_2$-operad}. 
	 Those are explicitly characterized in~\cite{cyclic}. In the familiar case where $\cat S$ is some $(2,1)$-category of linear categories, cyclic framed $E_2$-algebras
	  coincide with balanced braided monoidal categories equipped with an appropriate analog of a non-degenerate pairing. 
	  Simply put, this paper is about identifying necessary and sufficient conditions on such a cyclic framed $E_2$-algebra to produce a modular functor and the question of whether such an extension to a modular functor is unique. 
	  
	Most previously known examples of modular functors arise as the restriction to dimension two
	  of an (at least partially defined) three-dimensional topological field theory which, in turn, is constructed 
	  from a braided monoidal category satisfying strong finiteness and duality
	  conditions~\cite{rt2, baki,kl,gai}. 
	  In fact, the modular functor is often the main ingredient in these constructions, and in all of those cases coincides with the one which was defined by Lyubashenko~\cite{lyubacmp,lyu,lyulex,kl}. It is a common theme in the study of topological field theories
	   that the extension to the top dimension (here this would mean
	   producing numerical invariants of three-dimensional
	   manifolds) is the step that requires the strongest  conditions. If one is only interested in constructing modular functors, it is natural to expect that these conditions can be relaxed greatly.
	   Solving the classification problem for modular functors 
	  amounts to precisely understanding these conditions.

\subsection{Statement of the main result}
In order to state precisely
the main results of this article,
let us discuss the definition 
of a modular functor 
in more detail using the language 
of operads: Recall that an operad, as introduced by Boardman, Vogt and May \cite{bv68,mayoperad,bv73}, describes an (algebraic) structure by giving separately the needed nullary, unary, binary --- more generally, $n$-ary operations (operations with $n$ inputs and one output)
 and a prescription for how to compose them. The strength of this description lies in the fact that the $n$-ary operations do not necessarily have to form a set, but can form for instance 
 a space, a groupoid
  or a chain complex. This allows for the description of highly intricate algebraic structures.
An important family of operads,
that in fact motivated to a large extent the invention of the concept in \cite{bv68,mayoperad,bv73},
 are the \emph{operads $E_r$ of little disks of dimension $r\ge 1$}.
These are
 topological operads whose space
  of operations of arity $n$ is given by the space of embeddings of $n$ many $r$-dimensional
 disks into one $r$-dimensional 
 disk that are composed of translations and rescalings. If one additionally allows rotations, one obtains the 
\emph{framed $E_r$-operads}.

Cyclic operads, as defined by 
Getzler and Kapranov~\cite{gk}, are operads that come with a prescription how to cyclically permute the inputs of operations
 with the output, which essentially amounts to consistently erasing the 
  distinction between inputs and the output altogether.
 Modular operads \cite{gkmod} come additionally with a self-composition for operations that can also be thought of as a trace.
 
   The prototypical example of a modular operad is the groupoid-valued
   modular operad $\Surf$ of compact oriented surfaces and their mapping classes or, for short, the modular surface operad. 
     The groupoid of arity $n$ operations has connected compact oriented surfaces with $n+1$ boundary components as objects (here $n+1$ has to be read as $n$ inputs plus one output)
 and mapping classes between these surfaces as morphisms, see Section~\ref{cyclicandmodularalgebras} for details.
  Crucially, the genus zero part of $\Surf$ is equivalent to the operad $\framed$ of framed little disks. In particular, $\framed$ has a natural cyclic structure which is not entirely obvious from its original definition.
 
 Given a modular operad $\cat{O}$, one can consider modular algebras over it. These algebras take values in a suitable (higher) symmetric monoidal category $\smc$. If $\cat{O}$ is groupoid-valued, then it is natural to choose $\smc$ as a symmetric monoidal $(2,1)$-category because it is enriched over groupoids,
 and this is in fact the only case that will be relevant for the treatment of modular functors in this article. More concretely, one should imagine $\smc$ as a suitable choice of symmetric monoidal $(2,1)$-category of $\k$-linear categories, where $\k$ is a fixed algebraically closed field;
 an example is the symmetric monoidal $(2,1)$-category $\Rex$ of finitely cocomplete linear categories and finitely cocontinuous functors, equipped with the Deligne-Kelly monoidal product~\cite{Franco2013}.
 For an object $X\in \smc$, a non-degenerate symmetric pairing $\kappa : X \otimes X \to I$ (here $I$ is the unit of $\smc$) turns the endomorphism operad of $X$ into a modular operad. This modular operad is denoted by $\End_\kappa^X$. 
 The structure of a modular $\cat{O}$-algebra on $(X,\kappa)$ is then by definition
  a map $\cat{O}\to \End_\kappa^X$ of modular operads. 
 This is essentially the original definition of Getzler and Kapranov, but we have to take into account that, since we are considering bicategorical operadic algebras, it is crucial to relax all of the familiar axioms for modular operads and their algebras up to coherent isomorphism. 
 This is done in \cite{cyclic}, and we give a concise summary in Section~\ref{cyclicandmodularalgebras}.
 
 Equipped with the theory of modular operads and their algebras, one might tentatively define a modular functor with values in $\smc$ as a 
 modular algebra over $\Surf$. While in principal this seems like a reasonable definition, experience tells us that we should in fact also allow  extensions of mapping class groups. This very important point in the definition
 of a modular functor takes into account the well-known fact that the naturally occurring mapping class group representations built from representation categories tend to be \emph{projective representations} because of the so-called \emph{framing anomaly} \cite{atiyahframing}, see \cite[Section~3]{gilmermasbaum} for a detailed discussion. 
 In order to incorporate this,
  we use the following definition:
 A \emph{modular functor with values in a symmetric monoidal $(2,1)$-category $\smc$} is 
 a modular algebra over 
 an extension of $\Surf$ as a modular operad (see Section~\ref{secexthbdy} for what is meant by extension here). We emphasize that the extension is \emph{part of the data}; it can a priori be arbitrary and will only be subject to the consistency conditions
 that we give in Definition~\ref{defextension} and~\ref{deflambdaequiv}. 
 
As already mentioned, 
 the present article is concerned with an obvious problem in the study of modular functors, namely the \emph{classification of modular functors}. 
 More specifically, the present article offers a classification
	of modular functors in terms of so-called \emph{self-dual balanced braided algebras} 
	satisfying a condition formulated in terms of factorization homology. Let us first recall the following definition from~\cite{cyclic}:

	 \begin{repxdefinition}{defsda}[$\text{\cite[Definition~5.4]{cyclic}}$]
	 A \emph{self-dual balanced braided algebra} in  $\smc$ is an object $\cat{A} \in \smc$
	 together with the following structure: \begin{itemize}
	 	\item $\cat{A}$ is a \emph{balanced braided algebra} in $\smc$ --- in more detail: 
	 	\begin{itemize}
	 		\item 
	 		$\cat{A}$ has a multiplication $\mu : \cat{A} \otimes \cat{A} \to \cat{A}$ that is associative and unital up to coherent isomorphism
	 		(we denote the unit  1-morphism  by $u:I\to \cat{A}$);
	 		\item $\cat{A}$ comes with an isomorphism $c: \mu \to \mu^\op = \mu \circ \tau$ (where $\tau$ is the symmetric braiding of $\smc$) called \emph{braiding} which satisfies the hexagon relations;
	 		\item $\cat{A}$ comes with an isomorphism $\theta : \id_\cat{A} \to \id_\cat{A}$ referred to as \emph{balancing} satisfying
	 		\begin{align}
	 			\theta \circ \mu  &= c^2 \circ \mu(\theta  \otimes \theta) \ , \\
	 			\theta \circ u &= \id_u \ . 
	 		\end{align}
	 	\end{itemize}
	 	\item $\cat{A}$ comes with a non-degenerate pairing
	 	$\kappa : \cat{A} \otimes \cat{A} \to I$
	 	(a 1-morphism
	 	exhibiting $\cat{A}$ as its own dual in the homotopy category of $\smc$)
	 	together with an isomorphism $\gamma : \kappa(u\otimes \mu)\to \kappa$ subject to the following conditions:  \begin{itemize}\item The isomorphism
	 		$\kappa(u\otimes \mu(u,-))\longrightarrow \kappa(u,-)$ coming from  
	 		$\gamma$ coincides with the isomorphism coming from the unit constraint $\mu(u,-)\cong \id_\cat{A}$.
	 		\item The equality $\kappa(\theta \otimes \id_\cat{A})=\kappa(\id_\cat{A} \otimes \theta)$ holds.
	 	\end{itemize}
	 \end{itemize}
	 \end{repxdefinition}
	 
	 Self-dual balanced braided algebras in $\smc$ are equivalent to cyclic framed $E_2$-algebras in $\smc$ \cite{cyclic}.
	 It is generally not true that such a structure already extends to a modular functor. Rather, a cyclic framed $E_2$-algebra
	 uniquely extends to an \emph{ansular functor} \cite{mwansular}, a consistent system of representations of mapping class groups of handlebodies
	 (throughout assumed to be compact and oriented) with parametrized disks on their boundaries. The main result of this paper gives, roughly speaking, a necessary and sufficient condition for a given framed cyclic $E_2$-algebra to extend to a modular functor, and shows that
	 it does so
	 in a unique way, up to a contractible choice.

The main ingredient 
for understanding this necessary and sufficient condition is \emph{factorization homology} \cite{higheralgebra,andrade,AF},
a type of homology theory
with origins in~\cite{bdca} that allows us to integrate $E_n$-algebras 
in a higher symmetric monoidal category $\smc$
over  $n$-dimensional manifolds. We will throughout be interested in oriented manifolds. Therefore, we will need \emph{framed} $E_n$-algebras as coefficients (the terminology is very unfortunate here). For factorization homology 
to be well-defined, the symmetric monoidal category $\smc$, in which our framed $E_n$-algebras take their values,
needs to satisfy some rather mild technical assumptions given in Section~\ref{secprefh}, and we restrict throughout to those $\smc$ meeting the requirements
 (these conditions
  will be satisfied in the cases of interest such as $\smc=\Rex$). 
Let us briefly mention 
 the philosophy, formulated for the two-dimensional situation that is relevant for us: Factorization homology with coefficients in a framed $E_2$-algebra $\cat{A}$ in $\smc$ evaluated on a surface $\Sigma$ is an object in $\smc$ denoted by $\int_\Sigma \cat{A}$. The functor $\Sigma \mapsto \int_\Sigma \cat{A}$ is characterized through a local-to-global principle, i.e.\ an analog of the Eilenberg-Steenrod axioms: If $\Sigma$ is a disk, then $\int_{\mathbb{D}^2}\cat{A}$ is just $\cat{A}$. For a more complicated surface, $\int_\Sigma \cat{A}$ can be computed via \emph{excision}. When $\smc$ is a suitable symmetric monoidal 
$(2,1)$-category of linear categories,
$\int_\Sigma \cat{A}$ may, in physical terms, be interpreted as a refinement of the category of modules
over the algebra of observables of $\Sigma$~\cite{cg-fa}.
We give a recollection of factorization homology in Section~\ref{secprefh}.

Suppose  that
$\cat{A}$ is a 
self-dual balanced braided algebra according to Definition~\ref{defsda}. The first bullet point of the definition tells us exactly that $\cat{A}$ is a framed $E_2$-algebra in $\smc$ by the description of non-cyclic framed $E_2$-algebras in \cite{WahlThesis,salvatorewahl}. For this reason, we may consider factorization homology $\int_\Sigma \cat{A}$ for a compact oriented surface $\Sigma$ with $n$ boundary components. We prove in Proposition~\ref{propenvfh} that the universal property of factorization homology implies that for any three-dimensional handlebody $H$ 
  with boundary $\Sigma$ 
  the ansular functor of $\cat{A}$ gives rise to a 1-morphism
  \begin{align}\label{eq:phihx}
	\PhiA(H):\int_\Sigma \cat{A}\to \cat{A}^{\otimes n} 
	\end{align}
that can be interpreted as a \emph{generalized skein module for $H$}, see Section~\ref{secskeintheorysub} of this introduction and Section~\ref{secskeinalgmod} for this viewpoint.
	A crucial property of factorization homology is that it is \emph{canonically pointed}: There exists a canonical map 
\begin{align} \cat{O}_\Sigma^\cat{A} :
I \to \int_\Sigma \cat{A}
\end{align}
induced by the embedding of the empty manifold into $\Sigma$, and we show that the value on $H$ of the ansular functor associated with $\cat{A}$ is given by the composition
\begin{equation}\label{eq:distans}
	I \ra{\cat{O}_\Sigma^\cat{A}} \int_\Sigma \cat{A} \ra{\PhiA(H)} \cat{A}^{\otimes n}.
\end{equation}
The construction of 
the map~\eqref{eq:phihx} is given in Section~\ref{modenvfhsec}. 
In fact, a more refined statement can be made: It is a standard fact that the object $\int_{\partial \Sigma \times \opint } \cat{A}$ is an algebra in $\smc$ that acts on both $\int_\Sigma \cat{A}$ and $\cat{A}^{\otimes n}$. With respect to
these module structures, $\PhiA(H)$ comes canonically with the structure of a module map
(Proposition~\ref{propmodulemap}). We may now define $\Omega_\cat{A}(\Sigma)$ 
as the replete full subgroupoid (i.e.\ full and closed under isomorphisms) of the category of module maps $\int_\Sigma \cat{A}\to \cat{A}^{\otimes n}$
	spanned by the maps $\PhiA(H)$ with $H$ running over all handlebodies with boundary $\Sigma$. 
		We call
		a cyclic framed $E_2$-algebra $\cat{A}$ or, equivalently,
		a self-dual balanced braided 
		algebra in $\smc$
		\emph{connected}
		if the groupoid $\Omega_\cat{A}(\Sigma)$ 
		is connected for any surface $\Sigma$ (Definition~\ref{defconnected}).

Although this connectedness condition is fairly abstract, we prove in  Proposition~\ref{propconnected} that it can be verified in a concrete way:
Let $H,H'$ be two handlebodies
with an embedded disk
 whose boundary surface is the torus $\mathbb{T}_1^2$ with one boundary circle (when taking the boundary surface, the embedded disk is converted to a boundary circle) whose
	mapping class groups generate $\Map(\mathbb{T}_1^2)\cong B_3$, the braid group on three strands. All mapping class groups are taken relative to the embedded disk or the boundary, respectively, see Section~\ref{secprefh}. 
	Then $\cA$ is connected if and only if $\PhiA(H)$ and $\PhiA(H')$ are isomorphic as $\int_{\mathbb{S}^1\times [0,1]} \cat{A}$-module maps.
		This reduction to genus one is of course a familiar feature of the examples of modular functors constructed from modular categories. There, roughly speaking,
		 one has a natural candidate for such an isomorphism
		 (in the semisimple case, this is the so-called $S$-matrix), and the category at hand is modular if and only if this map is indeed invertible. Note however that in those examples, as explained below, this actually follows from the fact that modular categories are in particular cofactorizable, which turns out to be a \emph{genus zero} condition. This seems to be a particular feature that  makes modular categories remarkable.

	We are now in a position to state our main result, namely the classification of modular functors. Instead of just classifying them up to equivalence, 
	we will in fact describe the \emph{moduli space} $\moduli$ of modular functors 
	 whose detailed construction is given in Definition~\ref{defmoduli}.

	\begin{reptheorem}{thmproofclassmf}[Classification of modular functors]
		The moduli space $\moduli$
		of modular functors
		with values in the symmetric monoidal $(2,1)$-category $\smc$ is equivalent to 
		the 2-groupoid of connected
		self-dual balanced braided algebras $\cat{A}$
		in $\smc$. 
		This equivalence is afforded by restricting modular functors to 
		surfaces of genus zero.
	\end{reptheorem}
	
	 For a given connected
	 self-dual balanced braided algebra $\cat{A}$,
	 the modular functor
	 corresponding to $\cat{A}$ under the above equivalence is denoted by $\mathfrak{F}_\cat{A}$ and can be explicitly described, see Section~\ref{secmodularfunctorsconstruction}. In
	 Section~\ref{secapp}, we focus on the linear case: We give a formula for all spaces of conformal blocks (Corollary~\ref{corconfblocks}) and conclude that they are always finite-dimensional provided that $\cat{A}$ is a finite linear category in the sense of Etingof-Ostrik~\cite{etingofostrik}.

	Having stated the main result, let now us now expand 
	on further results, that are either needed for the proof of the main result or arise as a consequence, and elaborate on the relation to existing modular functor constructions.
	
\subsection{Uniqueness of extensions to a modular functor}	An intermediate result of independent interest
	appearing in the proof of the main result 
	 is the following uniqueness result for extensions to modular functors:

	\begin{reptheorem}{uniquenessthm}[Uniqueness of extensions]
	For any modular functor defined in genus zero, i.e.\ a cyclic framed $E_2$-algebra $\cat{A}$
	in $\smc$, the space of extensions of $\cat{A}$ to a modular functor
		is empty or contractible. In other words,
		if there is an extension of $\cat{A}$ to a modular functor, this extension is essentially unique, i.e.\ unique up to a contractible choice.
	\end{reptheorem}
	
	This uniqueness result can be seen as a generalization of a similar result of Andersen-Ueno~\cite{andersenueno}
	who consider 
	 modular functors that are built from finitely semisimple categories equipped with an involution.  
	As a small caveat, we should note that even if we restrict to a symmetric monoidal $(2,1)$-category of finitely semisimple linear categories, our definition of a modular functor
	substantially differs from the one used in~\cite{andersenueno}. Instead, our notion
	 is closer to
	the one of Tillmann~\cite{tillmann} and contains as a special case the one of Turaev~\cite{turaev} and Bakalov-Kirillov~\cite{baki}.

	\subsection{Projective representations and equivalences}
	The construction of the moduli space $\moduli$ of modular functors 
	provides in particular a convenient definition of equivalences of modular functors. 
	By definition a modular functor 
	leads to a coherent system of representations of certain extensions of mapping class groups of surfaces. 
	We allow for extensions subject only to rather mild conditions, and define $\moduli$ in such a way that if, roughly speaking,
	two modular functors  produce systems of 
	representations of possibly different extensions
	such that one factors coherently through the other, we regard them to be equivalent. 
General equivalences are given by zigzags of those kind of maps. 

For a fixed connected cyclic framed $E_2$-algebra $\cA$, the particular extension $\SurfA$ of $\Surf$ that enters the construction of the modular functor $\mathfrak F_\cA$ associated to $\cat{A}$ plays a distinguished role: For any other extension $\cat{Q}$, the modular $\cat{Q}$-algebra structures on $\cA$ are in one-to-one correspondence with genuine maps of extensions $\cat{Q}\to \SurfA$ (Theorem~\ref{thmuniversal}). Beware that this does not mean that $\SurfA$ is somehow minimal: For example, if $\cA$ can actually be given the structure of a $\Surf$-algebra, leading to genuine representations of mapping class groups, it does not mean that $\SurfA=\Surf$, but rather that there is a map of extensions $\Surf\to \SurfA$, i.e.\ a section. Different sections could a priori lead to non-equivalent $\Surf$-structures on $\cA$ which are still identified in $\moduli$.
In other words, even if the projective representations at hand can be lifted to genuine representations, the correct notion of equivalence is still that of projective equivalence.

	 The origin of the  projectiveness of the mapping class group representations is perhaps best illustrated in the  special case where $\cA$ is a modular category and $\Sigma$ is closed: For a handlebody $H$ with boundary $\Sigma$, 
	the connectedness condition guarantees that the functors $	\PhiA(H) :	\int_\Sigma \cat{A} \to \vect$
	for different choices of $H$ are isomorphic, but non-canonically.
 In that case, we also have 
\begin{align}
\Aut(\PhiA(H))\cong \Aut(\id_\vect)\cong \k^\times \ .
\end{align}
	There is a natural action of the mapping class group on $\int_\Sigma \cat{A}$ which happens to be (again, non-canonically) trivial in that particular situation, so that it can be transported to a trivial action of $\Map(\Sigma)$ on $\vect$. In the categorical setting, the action being trivial means that each group element acts by the identity functor and that the composition is controlled by a $\k^\times$-valued 2-cocycle $\eta$ on $\Map(\Sigma)$. Having fixed $H$, changing the trivialization of the $\Map(\Sigma)$-action leads to an equivalent trivial action, i.e.\ one defined by a cocycle in the same cohomology class as $\eta$. The vector space $\PhiA(H)\cat{O}_\Sigma^\cat{A}$ then carries a canonical structure of a homotopy fixed point for this action. This is well-known to be the same as a representation 
	of the central extension determined by $\eta$
	on that space, thereby producing
	  a projective representation of $\Map(\Sigma)$ on $\PhiA(H)\cat{O}_\Sigma^\cat{A}$. Those are precisely the representations that the modular functors $\mathfrak F_\cA$ produces in that special case.

	 \subsection{Cofactorizability as a sufficient condition}
	 The condition for a self-dual balanced braided algebra
	 to be connected is relatively transparent from a topological perspective, but somewhat abstract from an algebraic perspective. 
	 Therefore, a large part of the article is devoted to providing instead 
	 sufficient conditions that can be verified more easily in concrete situations.
	 Most importantly, we prove 
	 that \emph{cofactorizability} in the sense of~\cite{bjss} is a sufficient condition
	 that can be verified in genus zero:

	 \begin{reptheorem}{thmcofactorizable}
	 	For any cofactorizable cyclic framed $E_2$-algebra, there is an essentially unique extension to a modular functor. More precisely, we have
	 	an embedding of 2-groupoids
	 	\begin{align}
	 		\{\text{cofactorizable self-dual balanced braided algebras in $\smc$}\}	\ha{ \ \ \ }    \moduli  \ . 
	 	\end{align}
	 \end{reptheorem}

For general $\smc$, this embedding is not essentially surjective
(Remark~\ref{remcofnotesssurj}).

\subsection{Recovering the Lyubashenko construction}
	We have the following relevant special case
	of Theorem~\ref{thmcofactorizable}: If $\smc=\Rex$ and if the self-dual balanced braided algebra $\cat{A}$
	 is actually a finite ribbon category
	 (not every self-dual balanced braided algebra in $\Rex$ is of this form, see the next subsection and~\cite{mwansular}), then cofactorizability is equivalent to the non-degeneracy of the braiding of $\cat{A}$, which means that $\cat{A}$ is a modular category (this definition of modularity does \emph{not} include semisimplicity). 	
	 In this case, the modular functor afforded by Theorem~\ref{thmproofclassmf} agrees with Lyubashenko's modular functor~\cite{lyubacmp} (Corollary~\ref{cormodular}).
	 This provides a purely topological construction 
	 of this well-known class of modular functors 
	 via factorization homology.  
	 The new insight about this construction is its universality: For any modular category, Lyubashenko's construction yields
	 the essentially \emph{unique} extension to a modular functor.
	 This is a conceptually appealing description of Lyubashenko's construction that, to the best of our knowledge, is new. As a concrete consequence, this implies that the ribbon automorphisms of a modular category (as 2-group)
	 act on the associated modular functor.
	 The comparison to Lyubashenko's construction is fairly abstract, 
	 but it can be made explicit 
	 for ribbon categories using the results of~\cite{bzbj}, 
	 thereby recovering 
	 Lyubashenko's mapping class group representations~\cite{lyubacmp} (including the notorious $S$-transformation) and also their description in~\cite{gai2}
	  in a direct way that does not appeal to the uniqueness result. We plan
	 to  address this
	  in a forthcoming paper.

	 \subsection{Modular functors that do not come from modular categories}

	  By combining our results with those of Allen-Lentner-Schweigert-Wood~\cite{alsw}, we give in Corollary~\ref{corvoax} a criterion for a vertex operator algebra $V$ to produce a modular functor.
	  This comes with a prescription to obtain spaces of conformal blocks for $V$. These spaces of conformal
	  blocks are \emph{universal}: They are up to equivalence the only ones extending the `obvious'
	  hom space genus zero blocks to a modular functor. 
	 This construction produces modular functors 
	   that do not come from modular categories, for example the Feigin-Fuchs boson (Example~\ref{exfeiginfuchs}). 
	   As another class of not necessarily modular categories that yield modular functors, we discuss Drinfeld centers of (not necessarily spherical) pivotal finite tensor categories (Corollary~\ref{corzentrum}) building on the algebraic results of \cite{mwcenter}.
	   
	   \subsection{Relation to skein theory}\label{secskeintheorysub}
	   Our main result  is largely inspired by and specializes to the skein-theoretic approach to modular functors~\cite{roberts,masbaumroberts}. To illustrate this, let $\cat{A}$ be the modular fusion category obtained by taking the semisimplifcation of the representation category of quantum $\text{SL}_2$ at a root of unity. Let $\Sigma$ be a closed surface and $H$ a handlebody  with $\partial H=\Sigma$. In that case:
\begin{itemize}
	\item The vector space $\PhiA(H) \cat{O}_\Sigma^\cat{A}$ is identified with the $\cat{A}$-skein module $\Sk_\cat{A} (H)$ of $H$, a quotient of the famous Kauffman bracket skein module of $H$ evaluated at that root of unity.\footnote{There is a minor sign issue which is fixed by picking the non-standard ribbon structure described in~\cite{Tingley}.}
	\item The category $\int_\Sigma \cA$ is identified with the category of modules over $\End_{\int_\Sigma \cA}(\cO_\Sigma)$, which is in turn identified with the skein algebra $\SkAlg_\cat{A}(\Sigma)$ of $\Sigma$ by \cite{cooke}. 
	\item Since $\PhiA(H)$ is an equivalence in that case
	(a result that will be proven in subsequent work; we do not logically depend on it in this article and just mention it for illustration),
	 this recovers the fact proven in~\cite{roberts,masbaumroberts} that there is an algebra isomorphism
		\begin{align}
\SkAlg_\cat{A}(\Sigma) \cong \End(\Sk_\cat{A}(H))\ .
	\end{align}
	\item 		In other words, $\PhiA(H)$ factors as
	\begin{align}
\int_\Sigma \cA \simeq \End(\Sk_\cat{A}(H))\catf{-mod} \simeq \vect
		\end{align}
		where the second equivalence is the standard Morita triviality of matrix algebras. Hence, the whole functor $\PhiA(H)$  can  be identified with the skein module of $H$, this time seen as a $\k$-$\SkAlg_\cat{A}(\Sigma)$ bimodule.
	\item The statement that $\cA$ is connected in that case translates into the fact that for any other handlebody $H'$ with $\partial H'=\Sigma$, the skein modules of $H$ and $H'$ are isomorphic not just as vector spaces but as modules over $\SkAlg_\cat{A}(\Sigma)$. 
	\item The mapping class group action on $\int_{\Sigma} \cA$ is induced by the natural action of that group on the skein algebra of $\Sigma$. The fact that this categorical action is trivial means that the action on the skein algebra is by inner automorphisms. Of course, since the category at hand is equivalent to $\vect$ in that case, the triviality of the action
		is automatic. We understand this as being essentially a reformulation of the Skolem--Noether theorem that says that every automorphism of a matrix algebras over a field $\k$ is inner.
\end{itemize}
The  argument 
that in the general case of connected cyclic framed $E_2$-algebras 
produces the projective mapping class group representations 
reduces now in this special case to the argument in~\cite{roberts}.

	 \subsection{Relation to higher-dimensional fully extended topological field theories}
	In some special cases, the structures and conditions featured in this paper have a natural interpretation in the framework of fully extended topological field theories. Let us emphasize that this interpretation relies on various statements 
	 which are, as of yet, still conjectural. Our main results thus
	establish in an independent way (in particular without relying on the cobordism hypothesis) parts of the field theories that are expected to be associated with $\cat{A}$. 

Let us explain in more detail the necessary background:
	By~\cite{Johnson-Freyd2017,Haugseng2017}
	$E_2$-algebras in a symmetric monoidal $(2,1)$-category $\cat{S}$ form naturally the objects of a certain `Morita' 4-category in which they are automatically 2-dualizable~\cite{gwilliamscheimbauer}. 
	 The Baez-Dolan-Lurie
	  cobordism hypothesis~\cite{baezdolan,lurietft}
	   then implies that any $E_2$-algebra $\cat{A}$ induces a framed 
	 two-dimensional topological field theory
	  valued in that Morita 4-category. If $\cat{A}$ is a framed $E_2$-algebra, this  produces an oriented theory and its value 
	  on surfaces agrees with factorization homology with coefficients in $\cat{A}$. If one assumes that $\cat{A}$ is in fact 3-dualizable and equipped with the structure of a homotopy $\text{SO}(3)$-fixed point compatible with its framed structure, this extends to an oriented  three-dimensional topological field theory. In particular, the evaluation on a handlebody $H$ with boundary $\Sigma$, seen as a three-dimensional bordism from $\Sigma$ to a disjoint union of $n$ disks yields a map
	\begin{align}
	   \int_\Sigma \cat{A} \longrightarrow \cat{A}^{\otimes n}\ .\label{eqnwouldbephimap}
	   \end{align}
	   We expect that~\eqref{eqnwouldbephimap} agrees with the map $\PhiA(H)$ in that case.
	   There is then a canonical choice of a certain boundary conditions leading to the construction of a relative topological field theory in the sense of~\cite{Freed2012,Johnson-Freyd2017}. Such a relative theory attaches to a 3-manifold with (closed) boundary $M$ an endomorphism of the unit $I$ of $\cat{S}$ (e.g. a vector space in the case $\cat{S}=\Rex$). 	    We expect that the data for this relative theory to itself be oriented  includes a cyclic structure on $\cA$, and that the value of such a relative theory on handlebodies should then recover the ansular functor associated with $\cat{A}$.

	   If $\cat{A}$ enjoys the much stronger property of being invertible in the Morita 4-category of $E_2$-algebras, then it induces an invertible four-dimensional topological field theory. Invertibility implies that the value of that theory on a handlebody $H$ depends on $H$ only up to four-dimensional bordisms. The relative theory should then produce an anomalous, partially defined, oriented three-dimensional topological field theory. Such a structure induces, by restriction to dimension two, a modular functor, whose value on a surface is given by evaluating the relative theory on any handlebody $H$ for that surface. The anomaly in that case reflects the fact that the theory still depends, albeit very weakly, on the choice of $H$, which recovers the well-known topological description of the central extensions of mapping class groups in that particular setting~\cite{masbaumroberts}.
	  In summary, it is widely expected from this perspective that cyclic, invertible  framed $E_2$-algebras should lead to examples of modular functors. Our  results thus provide
	   a direct proof that this is indeed the case, under the generally weaker condition of cofactorizability. 

	 We note that Lyubashenko's modular functor is expected to fit into this general picture as follows: If $\cat{A}$ is a ribbon category with enough projective objects, then the main result of~\cite{bjsdualizability} states that $\cat{A}$ is 3-dualizable in the Morita category of braided tensor categories.	 On the one hand, it is expected that the ribbon structure on $\cat{A}$ canonically extends to a homotopy $\text{SO(3)}$-fixed point structure. On the other hand, $\cat{A}$ turns out to be canonically cyclic in that case~\cite{cyclic}. If $\cat{A}$ is finite and non-degenerate (i.e.\ modular), it is shown in~\cite{bjss} that it is in fact invertible, hence in particular cofactorizable
	  (in the finite ribbon case, cofactorizability is sufficient for invertibility, although we do not expect this to be the case in general).
	 Finally, if $\cat{A}$ is additionally  semisimple, the corresponding anomalous field theory is expected~\cite{Freeda,Freed2012a,Walker} to extend to dimension three and to recover the
	 Reshetikhin-Turaev theory and its underlying modular functor. We refer to~\cite{haiounbgt} for an overview of recent results and conjectures on that perspective.

		\vspace*{0.2cm}\textsc{Acknowledgments.} We are grateful to David Jordan, 
		Gwénaël Massuyeau,
		Lukas Müller, Pavel Safronov, 
	Christoph Schweigert, Noah Snyder,
	Nathalie Wahl,
	  Simon Wood
	  and Yang Yang
	for  helpful discussions related to this project.
		LW gratefully acknowledges support by 
	the Danish National Research Foundation through the Copenhagen Centre for Geometry
	and Topology (DNRF151)
	at K\o benhavns Universitet during the period in which this project was initiated. 
	Moreover, LW gratefully acknowledges support
	by the ANR project CPJ n°ANR-22-CPJ1-0001-01 at the Institut de Mathématiques de Bourgogne (IMB).
	The IMB receives support from the EIPHI Graduate School (ANR-17-EURE-0002).

	\spaceplease 
	\section{Preliminaries}
		\subsection{Factorization homology of framed $E_2$-algebras\label{secprefh}}
	In this subsection,	 we recall the framework of \emph{factorization homology for surfaces}. References include \cite[Chapter~5.5]{higheralgebra} and 
			\cite{andrade,AF} for  factorization homology in general and \cite{bzbj,bzbj2} for category-valued two-dimensional case.
	It allows us to integrate a framed little 2-disk algebra in a suitable symmetric monoidal $(2,1)$-category over a surface. 
	
In order to supply a few more details, let us fix some conventions and notation for the symmetric monoidal $(2,1)$-category $\smc$ that factorization homology and all operadic algebras will take their values in (unless otherwise stated):	
\begin{quote}
	For the rest of the article, $\smc$ will be a bicomplete  symmetric monoidal $(2,1)$-category with monoidal product $\otimes$ such that for $M\in\smc$ the functor $M\otimes -$ commutes with  colimits.
\end{quote}
An excellent reference for symmetric monoidal $(2,1)$-categories is~\cite[Chapter~2]{schommerpries}. However,
 having in mind the case $\smc=\Rex$, the symmetric monoidal $(2,1)$-category of finitely cocomplete linear categories, cocontinuous (`right exact') functors and natural isomorphisms, that we will define in Section~\ref{exseccyclicframede2}, is fully sufficient.

	We denote by $E_2$ the \emph{little disk operad}, a topological operad whose space of $n$-ary operations is given by the space of  embeddings of $n$ disks into one disk that are obtained as a combination of rescalings and translations. If additionally one allows rotations, one obtains the \emph{framed little disk operad} $\framed$. We refer to \cite[Section~5]{FresseI} for a textbook reference on the $E_2$-operad and additionally to \cite{salvatorewahl} for the framed case.
	Both $E_2$ and $\framed$ are aspherical, i.e.\ their homotopy groups $\pi_\ell$ for $\ell\ge 2$ are trivial. Hence, they can be seen as groupoid-valued operads. In fact, in arity $n$, the spaces $E_2(n)$ and $\framed (n)$ are equivalent to the classifying spaces of the pure braid group on $n$ strands and the framed pure braid group on $n$ strands, respectively.
	Since $\smc$, by virtue of being a $(2,1)$-category,
	 is enriched over $\Grpd$, we can consider $E_2$-algebras and $\framed$-algebras in $\smc$. An algebra in $\smc$ over a $\Grpd$-valued operad $\cat{O}$  is an object $\cat{A}$ in $\smc$ together with maps $\cat{O}(n)\to\smc(\cat{A}^{\otimes n},\cat{A})$ that respect equivariance under the symmetric group actions and operadic composition up to coherent isomorphism;
	 we will explain how to formally define cyclic and modular operadic algebras
	  in Section~\ref{cyclicandmodularalgebras}.

	Factorization homology attaches to a given $\framed$-algebra $\cat{A}$ in $\smc$ and 
	a surface $\Sigma$ 
	 an object 	$\int_\Sigma \cat{A} \in \smc$.
	 This construction is functorial with respect to oriented embeddings.
	For us, `surface' means throughout compact oriented two-dimensional manifold with parametrized, possibly empty boundary.
	Here a parametrization of the boundary is the datum of an orientation-preserving diffeomorphism \begin{align} \bigsqcup_{\substack{\text{set of boundary} \\ \text{components of $\Sigma$}}} \mathbb{S}^1\ra{\cong} \partial \Sigma \ , 
		\end{align}
	for an orientation on $\mathbb{S}^1$ fixed throughout.  
	 A priori, factorization homology
	with coefficients in a framed $E_2$-algebra can only be evaluated on an oriented two-dimensional manifold without boundary. For the extension to the case with boundary, it is understood throughout that we extend
	 the framed $E_2$-algebra in a trivial way to a framed Swiss-Cheese algebra, see \cite[Remark~2.2]{bzbj} for this convention.
	
	Roughly speaking, factorization homology is defined on disks and their embeddings into each other 
	using the $\framed$-structure of $\cat{A}$. Then one performs a homotopy left Kan extension along the inclusion of the symmetric monoidal category of disks into the symmetric monoidal category of all surfaces. 
	This entails that $\int_\Sigma \cat{A}$ can be described as a homotopy colimit (or simply put, the `correct' $(2,1)$-categorical colimit)
	\begin{align}
	\int_\Sigma \cat{A} = \hocolimsub{\substack{\varphi :  ( \mathbb{D}^2)^{\sqcup n} \to \Sigma   \\ n\ge 0}} \cat{A}^{\otimes n}
	\label{fheqncolim}
	\end{align}
	 over all oriented embeddings of $n$ disks into $\Sigma$
	 with $n$ running over all non-negative integers.
	 As usual, $\cat{A}^{\otimes \, 0} = I$.

	The object $\int_\Sigma \cat{A} \in \smc$ is characterized by the fact that$\int_{\mathbb{D}^2} \cat{A}\simeq \cat{A}$ and by a locality property called \emph{excision}:
	Suppose that a surface $\Sigma$ is obtained by gluing surfaces $\Sigma_0$ and $\Sigma_1$ together along a boundary component,
	 then $\int_{\Sigma_0} \cat{A}$ and $\int_{\Sigma_1} \cat{A}$ is a left and right module, respectively,
	 over the algebra $\int_{\mathbb{S}^1 \times \opint } \cat{A}$, and the embedding $\Sigma_0 \sqcup \Sigma_1 \to \Sigma$
	 obtained by choosing a collar of the gluing boundary component
	  induces an equivalence
	 $\int_{\Sigma_0}\cat{A} \otimes_{\int_{\mathbb{S}^1 \times \opint }\cat{A}} \int_{\Sigma_1 }\cat{A}\ra{\simeq}
	 \int_\Sigma \cat{A}$ from the relative monoidal product over $\int_{\mathbb{S}^1 \times \opint } \cat{A}$
	 to $\int_\Sigma \cat{A}$; we refer to \cite[Section~3.3]{AF} for details.

For any  surface $\Sigma$, the 
diffeomorphisms $\Sigma\to\Sigma$ preserving the orientation and boundary parametrization
 form a topological group $\catf{Diff}(\Sigma)$. 
The group of path components $\Map(\Sigma):=\pi_0(\catf{Diff}(\Sigma))$ is the mapping class group of the  surface $\Sigma$. If we mention 
in the sequel
diffeomorphisms or mapping classes, it will always be implicit that we are only interested in those that preserve
the orientation and boundary parametrization.
By construction $\int_\Sigma \cat{A}$ comes with a homotopy coherent $\catf{Diff}(\Sigma)$-action. 
	
	For a surface $\Sigma$,
	the embedding $\emptyset \to \Sigma$ yields a map $\cat{O}_\Sigma^\cat{A} : I \to \int_\Sigma\cat{A}$ from the unit $I$ of $\smc$ to $\int_\Sigma\cat{A}$. 
	This uses the  functoriality of factorization homology with respect to embeddings and the canonical equivalence $\int_\emptyset \cat{A}\simeq I$.
	Hence, $\int_\Sigma \cat{A}$ comes with a generalized object
	 that we refer to as
	the \emph{distinguished object} of $\int_\Sigma \cat{A}$. In \cite[Section~5.1]{bzbj} it is also referred to as \emph{quantum structure sheaf}.
	Since the diffeomorphism group of $\Sigma$ preserves tautologically the embedding $\emptyset \to \Sigma$, the
	distinguished object $\cat{O}_\Sigma^\cat{A}$ canonically comes with the structure of a homotopy $\catf{Diff}(\Sigma)$-fixed point.

\subsection{The modular extension of cyclic framed little disks algebras\label{cyclicandmodularalgebras}}
	The operad $\framed$ is a \emph{cyclic} operad 
	in the sense of Getzler and Kapranov~\cite{gk}, i.e.\
	the spaces of operations come with a prescription to cyclically permute the inputs with the output. 
	Below we will make ample use of this cyclic structure.
	In order to handle cyclic operads, but also modular operads \cite{gkmod} efficiently, we will use the description given in \cite{costello} based on certain categories of graphs:
	First recall that a graph consists of a set of half edges and a set of vertices (for us always finite) plus a map from half edges to vertices telling us to which vertex
	a half edge is attached and an involution on the set of half edges telling us how the half edges are glued together. 
		A \emph{corolla} is a contractible
		graph with one vertex and a finite number of legs attached to it.
		 We denote by $\Legs(T)$ the set of legs of a corolla $T$. Unless otherwise stated, $\Legs(T)$ can be empty.
		 	Now we denote by $\Graphs$
	the category whose objects are finite disjoint unions of {corollas}. 
	A morphism $T\to T'$ between two disjoint unions of corollas is an equivalence class of graphs $\Gamma$ 
	 together with \begin{itemize}
	 	\item an identification $\varphi_1$ of $T$ with \begin{align}\nu(\Gamma):= \text{the disjoint union of corollas obtained from $\Gamma$ by cutting $\Gamma$ at all internal edges}
	\end{align} 
	(an identification of graphs is here a bijection between the sets of half edges and vertices compatible in the obvious way with the gluing information),
	\item  and an identification of $T'$ 
	with \begin{align}\pi_0(\Gamma):=\text{the disjoint union of corollas obtained from $\Gamma$ by contracting all internal edges.}\end{align} 
	\end{itemize}
We 
illustrate
the definition of $\nu(\Gamma)$ and $\pi_0(\Gamma)$
in Figure~\ref{gmfig}. 
Such triples $(\Gamma,\varphi_1,\varphi_2)$ 
and $(\Gamma',\varphi_1',\varphi_2')$ are defined to be equivalent if
there is an identification $\Gamma\cong \Gamma'$ of graphs compatible in the obvious way with $\varphi_1,\varphi_2,\varphi_1'$ and $\varphi_2'$.

		\begin{figure}[h]
	\centering 
		\begingroup%
		\makeatletter%
		\providecommand\color[2][]{%
			\errmessage{(Inkscape) Color is used for the text in Inkscape, but the package 'color.sty' is not loaded}%
			\renewcommand\color[2][]{}%
		}%
		\providecommand\transparent[1]{%
			\errmessage{(Inkscape) Transparency is used (non-zero) for the text in Inkscape, but the package 'transparent.sty' is not loaded}%
			\renewcommand\transparent[1]{}%
		}%
		\providecommand\rotatebox[2]{#2}%
		\newcommand*\fsize{\dimexpr\f@size pt\relax}%
		\newcommand*\lineheight[1]{\fontsize{\fsize}{#1\fsize}\selectfont}%
		\ifx\svgwidth\undefined%
		\setlength{\unitlength}{682.32338327bp}%
		\ifx\svgscale\undefined%
		\relax%
		\else%
		\setlength{\unitlength}{\unitlength * \real{\svgscale}}%
		\fi%
		\else%
		\setlength{\unitlength}{\svgwidth}%
		\fi%
		\global\let\svgwidth\undefined%
		\global\let\svgscale\undefined%
		\makeatother%
	 \resizebox{1.38\linewidth}{!}{	\begin{picture}(1,0.30324158)%
			\lineheight{1}%
			\setlength\tabcolsep{0pt}%
			\put(0,0){\includegraphics[width=\unitlength,page=1]{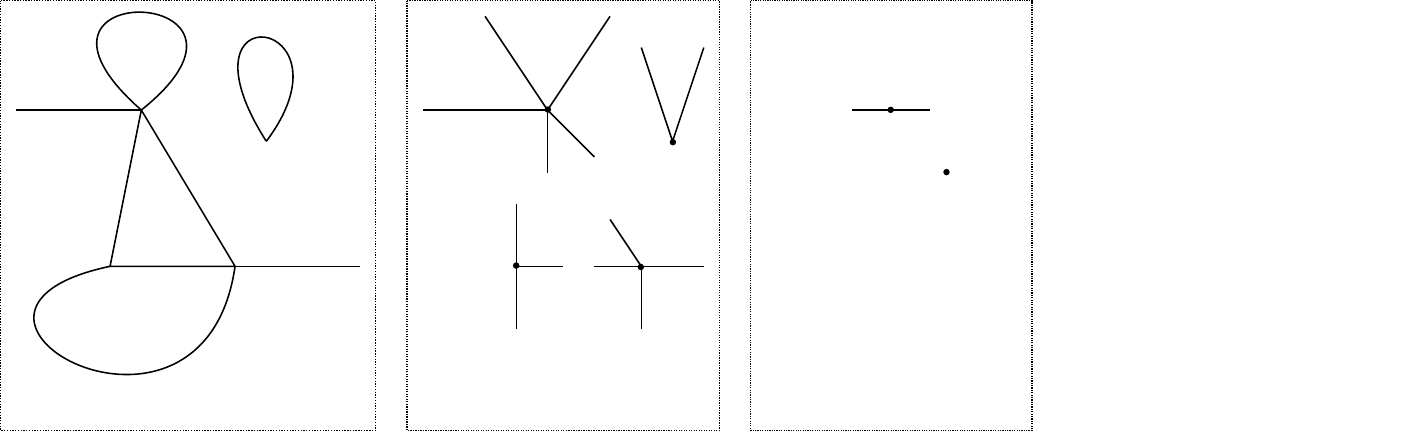}}%
			\put(0.120913,0.0156182){\color[rgb]{0,0,0}\makebox(0,0)[lt]{\lineheight{1.25}\smash{\begin{tabular}[t]{l}$\Gamma$\end{tabular}}}}%
			\put(0.384167,0.01710117){\color[rgb]{0,0,0}\makebox(0,0)[lt]{\lineheight{1.25}\smash{\begin{tabular}[t]{l}$\nu(\Gamma)$\end{tabular}}}}%
			\put(0.62007372,0.02124648){\color[rgb]{0,0,0}\makebox(0,0)[lt]{\lineheight{1.25}\smash{\begin{tabular}[t]{l}$\pi_0(\Gamma)$\end{tabular}}}}%
			\put(0,0){\includegraphics[width=\unitlength,page=2]{graphX.pdf}}%
		\end{picture}%
	}
		\endgroup%
		\caption{On the definition of $\nu(\Gamma)$ and $\pi_0(\Gamma)$.}
		\label{gmfig}
	\end{figure}

	In order to form the composition $\Gamma_2 \circ \Gamma_1$, one replaces the vertices of $\Gamma_2$ with the graph $\Gamma_1$, we refer to \cite[Section~2.1]{cyclic} for more details and explaining
	pictures. Disjoint union endows $\Graphs$ with a symmetric monoidal structure. We denote by $\Forests\subset \Graphs$ the symmetric monoidal subcategory of $\Graphs$ whose objects are finite disjoint unions of corollas (hence, it has the same objects as $\Graphs$)
	and whose morphisms are forests, i.e.\ disjoint
	unions of \emph{trees} (a tree is a contractible graph).

A \emph{cyclic operad} in a symmetric monoidal $(2,1)$-category $\smc$ (the requirements on $\smc$
 encountered in the last section on factorization homology are not needed here)
can now be defined as a symmetric monoidal functor $\cat{O}:\Forests\to\smc$, where the 1-category $\Forests$ is seen as a symmetric monoidal $(2,1)$-category without non-identity 2-morphisms and where `symmetric monoidal functor' is understood up to coherent isomorphism \cite[Chapter~2]{schommerpries}; a summary of what that means is given in~\cite[Section~2.1]{cyclic}.
A \emph{modular operad} in $\smc$ is a symmetric monoidal functor $\cat{O}:\Graphs\to\smc$.
(Ordinary non-cyclic operads can be described as symmetric monoidal functors out of $\RForests$, the version of $\Forests$ where the graphs come equipped with a distinguished leg, the so-called \emph{root}.)
	We have to highlight a minor technical point:
	Cyclic and modular operads, when defined this way using graph categories, do a priori not have an operadic identity, i.e.\ a distinguished unary operation that behaves neutrally with respect to composition --- of course up to coherent isomorphism. Operadic identities can be included via \cite[Definition~2.5]{cyclic}.
	In this article, the term `operad' (and its cyclic and modular variants) will by default always include operadic identities,
	and all maps between operads preserve these operadic identities up to coherent isomorphism. 

An important example for us is the \emph{modular handlebody operad} $\Hbdy:\Graphs\to\Cat$. To a corolla $T$, it assigns the groupoid $\Hbdy(T)$ whose objects are compact
 oriented
three-dimensional	handlebodies (hereafter just referred to as handlebodies) that are connected  and come with $|\Legs(T)|$ many 
disks embedded in the boundary. The embedded disks are parametrized, i.e.\ they come with an orientation-preserving diffeomorphism of the boundary disks of $H$ to $\left(   \mathbb{D}^2\right)^{\sqcup \Legs(T)}$ for an orientation fixed throughout on $\mathbb{D}^2$.
	Morphisms are 
	isotopy classes of orientation-preserving diffeomorphisms of  handlebodies respecting the disk parametrizations (we will just call these `mapping classes'). 
	If $T$ has $n$ legs, then in short
	\begin{align}
	\Hbdy(T) = \left\{ \begin{array}{c} \text{groupoid of connected handlebodies}  \\ \text{with $n$ parametrized disks embedded in the boundary}  \\ \text{and mapping classes as morphisms}
	\end{array}\right\} \ . 
	\end{align}
	The assignment is extended monoidally, i.e.\ $\Hbdy(T\sqcup T')=\Hbdy(T)\times \Hbdy(T')$ for corollas $T$ and $T'$. We see a pair $(H,H')\in \Hbdy(T\sqcup T')=\Hbdy(T)\times \Hbdy(T')$ as the handlebody $H\sqcup H'$. In other words, non-connected handlebodies are allowed in the operad $\Hbdy$, but they are associated to an object in $\Graphs$ that is not a corolla.
	Operadic composition is by gluing. The cyclic action on the parametrizations of embedded disks
	 gives us the cyclic structure; we refer to \cite[Section~4.3]{giansiracusa} for details.
	 
	If $T$ has $n$ legs, we might be tempted to denote $	\Hbdy(T)$ as $	\Hbdy(n)$ with $n$ as the total arity (or, alternatively, $\Hbdy(n-1)$ since we have $n-1$ inputs and one output),
	 but seeing the arity really as a (disjoint union of) corollas
	has the advantage that no ordering of the legs is implied through the numbering. 

	If we restrict to genus zero handlebodies,
	we obtain a cyclic operad that we can identify with the framed $E_2$-operad $\framed$. 
	In other words, $\framed$ is the cyclic operad of genus zero handlebodies or, equivalently, genus zero surfaces.
	
	By replacing in the definition of $\Hbdy$ handlebodies with  
	 surfaces (again, let us recall that for us these are always compact and oriented), we obtain the \emph{modular operad $\Surf$ of 
	 	 surfaces}, also called the \emph{modular surface operad}. Again, if $T$ has $n$ legs, this definition amounts to
	\begin{align}
\Surf(T) = \left\{ \begin{array}{c} \text{groupoid of connected surfaces}  \\ \text{with $n$ parametrized boundary components}  \\ \text{and mapping classes as morphisms}
	\end{array}\right\} \ , 
	\end{align}
with the notion of boundary parametrization being the one from Section~\ref{secprefh}.
For $\Sigma \in \Surf(T)$,
$
\Map(\Sigma) = \pi_1(\Surf(T),\Sigma)
$
is the mapping class group of $\Sigma$. Since it consists of \emph{automorphisms} of $\Sigma$ in $\Surf(T)$, the boundary parametrizations must be preserved. This version of the mapping class group is often referred to as the pure mapping class group.
Nonetheless, `non-pure' mapping classes of $\Sigma$ are still contained in $\Surf(T)$, but as \emph{isomorphisms} instead of \emph{automorphisms}.

Given any cyclic or modular operad $\cat{O}$ in $\Cat$,
one can consider cyclic or modular algebras in any symmetric monoidal $(2,1)$-category $\smc$.
We will give here only a rather dense summary and refer for the details to \cite{gk} for the 1-categorical case and moreover to \cite[Section~2]{cyclic} for the bicategorical generalization that we are interested in:
Suppose that $X$ is a self-dual object in the symmetric monoidal $(2,1)$-category $\smc$, i.e.\ $X$ comes with a map $\kappa :X\otimes X \to I$, called \emph{evaluation}, and a map $\Delta : I \to X \otimes X$, called \emph{coevaluation}, such that the usual zigzag identities are fulfilled up to isomorphism (this exhibits $X$ as its own dual in the homotopy category of $\smc$). A map $\kappa :X\otimes X \to I$ that is part of a self-duality of $X$ and that is symmetric up to coherent isomorphism is referred to as \emph{non-degenerate symmetric pairing on $X$.}
For any corolla $T $, we denote by $X^{\otimes \Legs(T )}$ the unordered monoidal product of $X$ running over the set $\Legs(T )$ of legs of $T $ and set $\End_\kappa^X(T ):=\smc(I,X^{\otimes\Legs(T )})$ for a  corolla $T$
(this is dual to the conventions in \cite[Section~2]{cyclic}, but the difference is insignificant). The definition is extended monoidally to the non-connected case, i.e.\
$\End_\kappa^X(T )=\prod_{c\in\mathsf{C}(T)}\smc(I,\cat{A}^{\otimes \Legs(c)})$ for a possibly non-connected object $T$ in $\Graphs$ with set $\mathsf{C}(T)$ of path components. This assignment extends to a symmetric monoidal functor $\Graphs\to\Cat$, i.e.\ to a modular operad, the \emph{endomorphism operad of $(X,\kappa)$}. Its restriction to $\Forests$ is called the \emph{cyclic endomorphism operad of $(X,\kappa)$}. 

For a given $\Cat$-valued cyclic operad $\cat{O}$, the structure of a cyclic $\cat{O}$-algebra on $X \in \smc$ is defined as
the choice of a non-degenerate symmetric pairing
$\kappa :X\otimes X \to I$
 on $X$ and a morphism $\cat{O}\to\End_\kappa^X$ of cyclic operads. Modular algebras are defined analogously.

For the cyclic framed $E_2$-operad $\framed$, which can be regarded as being groupoid-valued,
one can now consider cyclic $\framed$-algebras in any symmetric monoidal $(2,1)$-category $\smc$.
Similarly, one can consider modular $\Hbdy$-algebras in $\smc$. 
A modular $\Hbdy$-algebra is called \emph{ansular functor} in \cite{mwansular}.

The results in \cite{mwansular}
building on~\cite{cyclic,mwdiff}	
give an explicit description of cyclic framed $E_2$-algebras and ansular functors in $(2,1)$-categories.
In order to state these results, we need the following notion:

\begin{definition}[$\text{\cite[Definition~5.4]{cyclic}}$]\label{defsda}
A \emph{self-dual balanced braided algebra} in  $\smc$ is an object $\cat{A} \in \smc$
together with the following structure: \begin{itemize}
	\item $\cat{A}$ is a \emph{balanced braided algebra} in $\smc$ --- in more detail: 
	\begin{itemize}
		\item 
		$\cat{A}$ has a multiplication $\mu : \cat{A} \otimes \cat{A} \to \cat{A}$ that is associative and unital up to coherent isomorphism
		(we denote the unit  1-morphism  by $u:I\to \cat{A}$);
		\item $\cat{A}$ comes with an isomorphism $c: \mu \to \mu^\op = \mu \circ \tau$ (where $\tau$ is the symmetric braiding of $\smc$) called \emph{braiding} which satisfies the hexagon relations;
		\item $\cat{A}$ comes with an isomorphism $\theta : \id_\cat{A} \to \id_\cat{A}$ referred to as \emph{balancing} satisfying
		\begin{align}
			\theta \circ \mu  &= c^2 \circ \mu(\theta  \otimes \theta) \ , \\
			\theta \circ u &= \id_u \ . 
		\end{align}
	\end{itemize}
	\item $\cat{A}$ comes with a non-degenerate pairing
	$\kappa : \cat{A} \otimes \cat{A} \to I$
	(a 1-morphism
	exhibiting $\cat{A}$ as its own dual in the homotopy category of $\smc$)
	together with an isomorphism $\gamma : \kappa(u\otimes \mu)\to \kappa$ subject to the following conditions:  \begin{itemize}\item The isomorphism
		$\kappa(u\otimes \mu(u,-))\longrightarrow \kappa(u,-)$ coming from  
		$\gamma$ coincides with the isomorphism coming from the unit constraint $\mu(u,-)\cong \id_\cat{A}$.
		\item The equality $\kappa(\theta \otimes \id_\cat{A})=\kappa(\id_\cat{A} \otimes \theta)$ holds.
	\end{itemize}
\end{itemize}
\end{definition}

We give examples of
self-dual balanced braided algebras in~Subsection~\ref{exseccyclicframede2}.
Framed $E_2$-algebras in a 
 symmetric monoidal $(2,1)$-category 
 are equivalent to balanced braided  algebras~\cite{WahlThesis,salvatorewahl}. The following is an extension of this result characterizing \emph{cyclic} algebras over that operad:

\spaceplease
\begin{theorem}[\cite{cyclic,mwansular}]\label{thmmodext}
	The following structures
	are equivalent:
	\begin{pnum}
		
		\item A cyclic framed $E_2$-algebra in $\smc$.\label{thmmodexti}
		
		\item A self-dual balanced braided algebra in $\smc$.\label{thmmodextii}
		
		\item A modular $\Hbdy$-algebra in $\smc$, i.e.\ an $\smc$-valued ansular functor.\label{thmmodextiii}
		
		\end{pnum}
The algebraic structures described under each of these points form naturally a 2-groupoid, and the equivalence is an equivalence of 2-groupoids.
	\end{theorem}

The equivalence between~\ref{thmmodexti} and~\ref{thmmodextii}
is the main result of \cite{cyclic}. 
 The equivalence between~\ref{thmmodexti} and~\ref{thmmodextiii} is established
  in \cite{mwansular} and is, roughly, based on the following:
   Clearly, any modular $\Hbdy$-algebra can be restricted to a cyclic $\framed$-algebra. The inverse is given by the \emph{modular extension} for cyclic algebras~\cite[Section~4]{mwansular} and uses results of Giansiracusa \cite{giansiracusa} on the derived modular envelope of $\framed$
   (the notion of the modular envelope is due to Costello~\cite{costello}) and the strengthening of Giansiracusa's result in \cite{mwdiff,mwansular}.
 We denote the modular extension of a framed $E_2$-algebra $\cat{A}$, i.e.\ the unique extension
 to a $\Hbdy$-algebra
 (an ansular functor) by $\widehat{\cat{A}}$. The evaluation of the modular extension on an arbitrary handlebody is explicitly described
in \cite[Corollary~6.3]{mwansular}
via a `coend formula' that will appear in Corollary~\ref{subsecconfblock} below.

\begin{remark}The 
	assumptions on $\smc$
	from Section~\ref{secprefh}
	are not needed in Definition~\ref{defsda} and Theorem~\ref{thmmodext}. Both are true for any symmetric monoidal $(2,1)$-category.
	\end{remark}

\begin{remark}\label{remSURF}
	Even though we will, for the construction of modular functors in the context of this article, always be interested in mapping classes instead of diffeomorphisms, let us, for some applications in the proof of Theorem~\ref{thmdefsurfa}, mention that there is also the modular operad $\SURF$ of surfaces and \emph{diffeomorphisms} (as always, it is implicit that surfaces are compact, oriented and with boundary parametrization,
and that the diffeomorphisms preserve this).
Note that
$\SURF$, as opposed to $\Surf$, is not groupoid-valued, but a priori a topological modular operad, or equivalently a modular operad with values in $\infty$-groupoids, i.e.\
a symmetric monoidal functor
from $\Graphs$ to spaces or $\infty$-groupoids up to coherent homotopy. The framework for this $\infty$-categorical approach to operads is laid out in \cite[Chapter~2]{higheralgebra}, see also \cite{hry}. The details of the $\infty$-categorical approach will not concern us here because we are, at the end of the day, interested in the
 1-categorical truncation $\Surf=\Pi \SURF$, where $\Pi$ is the fundamental groupoid.
\end{remark}

\subsection{The boundary map\label{boundarymapsec}}
For $H\in \Hbdy(T)$ with $T\in\Graphs$, we can form the boundary surface $\partial H$ of $H$
which is defined as the boundary of $H$ minus the interior of the parametrized disks embedded in its boundary.
Then we obtain a functor $\partial : \Hbdy(T)\to\Surf(T)$, and in fact, a morphism of modular operads
\begin{align}
\partial : \Hbdy \to\Surf \ ,  \label{eqnpartialmapop}
\end{align} 
the \emph{boundary map}, see also~\cite[Section~7.3]{cyclic}.

In order to investigate~\eqref{eqnpartialmapop}, let us recall the following notion:
We call a functor $F:\Gamma \to \Omega$ between groupoids  a
\emph{fibration} if any 
lifting problem 
in simplicial sets
\begin{equation}\label{eqnaliftingproblem}
\begin{tikzcd}
0 \ar[rr] \ar[dd,"",swap] && B\Gamma\ar[dd,"BF"] \\ \\ 
\Delta^1 \ar[rr] \ar[rruu,"\exists", ,dashed ] & & B\Omega
\end{tikzcd}
\end{equation}
has a  solution (here $B$ denotes the nerve of a category). 
Functors between categories with this lifting property are also called \emph{isofibrations}.
For categories, there is the a priori stronger notion of a \emph{Grothendieck fibration}, but for groupoids, both notions agree. Therefore, we will just call this type of functor a fibration. For fibrations, the strict fiber and the homotopy fiber are equivalent.
We call a map of (modular) groupoid-valued operads a
fibration 
if it has this property arity-wise.

\begin{proposition}\label{propisofib}
	The natural map $\partial : \Hbdy \to \Surf$ 
	taking the boundary of handlebodies
	is a fibration of modular operads.
	The fiber and, equivalently, the homotopy fiber over some $\Sigma\in\Surf(T)$ is discrete up to equivalence in the sense that its automorphism groups are trivial. Its isomorphism classes can be identified with the set $\Map(\Sigma)/\Map(H)$ of left cosets of $\Map(H)$ in $\Map(\Sigma)$, where $H$ is any handlebody with boundary $\Sigma$.
\end{proposition}

Since the fibers of the fibration $\partial : \Hbdy \to \Surf$ are discrete, we should think of $\partial$ as a \emph{covering} of modular operads.

In the sequel,
we will not only use the lifting property, but also sometimes rather
specific lifts. Before giving the full proof of Proposition~\ref{propisofib}, let us record these for later reference:
For a diffeomorphism $f:\Sigma \to \Sigma'$ and
a handlebody $H$ with $\partial H=\Sigma$, 
we define the handlebody $f.H$ as the pushout 
\begin{equation}	
\begin{tikzcd}
\Sigma\times\{0\} \cong \Sigma=\partial H \ar[rr] \ar[dd,"f\times 0",swap] && H \ar[dd,"\widetilde f"] \\ \\ 
\Sigma'\times [0,1] \ar[rr, swap]  & & f.H \ ;
\end{tikzcd}
\end{equation}
similar constructions are often described as `gluing a mapping cylinder' to $H$.

\begin{lemma}\label{lemmaisofibprop}
	The diffeomorphism 
	$\widetilde f : H\to f.H$ satisfies $\partial \widetilde f =f$, and the
	mapping class $\widetilde f : H\to f.H$
	solves the lifting problem
	\begin{equation}\label{eqnliftsquare}
	\begin{tikzcd}
	0 \ar[rr,"H"] \ar[dd,"",swap] && B\Hbdy(T) \ar[dd, "B\partial"] \\ \\ 
	\Delta^1 \ar[rr,"f", swap] \ar[rruu,"\widetilde f", ,dashed ] & & B\Surf(T) \ . 
	\end{tikzcd}
	\end{equation}
\end{lemma}

\begin{proof}
	This follows immediately from the definition of $f.H$ and $\widetilde f$.
\end{proof}

\begin{remark}
	In Lemma~\ref{lemmaisofibprop} and similar situations below, we will use the same symbol for a diffeomorphism and the associated mapping class in order to avoid clumsy notation. We will always make clear through the context what is meant. 
\end{remark}

\begin{proof}[{\slshape Proof of Proposition~\ref{propisofib}}]
	Lemma~\ref{lemmaisofibprop} implies that the boundary map is a fibration.
	The 
	homotopy fiber $\partial/\Sigma$
	of $\partial$ over $\Sigma$ is
	the groupoid of pairs $(H,f)$ consisting of a handlebody $H\in\Hbdy(T)$ plus a mapping class $f:\partial H\to \Sigma$.  
	Every object in $\partial / \Sigma$, up to isomorphism, is of the form $(H_0,f)$, where $H_0$ is one fixed handlebody with $\partial H_0\cong \Sigma$ and $f:\partial  H_0 \to \Sigma$ is any mapping class (this follows from the fact that two handlebodies with isomorphic boundary are isomorphic, although not isomorphic relative boundary of course). A morphism $(H_0,f)\to (H_0,f')$ is a mapping class $g:H_0\to H_0$ with $f' \partial g=f$ which implies $\partial g = (f')^{-1}f$. Since the mapping class group of a handlebody is a subgroup of the mapping class group of its boundary, $\partial$ is injective on morphisms. As a result, the automorphism groups of the groupoid $\partial/\Sigma$ are trivial, and $\pi_0(\partial/\Sigma)$ can be identified with the 
	set $\Map(\Sigma) / \Map(H_0)$
	of left cosets of $\Map(H_0)$ in $\Map(\Sigma)$.
\end{proof}


\begin{remark}\label{remisofibprop2}
	Let $f:\Sigma \to \Sigma'$ be a diffeomorphism.
	If $f = \partial g$ for a diffeomorphism $g:H\to H'$ of handlebodies, 
	there is a unique diffeomorphism $\xi_g : f.H\to H'$ with $\partial \xi_g = \id_{\Sigma'}$ and $\xi_g \circ \widetilde f = g$. 
	Indeed, we are forced to define $\xi_g := g \circ \left(\widetilde f\right)^{-1}$.
	Now $\partial \xi_g = \id_{\Sigma'}$ follows from Lemma~\ref{lemmaisofibprop}.
\end{remark}

\subsection{A concrete description of cyclic framed $E_2$-algebras with values in linear categories\label{exseccyclicframede2}}
There is an inexhaustible supply of examples of 
cyclic framed $E_2$-algebras or, equivalently, self-dual balanced braided algebras
coming from representation theory. This section will briefly name the main sources.
The reader only interested in the abstract treatment of modular functors may skip this subsection.

	If we specialize the symmetric monoidal
$(2,1)$-category $\smc$ in Theorem~\ref{thmmodext}
to be $\Fin$, the symmetric monoidal bicategory of
\begin{itemize}
	\item finite categories, i.e.\ $\k$-linear abelian categories with finite-dimensional morphism spaces, finitely many simple objects up to isomorphism, enough projective objects, finite length for every object (such categories are exactly the ones that are linearly equivalent to finite-dimensional modules over some algebra),
	
	\item left exact functors,
	
	\item and natural isomorphisms
	
\end{itemize}
with the Deligne product as monoidal product,
then each of the structures mentioned in Theorem~\ref{thmmodext} can equivalently be described as ribbon Grothendieck-Verdier categories in the sense of Boyarchenko and Drinfeld~\cite{bd} by the results of \cite{cyclic}: 
A \emph{Grothendieck-Verdier category} in $\Lexf$ is a monoidal category $\cat{A}\in\Lexf$ in which the functor $\cat{A}(K,X\otimes -)$ is representable by $DX\in \cat{A}$ for each $X$ (that means $\cat{A}(K,X\otimes-)\cong \cat{A}(DX,-)$)
such that the functor $D:\cat{A}\to\cat{A}^\op$ sending $X$ to $DX$
is an equivalence. One calls $D$ the \emph{duality functor}. It is very important to emphasize that this is not necessarily a rigid duality (we discuss the rigid case in Example~\ref{exfiniteribboncategories}). In fact even if $\cat{A}$ is rigid, then $D$ needs not coincide with the rigid duality.
A \emph{ribbon Grothendieck-Verdier category}\label{defgvpage} in $\Lexf$ comes additionally with \begin{itemize}
	\item a braiding, i.e.\
	isomorphisms $c_{X,Y}:X\otimes Y\to Y\otimes X$ satisfying the 
	hexagon axioms, \item and a balancing $\theta_X : X \to X$, i.e.\
a natural automorphism of the identity functor with \begin{itemize}\item $\theta_{X\otimes Y}=c_{Y,X}c_{X,Y}(\theta_X \otimes\theta_Y)$ for $X,Y \in \cat{A}$,
\item and $\theta_I=\id_I$ for the monoidal unit $I$ of $\cat{A}$,
\end{itemize} 
\end{itemize}
such that additionally the ribbon condition $\theta_{DX}=D\theta_X$ is satisfied for all
all $X\in\cat{A}$. 

We should mention that $\Lexf$, due to lack of (co)completeness,
is not suitable as a target category for factorization homology.
Instead,  one can use $\Rex$,\label{defrexpage}
the symmetric monoidal $(2,1)$-category of \begin{itemize}\item finitely cocomplete linear categories, \item finitely cocontinuous functors (that one calls right exact functors even though we are not necessarily dealing with abelian categories) \item and  natural isomorphisms. \end{itemize}
The monoidal product is the Deligne-Kelly product, see~\cite[Section~4]{Franco2013} for details.
The symmetric monoidal $(2,1)$-category $\Rex$ meets the technical requirements mentioned in Section~\ref{secprefh}
that allow us the define factorization homology with values in $\Rex$.

There is a symmetric monoidal sub-$(2,1)$-category $\Rexf$ of $\Rex$ spanned by all finite linear categories, as defined above as the objects of $\Lexf$. We can see $\Rexf$ as dual version of $\Lexf$. More precisely, the functor taking
the opposite category of a finite linear category provides us with a symmetric monoidal equivalence $\Lexf \simeq \Rexf$.
Therefore, any cyclic framed $E_2$-algebra in $\Lexf$ gives rise to a cyclic framed $E_2$-algebra in $\Rexf$ and hence in $\Rex$.

\begin{example}[Finite ribbon categories]\label{exfiniteribboncategories}
	A \emph{finite tensor category} in the sense of Etingof, Gelaki, Nikshych and Ostrik
	\cite{etingofostrik,egno}
	is a finite linear category with a monoidal product that is rigid, i.e.\ admits left and right duals, and has a simple monoidal unit.
	A \emph{finite ribbon category} $\cat{A}$ is a finite tensor category
	that comes with a braiding $c$, a balancing $\theta$ such that $\theta_{X^\vee} = \theta_X^\vee$, where $X^\vee$ is the dual object of $X\in\cat{A}$ (in finite ribbon categories, left and right duality coincide). 
One may obtain a finite ribbon category for example by taking finite-dimensional modules over a finite-dimensional ribbon Hopf algebra, see \cite[Section~XIV.3]{kassel} for a textbook reference.
Thanks to rigidity, the monoidal product of a finite tensor category is exact. Therefore, a finite ribbon category provides also an example of a cyclic framed $E_2$-algebra in $\Lexf$, $\Rexf$ or $\Rex$. 
For a braided category $\cat{A}$, one defines the Müger center $Z_2(\cat{A}) \subset \cat{A}$\label{defmueger} as
the full subcategory spanned by the \emph{transparent} objects, i.e.\ those objects
that trivially double braid with every other object.
A \emph{modular category} \cite{mose88,turaevmod,turaev} is a finite ribbon category whose Müger center 
contains just the monoidal unit and finite direct sums of it.
In this case, we call the braiding \emph{non-degenerate}. 
\end{example}

\begin{example}[Module categories of vertex operator algebras]\label{exvoa}
	A rich source for ribbon Grothen\-dieck-Verdier categories are vertex operator algebras. This is investigated by Allen, Lentner, Schweigert and Wood in \cite{alsw}, and this example is essentially a short summary of some of their results. 
	Presenting a detailed definition of vertex operator algebras and the construction of their categories of modules is  beyond the scope of this article. Instead, it will suffice for us to use the
	 connection between vertex operator algebras and ribbon Grothendieck-Verdier duality 
	  presented in \cite{alsw}, namely that suitable categories of modules over a $C_2$-cofinite vertex operator algebra form a ribbon Grothendieck-Verdier category. 
Let us discuss an example~\cite[Section~3]{alsw}
that can be described 
in relatively simple terms and that will be relevant for us later:	The \emph{Feigin-Fuchs boson}
is a ribbon Grothendieck-Verdier category or rather a class of ribbon Grothendieck-Verdier categories.
While the underlying monoidal category will still be rigid, the duality functor of the Grothendieck-Verdier structure will not necessarily coincide with the rigid duality.
The categories in question are built from \emph{bosonic lattice data}, i.e.\ a quadruple $\Psi = (\mathfrak{h},\spr{-,-},\Lambda,\xi)$,
where
$\mathfrak{h}$ is a finite-dimensional real vector space with non-degenerate symmetric real-valued bilinear form $\spr{-,-}$, a lattice $\Lambda \subset \mathfrak{h}$ (i.e.\ a discrete subgroup) that is even with respect to $\spr{-,-}$, and an element $\xi \in \Lambda ^* / \Lambda$, where $\Lambda^* := \{x\in\mathfrak{h} \, | \, \spr{x,\Lambda} \subset \mathbb{Z}\}$.
We will assume that $\Lambda$ has full rank. In this case, $G:= \Lambda^* / \Lambda$ is a finite abelian group.  
From bosonic lattice data $\Psi = (\mathfrak{h},\spr{-,-},\Lambda,\xi)$, one may construct a vertex operator algebra 
 whose category of modules $\catf{VM}(\Psi)$ is a ribbon Grothendieck-Verdier category over $\mathbb{C}$. The details on the vertex operator algebra 
 are given in \cite[Section~3.3]{alsw} and need not concern us here because there is a group-cohomological description \cite[Theorem~3.12]{alsw}:
A class $\omega \in H_\text{ab}^3(G;\mathbb{C}^\times)$ in the third abelian group cohomology of $G= \Lambda^* / \Lambda$ 
(we refer to \cite{eilenbergmaclane} for a background on abelian group cohomology)
determines on the category $\vect_G$ of finite-dimensional $G$-graded vector spaces over $\mathbb{C}$ a braided monoidal structure \cite[Section~8.4]{egno}
that we denote by $\vect_G^\omega$. By \cite[Theorem~4.2.2]{Zetsche} the choice of a group element $g_0 \in G$ with $g_0=2h_0$ for some $h_0\in G$ gives us in fact a ribbon Grothendieck-Verdier structure whose dualizing object is the ground field $\mathbb{C}_{g_0}$ seen as graded vector space supported in degree $g_0$. 
We denote the resulting
ribbon Grothendieck-Verdier category by $\vect_G^{\omega,g_0}$.
Now by \cite[Theorem~3.12]{alsw} we can construct from bosonic lattice data $\Psi = (\mathfrak{h},\spr{-,-},\Lambda,\xi)$ a class $\omega(\Psi)\in H_\text{ab}^3(G;\mathbb{C}^\times)$ such that
\begin{align}
\catf{VM}(\Psi) \simeq \vect_G^{\omega(\Psi),2\xi} \quad \text{with}\quad G=\Lambda^* / \Lambda \label{eqngrpcohom}
\end{align}
as ribbon Grothendieck-Verdier categories.
The Feigin-Fuchs boson forms a rigid monoidal category, but its Grothendieck-Verdier duality
will not necessarily coincide with the rigid duality. In fact, this is the case if and only if $\xi=0$. We should mention that the results of \cite{alsw} can also produce ribbon Grothendieck-Verdier categories whose monoidal product is not exact, so that the category cannot even be rigid in that case (rigidity would imply exactness). 
	\end{example}

\section{The definition of a modular functor and the moduli space of modular functors}

The purpose of this section is to give a sufficiently
general definition of a modular functor.
Having at our disposal the groupoid-valued surface operad $\Surf$
and the framework to define its modular $(2,1)$-categorical algebras, 
one might be tempted to define modular functors with values in a symmetric monoidal $(2,1)$-category $\smc$ as 
modular $\smc$-valued $\Surf$-algebras.
As already briefly explained in the introduction, this definition would be unsatisfactory as it is known  to exclude rich classes of examples, namely systems of \emph{projective} mapping class group representations.
Those can be included by allowing for more general extensions of mapping class groups.
The `projectiveness' of the mapping class group actions of a modular functor is sometimes referred to as \emph{anomaly}.
We refer to~\cite{gilmermasbaum} for a discussion about how and why this appears in the classical constructions of modular functors.

\subsection{The definition of a modular functor\label{secexthbdy}}	
Since we  would like to systematically treat all extensions at once, we need to develop a rather general theory of extensions
for modular operads, and more specifically the surface operad. 
Fortunately, a formalism to treat extensions of groupoids is available 	in~\cite{bbf04}, 
where the classical extension theory using Dedecker-Schreier cocycles~\cite{schreier,dedecker} is generalized to groupoids.
 We can easily adapt this to groupoid-valued modular operads.

\begin{definition}\label{defextension}
	Let $\cat{O}$ be a modular groupoid-valued operad.
	An extension of $\cat{O}$ is a
	map $q : \cat{P}\to\cat{O}$ of groupoid-valued modular operads such that for any $o\in \cat{O}(T)$ and any $T\in\Graphs$ 
	the homotopy fiber $\cat{P}_o$ of $q_T:\cat{P}(T)\to \cat{O}(T)$ over $o\in\cat{O}(T)$
	is connected. A map between extensions is a map in the slice category $\catf{ModOp}(\Grpd) / \Surf$.
\end{definition}

As in this definition,
we will throughout
denote the homotopy fiber of $\cat{P}$
over $o$ 
by $\cat{P}_o$.
For an extension $q:\cat{P}\to\cat{Q}$, any $o\in \cat{O}(T)$ and any $p\in \cat{P}_o$, we obtain a short exact sequence of groups
\begin{align}
	1\to \Aut_{\cat{P}_o}(p) \to \Aut_{\cat{P}(T)}(p)\to \Aut_{\cat{O}(T)}(o)\to 1\   ; \label{eqnsesext}
\end{align} 
here we see $p$ also as object in $\cat{P}(T)$ by applying the forgetful map $\cat{P}_o \to \cat{P}(T)$, but we use the same symbol by a slight abuse of notation.
The short exact sequence~\eqref{eqnsesext}
is the only non-trivial part of the long exact sequence of homotopy groups for the fiber sequence
$|B\cat{P}_o|\to |B\cat{P}(T)|\to |B\cat{O}(T)|$
($B$ is the nerve of a category; $|-|$ is the geometric realization of a simplicial set).
We then say that \emph{at $o\in\cat{O}$ the extension is by the group $\Aut_{\cat{P}_o}(p)$}.
Definition~\ref{defextension}
can be formulated for cyclic operads as well.

For the definition of a modular functor, 
we will have to further refine the notion of an extension. 
This will mean the following:
\begin{itemize}
	\item
	We would like to consider extensions of 
	$\Surf$ that are relative to genus zero (i.e.\ come with a 
	section over the genus zero part of $\Surf$). Requirements in this direction are already implicit in the discussion of the framing anomaly in \cite[IV.3.5]{turaev}, see also the operadic discussion of this point in \cite[Lemma~7.15]{cyclic}. 
	
	\item We would like the extension term, i.e.\
	the fiber of the extension, to satisfy some very mild locality properties. Such 
	requirements are standard in the theory of modular functors, see e.g.~\cite[Section~3.2]{jfcs} for a discussion in a slightly different mathematical language. 
\end{itemize}
We will now make such requirements precise in our framework. To this end, we need
a few rather technical definitions:
For an extension $q:\cat{Q}\to \Surf$ of $\Surf$,
we denote by $q_0 : \cat{Q}_0 \to \Surf_0$ the extension of cyclic operads obtained after genus zero restriction. 

An \emph{extension $(\cat{Q},s_0)$ of $\Surf$ relative to genus zero} is an extension $q: \cat{Q}\to\Surf$ and a section of $q$ over $\Surf_0$, i.e.\ a map $s_0:\Surf_0 \to  \cat{Q}_0$ of extensions of $\Surf_0$. Here $\Surf_0$ is seen as a `trivial' extension over itself. Extensions of $\Surf$ relative to genus zero form naturally a bicategory.
Since $\Hbdy$, as $\Grpd$-valued modular operad, is the derived modular envelope of $\Surf_0$, as we know from \cite{giansiracusa} and the additions in \cite{mwansular}, the section $s_0$ can be equivalently be described 
as a 
section of $q$ over $\Hbdy$, i.e.\ a map $s:\Hbdy \to \cat{Q}$ of extensions of $\Hbdy$. Again, $\Hbdy$ is seen as a `trivial' extension over itself and $\cat{Q}$ is seen as extension of $\Hbdy$ whose homotopy fiber over $H$ is $\cat{Q}_{\partial H}$. In other words, we see $\cat{Q}$ as extension of 
$\Hbdy$
by homotopy pullback along $\partial : \Hbdy \to \Surf$. One gets this way for any choice of a handlebody $H$ a distinguished `basepoint' in the fiber over $\partial H$.
	The homotopy fiber $\cat{Q}_{\Sigma}$ of any genus zero surface comes with a pointing, namely $s_0(\Sigma)\in  \cat{Q}_{\Sigma}$. 
	This implies that 
	for any oriented embedding $\varphi : (\mathbb{D}^2)^{\sqcup J} \to \Sigma$, for any surface $\Sigma$, we have a functor
	\begin{align}\label{eqnlambdafunctor}
		\lambda_\varphi : 	\cat{Q}_{\Sigma \setminus\mathring{\operatorname{im}}\, \varphi} \ra{  \substack{ s_0(\Sigma     \setminus\mathring{\operatorname{im}}\, \varphi       )
				\times \id }}	\cat{Q}_{ (\mathbb{D}^2)^{\sqcup J}    } \times \cat{Q}_{\Sigma \setminus \mathring{\operatorname{im}}\, \varphi}  \ra{\substack{\text{operadic} \\  \text{composition}}} \cat{Q}_   \Sigma \ ,
	\end{align}
	where $\mathring{\operatorname{im}}\, \varphi$ is the interior of the image of $\varphi$.
	We say that $(\cat{Q},s)$ 
	\emph{admits insertions of vacua} if  for  
	any surface
	and any oriented embedding the functor~\eqref{eqnlambdafunctor}
	is an equivalence. 


\begin{definition}\label{deflambdaequiv} 
	We define the \emph{bicategory $\ExtSurf$ of extensions of $\Surf$ relative to genus zero that admit insertions of vacua} as the full subbicategory
	of the extensions of $\Surf$ relative to genus zero
	spanned 
	by those objects that admit insertions of vacua.
\end{definition}

This finally gives us the class of `reasonable' extensions of $\Surf$ that we would like to consider: They are extensions relative genus zero, and they satisfy a rather mild locality property in the sense that they are compatible with the insertion of disks.  

\begin{remark}
	The `usual' extensions of mapping class groups showing up in modular functors constructed from modular categories lead of course to an extension in the sense of Definition~\ref{deflambdaequiv}. Let us emphasize that there are several variants of that construction which strictly speaking are not isomorphic (see~\cite{masbaumroberts} and~\cite[Section~7.5]{Schommer-Pries2017} for a comparison). Definition~\ref{defmoduli} is made in such a way that those still all lead to equivalent modular functors in our sense. For a surface $\Sigma$, let $\Omega(\Sigma)$ be the fundamental groupoid of the space of Lagrangian subspaces in the first real cohomology group of the closure of $\Sigma$. This space is well-known to be connected with fundamental group isomorphic to $\mathbb{Z}$, and one way to construct a version of that extension is by applying the construction of Section~\ref{secsurfa} to this collection of groupoids. This extension admits insertion of vacua by construction. Every handlebody $H$ bounding $\Sigma$ defines a point in $\Omega(\Sigma)$ by taking the image of the first cohomology group of the closure of $H$, which is well-known to be a Lagrangian subspace.
\end{remark}

\begin{lemma}\label{lemma2morphinv}
	The 2-morphisms in the bicategory $\ExtSurf$ are invertible, thereby making
	$\ExtSurf$ a $(2,1)$-category. In other words,  the morphism categories $\Map_{   \ExtSurf   }(-,-)$ in $\ExtSurf$ are groupoids. 
\end{lemma}

\begin{proof}
	Let $\cat{Q}, \cat{Q}' \in \ExtSurf$ and let $\alpha_0 , \alpha_1  : \cat{Q}\to\cat{Q}'$ be 1-morphisms. Denote by $s:\Hbdy \to \cat{Q}$ and $s' : \Hbdy \to \cat{Q}'$ the respective sections. After unpacking the definitions, it follows from $\alpha_i : \cat{Q}\to\cat{Q}'$ being a 1-morphism in $\ExtSurf$ that $ \alpha_is \stackrel{\beta_i}{\cong} s'$ by a canonical isomorphism for $i=0,1$. Now if $\gamma : \alpha_0\to \alpha_1$ is a 2-morphism in $\ExtSurf$, then we have in particular $\beta_{1,H}\gamma_{s(H)}=\beta_{0,H}$ for any handlebody. This tells us that all components of $\gamma$ that lie in the image of $s: \Hbdy \to \cat{Q}$ are invertible.
	But $s$ is arity-wise essentially surjective (because the homotopy fibers of $\cat{Q}$ over $\Surf$ are connected). Therefore, it follows that \emph{all} components of $\gamma$ are invertible.
\end{proof}

We have now achieved the main goal of this section and are finally 
in a position to 
define modular functors:

\begin{definition}\label{defmodularfunctor}
	A \emph{modular functor} is a pair $(\cat{Q},\cat{B})$, where 
	$\cat{Q} \in \ExtSurf$ is an extension of $\Surf$ in the sense of Definition~\ref{deflambdaequiv}
	(i.e.\ it is implied that the extension is relative $\Hbdy$ and admits insertions of vacua),
	and $\cat{B}$ is a modular $\cat{Q}$-algebra.
\end{definition}

\subsection{The extension problem for cyclic framed $E_2$-algebras and the moduli space of modular functors}
In the next step, we will, using Definition~\ref{defmodularfunctor} as a starting point, define a \emph{moduli space of modular functors}.
To this end, we will define some auxiliary notions that will be convenient later.

Given an extension $\cat{Q}$ of $\Surf$ relative $\Hbdy$ and a cyclic framed $E_2$-algebra, it is natural to ask 
about a description of the space of modular $\cat{Q}$-algebras that extend $\cat{A}$.

\begin{definition}[Extensions of a cyclic framed $E_2$-algebra over a fixed extension of $\Surf$]\label{extofsurfdef}
	For a cyclic framed $E_2$-algebra  $\cat{A}$ in $\smc$ and $\cat{Q}\in \ExtSurf$, we define the 
	groupoid $\Ext(\cat{A};\cat{Q})$ of
	\emph{extensions of $\cat{A}$ on $\cat{Q}$} as 
	the homotopy fiber of
	\begin{align}
		\catf{ModAlg}(\cat{Q})\ra{\substack{\text{restriction along}\\  \text{the section $\Hbdy \to \cat{Q}$}}} 	\catf{ModAlg}(\Hbdy)  
		\ra{\substack{\text{genus zero}\\  \text{restriction}}}
		\catf{CycAlg}(\framed) 
	\end{align} over $\cat{A}$.
\end{definition}

\begin{remark}
	Cyclic and modular algebras with values in a symmetric monoidal bicategory form 2-groupoids \cite[Proposition~2.18]{cyclic}.
	Therefore, $\Ext(\cat{A};\cat{Q})$,
	as  a homotopy fiber of a map between 2-groupoids,
	is a priori a 2-groupoid itself.
	However, from the definition of 1-morphisms and 2-morphisms of cyclic and modular algebras \cite[Section~2.4]{cyclic}
	and the definition of the homotopy fiber, it follows that  all 2-automorphisms of $\Ext(\cat{A};\cat{Q})$
	are actually trivial. For this reason, we will  see  $\Ext(\cat{A};\cat{Q})$ as as groupoid. 
\end{remark}

\begin{definition}[All extensions of a cyclic framed $E_2$-algebra]\label{defallext}
	For a cyclic framed $E_2$-algebra  $\cat{A}$ in $\smc$, we define the bicategory $\Ext(\cat{A})$
	of \emph{all extensions of $\cat{A}$} as the bicategory of pairs
	$(\cat{Q},\cat{B})$ formed by $\cat{Q} \in \ExtSurf$ and $\cat{B} \in \Ext(\cat{A};\cat{Q})$.
	A 1-morphism
	$(\cat{Q},\cat{B})\to(\cat{Q}',\cat{B}')$ is a 1-morphism $\psi : \cat{Q}\to\cat{Q}'$ in $\ExtSurf$ together with a map $\alpha : \cat{B} \to \psi^* \cat{B}'$ in
	$\Ext(\cat{A};\cat{Q})$.
	The 2-morphisms are defined in a similar way.
\end{definition}

We can now define the moduli space of modular functors:

\begin{definition}
	We define the
	\emph{bicategory of modular functors}  $\MF $
	as the bicategory formed by pairs $(\cat{Q},\cat{B})$, where $\cat{Q}\in  \ExtSurf$ 
	and $\cat{B} \in \catf{ModAlg}(\cat{Q})$.
	A 1-morphism
	$(\cat{Q},\cat{B})\to(\cat{Q}',\cat{B}')$ is a 1-morphism $\psi : \cat{Q}\to\cat{Q}'$ in $\ExtSurf$ together with a map $\alpha : \cat{B} \to \psi^* \cat{B}'$ in
	$\Ext(\cat{A};\cat{Q})$, where $\cat{A}$ is the cyclic framed $E_2$-algebra obtained from $\cat{B}$ via genus zero restriction.
	The 2-morphisms are defined in a similar way.
\end{definition}

By taking nerves of morphism categories we may see a bicategory as a simplicial category to which we can assign its homotopy coherent nerve, see e.g.\
\cite[Section~16.3]{riehl}.
This allows us to speak about the nerve $B\cat{C}$ of a bicategory $\cat{C}$.
This version of the nerve is also called the \emph{Duskin nerve} \cite{duskin}. 
\begin{definition}\label{defmoduli}	We define the \emph{moduli space of modular functors} as the geometric realization $ \moduli:=|B\MF|$ of the Duskin nerve of $\MF$.
\end{definition}

\begin{remark}\label{remarkmodulispace}
	One may ask why one should consider the realization $|B\MF|$ of the nerve of $\MF$ as the moduli space of modular functors and not $\MF$, as plain bicategory. The reason is the following: Suppose we are given $(\cat{Q},\cat{B}),(\cat{Q}',\cat{B}')$ in $\MF$. A 1-morphism
	$(\cat{Q},\cat{B})\to(\cat{Q}',\cat{B}')$ in $\MF$
	is a map
	$\alpha : \cat{Q}\to\cat{Q}'$ of extensions of $\Surf$
	relative to $\Hbdy$ plus an equivalence 
	$\phi:\alpha^*\cat{B}'\simeq \cat{B}$
	of modular $\cat{Q}$-algebras. 
	In other words,
	$\cat{B}$  factors through $\alpha$ and produces $\cat{B}'$.
	By definition the map $\alpha$ is over $\Surf$; it just changes the fibers of the extensions.
	Spelled out in more detail, this means that for a surface $\Sigma$ the modular algebra $\cat{B}$ produces an action of an extension $\pi:G_{\cat{Q},\Sigma}\to\Map(\Sigma)$ on $\cat{B}(\Sigma)$ while
	$\cat{B}'$ produces an action of an extension $\pi':G_{\cat{Q}',\Sigma}
	\to \Map(\Sigma)$ on $\cat{B}'(\Sigma)$. 
	The map $\alpha : \cat{Q}\to\cat{Q}'$, when evaluated at $\Sigma$, gives us a group morphism $\alpha_\Sigma : G_{\cat{Q},\Sigma}\to 	G_{\cat{Q}',\Sigma}$ with $\pi' \alpha_\Sigma = \pi$, and  
	$\phi:\alpha^*\cat{B}'\simeq \cat{B}$ provides an isomorphism
	$\phi_\Sigma : 	\alpha_\Sigma^* \cat{B}'(\Sigma) \cong \cat{B}(\Sigma)$ of $G_{\cat{Q},\Sigma}$-representations.
	Then we want to regard $(\cat{Q},\cat{B})$ and $(\cat{Q}', \cat{B}')$ as the \emph{same} modular functor because $\cat{B}$ and $\cat{B'}$ both factor through $\cat{Q}'$ and agree as $\cat{Q}'$-algebras. But these pairs are only identified through a 1-isomorphism
	in $\moduli=|B\MF|$, and not $\MF$. In general, two objects in $\moduli$ are isomorphic if they are connected by a zigzag of such maps.
\end{remark}

	\section{Ansular functors and handlebody skein modules\label{modenvfhsec}}
This section is devoted to our first result, a connection between the ansular functor associated to a framed cyclic $E_2$-algebra and factorization homology. 
	The connection will be through handlebody skein modules that we define in Section~\ref{sechandlebodyskeinmodules}. For the convenience of the reader, we explain in Section~\ref{secskeinalgmod} the relation between our definition and the traditional definition of skein modules associated with semisimple ribbon categories.
	
	\subsection{The handlebody skein modules for a cyclic framed $E_2$-algebra\label{sechandlebodyskeinmodules}}
Let $\cat{A}$ be a cyclic framed $E_2$-algebra in $\smc$, $\Sigma\in\Surf(T)$ for $T\in \Graphs$ (we assume for the moment that $T$ is a corolla) and
$\varphi : \sqcup_J \mathbb{D}^2 \to \Sigma$
an oriented embedding ($J$ is a finite set).
If $H\in \Hbdy(T)$ is a handlebody with $\partial H=\Sigma$, then the boundary components of $\Sigma$ induce embedded disks in
$\partial H$. Additionally, we obtain embedded disks in $\partial H$ from $\varphi$. In total, we obtain 
a handlebody with $|J|+|\Legs(T )|$ embedded disks. We denote this handlebody by $H^\varphi$, see Figure~\ref{FPhimap}.\label{Hvarphipage}

 \begin{figure}[h]
	\centering
	\def\svgwidth{0.6\textwidth}
	\begingroup%
	\makeatletter%
	\providecommand\color[2][]{%
		\errmessage{(Inkscape) Color is used for the text in Inkscape, but the package 'color.sty' is not loaded}%
		\renewcommand\color[2][]{}%
	}%
	\providecommand\transparent[1]{%
		\errmessage{(Inkscape) Transparency is used (non-zero) for the text in Inkscape, but the package 'transparent.sty' is not loaded}%
		\renewcommand\transparent[1]{}%
	}%
	\providecommand\rotatebox[2]{#2}%
	\newcommand*\fsize{\dimexpr\f@size pt\relax}%
	\newcommand*\lineheight[1]{\fontsize{\fsize}{#1\fsize}\selectfont}%
	\ifx\svgwidth\undefined%
	\setlength{\unitlength}{526.82145438bp}%
	\ifx\svgscale\undefined%
	\relax%
	\else%
	\setlength{\unitlength}{\unitlength * \real{\svgscale}}%
	\fi%
	\else%
	\setlength{\unitlength}{\svgwidth}%
	\fi%
	\global\let\svgwidth\undefined%
	\global\let\svgscale\undefined%
	\makeatother%
	\begin{picture}(1,0.6939801)%
		\lineheight{1}%
		\setlength\tabcolsep{0pt}%
		\put(0,0){\includegraphics[width=\unitlength,page=1]{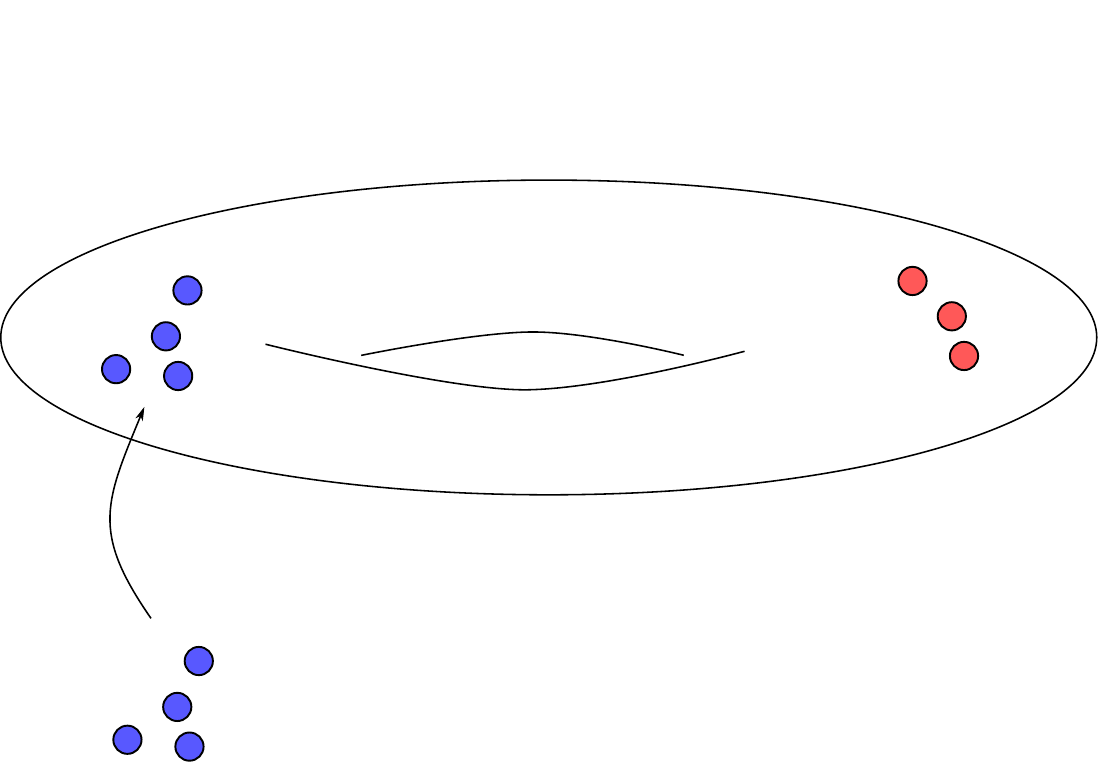}}%
		\put(0.36803527,0.29475036){\color[rgb]{0,0,0}\makebox(0,0)[lt]{\lineheight{1.25}\smash{\begin{tabular}[t]{l}$H$\end{tabular}}}}%
		\put(0.15876131,0.20007884){\color[rgb]{0,0,0}\makebox(0,0)[lt]{\lineheight{1.25}\smash{\begin{tabular}[t]{l}$\varphi : \sqcup_J \mathbb{D}^2 \to \Sigma = \partial H$\end{tabular}}}}%
		\put(0,0){\includegraphics[width=\unitlength,page=2]{Phimap.pdf}}%
		\put(0.36254829,0.6795584){\color[rgb]{0,0,0}\makebox(0,0)[lt]{\lineheight{1.25}\smash{\begin{tabular}[t]{l}boundary components of $\Sigma$ / embedded disks of $H$\end{tabular}}}}%
		\put(0.16030712,0.15542355){\color[rgb]{0,0,0}\makebox(0,0)[lt]{\lineheight{1.25}\smash{\begin{tabular}[t]{l}disks embedded via $\varphi$ become `additional' disks in the boundary of $H$\end{tabular}}}}%
	\end{picture}%
	\endgroup%
	\caption{On the definition of $H^\varphi$.}
	\label{FPhimap}
\end{figure}

The evaluation of the modular extension $\widehat{\cat{A}}$ 
(see Theorem~\ref{thmmodext} and the remarks afterwards)
 on $H^\varphi$ 
 yields a map $I\to \cat{A}^{\otimes\Legs(T )}\otimes \cat{A}^{\otimes J}$. Since $\cat{A}$ is a cyclic algebra, we can use the cyclic structure to obtain a map
\begin{align}
\widehat{\cat{A}}(H^\varphi):	\cat{A}^{\otimes J}\to\cat{A}^{\otimes\Legs(T )} \ .      \label{eqnfunctormodenv}
\end{align}

\begin{proposition}\label{propenvfh}
	Let $\cat{A}$ be a cyclic framed $E_2$-algebra in $\smc$, $T\in \Graphs$ a corolla,
	 $H\in\Hbdy(T)$ and $\Sigma=\partial H$.
	The maps $\widehat{\cat{A}}(H^\varphi):	\cat{A}^{\otimes J}\to\cat{A}^{\otimes\Legs(T )}$
	from~\eqref{eqnfunctormodenv}
	are natural in the embedding $\varphi: \sqcup_J \mathbb{D}^2 \to \Sigma$\label{propenvfh0}
	and descend to the factorization homology $\int_\Sigma \cat{A}$, inducing a map 
	\begin{align}
		\PhiA(H): \int_\Sigma\cat{A} \to \cat{A}^{\otimes\Legs(T )} 	 \ .
	\end{align}		
\end{proposition}

We call $\PhiA(H)$ the \emph{handlebody skein module} for the cyclic framed $E_2$-algebra $\cat{A}$ and the handlebody $H$, or just the \emph{$\Phi$-map} for $\cat{A}$ and $H$.

\begin{proof} 
	In order for the maps~\eqref{eqnfunctormodenv} to descend to $\int_\Sigma \cat{A}$, we need to show that they provide a coherent 
	co-cone; this is because of the homotopy colimit description of factorization homology recalled in~\eqref{fheqncolim}.
	To this end, let $J=\{1,\dots,n\}$ 
		and $J'=\{1,\dots,m\}$ (here $n$ or $m$ can be zero in which case $J=\emptyset$ or $J'=\emptyset$, respectively); moreover, let
		$\varphi : \sqcup_J \mathbb{D}^2 = (\mathbb{D}^2)^{\sqcup n}\to \Sigma$ be an oriented embedding and $\varphi' : \sqcup_{J'} \mathbb{D}^2 = (\mathbb{D}^2)^{\sqcup m} \to \mathbb{D}^2$ an $\framed$-operation.
		The operation $\varphi'$ gives rise to a morphism $\cat{A}(\varphi'):\cat{A}^{\otimes m}\to \cat{A}$. It suffices now to prove that for the $i$-th partial composition $\varphi\circ_i\varphi'$ the triangle 
		\begin{equation}\label{propenvfh2e}
			\begin{tikzcd}
				&	\cat{A}^{\otimes(j-1)} \otimes \cat{A}^{\otimes m} \otimes \cat{A}^{\otimes (n-j)}	 \ar[rrrd,"\widehat{\cat{A}}(H^{\varphi\circ_i\varphi'})"] \ar[dd,"\cat{A}^{\otimes (j-1)} \otimes \cat{A}(\varphi')\otimes \cat{A}^{\otimes (n-j)}  ",swap]  \\ &&&& \cat{A}^{\otimes \Legs(T )} \ ,   \\ 
				&		\cat{A} 	^{\otimes n} \ar[rrru,"\widehat{\cat{A}} (H^\varphi)", swap] 
			\end{tikzcd}
		\end{equation}
		commutes up to coherent isomorphism. This isomorphism is obtained as follows: Observe that $\cat{A}^{\otimes (j-1)} \otimes \cat{A}(\varphi')\otimes \cat{A}^{\otimes (n-j)}$ can be written up to canonical isomorphism as
		\begin{align}
			\cat{A}^{\otimes (j-1)} \otimes \cat{A}(\varphi')\otimes \cat{A}^{\otimes (n-j)} &= \widehat{\cat{A}} \left(   ([0,1] \times\mathbb{D}^2)^{\sqcup (j-1)} \sqcup \mathbb{B}^3_{n,1} \sqcup  ([0,1] \times\mathbb{D}^2)^{\sqcup (n-j)}  \right) \ ,\end{align}
		where $\mathbb{B}^3_{n,1}$ is a closed three-dimensional ball with $n$ incoming
		and one outgoing embedded disks
		in its boundary; this is 
		because the modular algebra $\widehat{\cat{A}}$ extends the cyclic algebra $\cat{A}$. 
		The statement follows now from the canonical isomorphism
		\begin{align}
			H^{\varphi\circ_i\varphi'}\cong
			H^\varphi \cup_i \left(   ([0,1] \times\mathbb{D}^2)^{\sqcup (j-1)} \sqcup \mathbb{B}^3_{n,1} \sqcup  ([0,1] \times\mathbb{D}^2)^{\sqcup (n-j)}  \right)
		\end{align}
		and the fact that $\widehat{\cat{A}}$ is a modular algebra and hence respects gluing.
\end{proof}

We now obtain a 
factorization homology description of the ansular functor.
The key ingredient, namely the $\Phi$-maps,
 have already been established.
The rest is relatively formal:
	Since $\int_\Sigma \cat{A}$ is the homotopy colimit
of $\cat{A}^{\otimes J}$ running over all embeddings $\varphi : \sqcup_J \mathbb{D}^2 \to \Sigma$,
there is for any oriented embedding $\varphi : \sqcup_J \mathbb{D}^2 \to \Sigma$ a canonical map
$\cat{A}^{\otimes J} \to \int_\Sigma \cat{A} $ that by abuse of notation we will again denote by $\varphi$.
In the case $J=\emptyset$, we have the unique embedding $\varphi_0 : \emptyset \to \Sigma$.
We consider the following diagram:
\begin{equation}
	\begin{tikzcd}
		I \ar[rrrr,"\cat{O}_\Sigma^\cat{A}"] \ar[rrrrrrrr,bend right = 90,swap,"\widehat{\cat{A}} (H)"]\ar[rrrrd,swap,"\id_I"]&&&& \int_\Sigma\cat{A} \ar[rrrr,"\PhiA(H)"] &&&& \cat{A}^{\otimes\Legs(T )} \\ && &&{\displaystyle \cat{A}^{\otimes \, 0 } =I } \ar[rrrru,swap," \widehat{\cat{A}}(H^{\varphi_0})"] \ar[u,"\varphi_0"]
	\end{tikzcd}
\end{equation}
First we investigate the two upper triangles: The left upper triangle commutes 
by definition of $\cat{O}_\Sigma^\cat{A}$ while the right upper triangle commutes by construction of  $\PhiA(H)$ in Proposition~\ref{propenvfh}. 
The lower triangle commutes up to a canonical isomorphism,	
again by construction. Let us summarize this:

\begin{theorem}[Factorization homology description of the ansular functor]\label{thmfhenv}
	Let $\cat{A}$ be a cyclic framed $E_2$-algebra in $\smc$. Then for any corolla $T\in \Graphs$ and $H\in\Hbdy(T)$ with boundary $\Sigma = \partial H$,
	the diagram
		\begin{equation}\label{propenvfh1e}
		\begin{tikzcd}
		I \ar[rr,"\cat{O}_\Sigma^\cat{A}"] \ar[rrrr,bend right = 45,swap,"\widehat{\cat{A}} (H)"]&& \int_\Sigma\cat{A} \ar[rr,"\PhiA(H)"] && \cat{A}^{\otimes\Legs(T )}
		\end{tikzcd}
		\end{equation}
		commutes up to canonical isomorphism.
\end{theorem}

		For a surface $\Sigma$ (which, for simplicity, we assume to be connected), it is
		 standard that $\int_\Sigma \cat{A}$ is a module over the algebra $\int_{\partial \Sigma \times \opint } \cat{A}$, see e.g.~\cite[Lemma~5.2]{ginot}.
		 The algebra structure on $\int_{\partial \Sigma \times \opint } \cat{A}$ comes from stacking 
		  cylinders over $\partial \Sigma$.
		One can obtain the action 
		$\int_{\partial \Sigma \times \opint } \cat{A}\otimes\int_\Sigma \cat{A}\to \int_\Sigma \cat{A}$ by choosing a collar for $\partial \Sigma$, i.e. an embedding \begin{align} j:(\partial \Sigma \times\opint)
		 \sqcup \Sigma \to \Sigma\label{eqnembj}\end{align} which by functoriality of factorization homology with respect to embeddings induces the action. Up to isomorphism, the action does not depend on the collar. 
		 By applying this to the disk, we find that $\cat{A}$ is a $\int_{\mathbb{S}^1\times\opint }\cat{A}$-module. This also implies that
		 $\cat{A}^{\otimes \Legs(T)}$ is a module over $\int_{\partial \Sigma \times \opint } \cat{A}$ for any corolla $T$ and $\Sigma \in \Surf(T)$.

			\begin{proposition}\label{propmodulemap}
				Let $\cat{A}$ be a cyclic framed $E_2$-algebra in $\smc$, $T\in \Graphs$ a corolla and $H\in\Hbdy(T)$ with boundary $\Sigma=\partial H$.
				Then the map $\PhiA(H):\int_\Sigma \cat{A}\to\cat{A}^{\otimes\Legs(T)}$ comes naturally with the structure of a $\int_{\partial \Sigma \times \opint } \cat{A}$-module map.
				\end{proposition}
			
			\begin{proof}
				Let $\psi : \sqcup_K \mathbb{D}^2 \to \partial \Sigma \times \opint $ and $\varphi : \sqcup_J \mathbb{D}^2\to \Sigma$ be oriented embeddings.  Postcomposition with~\eqref{eqnembj} gives us an embedding $j(\psi,\varphi) : \sqcup_{K\sqcup J} \mathbb{D}^2\to \Sigma$.
				Since $\widehat{\cat{A}}$ is a modular algebra, the map
				\begin{align} \cat{A}^{\otimes (K\sqcup J)} \ra{\widehat{\cat{A}}\left(H^{j(\psi,\varphi)}\right)} \cat{A}^{\otimes\Legs(T)} \label{eqnmapfromj} \end{align} is canonically isomorphic to the partial composition of \begin{align} \widehat{\cat{A}}(H^\varphi) : \cat{A}^{\otimes J} \to \cat{A}^{\otimes \Legs(T)}\label{eqnH-varphi} \end{align}
				with
				\begin{align}
					\widehat{\cat{A}} ( (\partial \Sigma \times \opint) ^\psi  ) : \cat{A}^{\otimes K}  \otimes \cat{A}^{\otimes \Legs(T)} \to  \cat{A}^{\otimes \Legs(T)}    \label{eqnH-varphi2} \ .  \end{align} 
				(A priori, $\widehat{\cat{A}} ( (\partial \Sigma \times \opint) ^\psi  )$ is a map
				$\cat{A}^{\otimes K}  \to  \cat{A}^{\otimes \Legs(T)}  \otimes \cat{A}^{\otimes \Legs(T)} $,
				but we move one of the $\cat{A}^{\otimes \Legs(T)}$ to the left via the pairing.)
			 But the maps~\eqref{eqnmapfromj} induce the map
				\begin{align}
				\label{eqnfirstinducedmap}	\int_{\partial \Sigma \times \opint }\cat{A} \otimes \int_\Sigma \cat{A} \ra{\text{action}}     \int_\Sigma \cat{A} \ra{\PhiA(H)} \cat{A}^{\otimes \Legs(T)} \end{align}
				while the (partial) composition of~\eqref{eqnH-varphi} and~\eqref{eqnH-varphi2}
				\begin{align}
					\cat{A}^{\otimes (K\sqcup J)} \ra{ \id \otimes \widehat{\cat{A}}(H^\varphi)  } \cat{A}^{\otimes K}  \otimes \cat{A}^{\otimes \Legs(T)}
					\ra{    \widehat{\cat{A}} ( (\partial \Sigma \times \opint)  ^ \psi  )  }     \cat{A}^{\otimes \Legs(T)}
					\end{align}
				induces the map 
				\begin{align}
					\label{eqnsecondinducedmap}\left(	\int_{\partial \Sigma \times \opint }\cat{A}\right) \otimes \int_\Sigma \cat{A} \ra{\id  \otimes \PhiA(H)     } \left( \int_{\partial \Sigma \times \opint }\cat{A}\right)  \otimes \cat{A}^{\otimes \Legs(T)}
					\ra{\text{action}}
					\cat{A}^{\otimes \Legs(T)} \ . 
					\end{align}
				Therefore, we obtain a canonical isomorphism between~\eqref{eqnfirstinducedmap}
				and~\eqref{eqnsecondinducedmap}. It is a straightforward observation that this equips $\PhiA(H)$ with the structure of a module map.
 				\end{proof}

	\subsection{Equivariance properties of handlebody skein modules under the mapping class group action}
In this subsection, we discuss the transformation behavior of handlebody skein modules under the action of the mapping class group on factorization homology.

\begin{definition}\label{defpartialM}
	Let $\cat{A}$ be a cyclic framed $E_2$-algebra in $\smc$ and $T$ a corolla.
	For $\Sigma \in \Surf(T)$, we denote by $\smc_\partial \left( \int_{\Sigma} \cat{A}, \cat{A}^{\otimes \Legs(T)}\right)$ the groupoid of $\int_{\partial \Sigma \times \opint } \cat{A}$-module maps $\int_\Sigma \cat{A}\to \cat{A}^{\otimes \Legs(T)}$ with module map isomorphisms as morphisms.
\end{definition}

\begin{remark}\label{remactiononfh}
	Any diffeomorphism
	$f:\Sigma\to \Sigma'$
	induces a 1-morphism
	$f_* : \int_{\Sigma}\cat{A}\to\int_{\Sigma'} \cat{A}$ 
	and an algebra map
	$\alpha_f : \int_{\partial \Sigma \times [0,1]}\cat{A}\to \int_{   \partial \Sigma' \times [0,1]      } \cat{A}$ (if $\Sigma=\Sigma'$, this is the identify)
	such that $f_*$ comes naturally with the structure
	of a
	$\int_{\partial \Sigma \times [0,1]}\cat{A}$-module map
	$\int_{\Sigma}\cat{A}\to   \alpha_f^*  \int_{\Sigma'} \cat{A}$.
	Here $ \alpha_f^*  \int_{\Sigma'} \cat{A}$ is the restriction
	of $  \int_{\Sigma'} \cat{A}$ along $\alpha_f$.
	The isomorphism class of the module map
	that $f$ gives rise to
	depends only on the mapping class of $f$.
\end{remark}

\begin{proposition}\label{propenvfh2x}
	Let $\cat{A}$ be a cyclic framed $E_2$-algebra in $\smc$, $T\in \Graphs$ a corolla and $f:\Sigma\to \Sigma'$ a diffeomorphism.
	Then for any $H\in \Hbdy(T)$ with $\partial H=\Sigma$, the  diagram
	\begin{equation}\label{propenvfh2ex}
		\begin{tikzcd}
			\int_\Sigma \cat{A}	 \ar[rrrd,"\PhiA(H)"] \ar[dd,swap,"f_*"]  \\ 
			&&& \cat{A}^{\otimes \Legs(T )}  \\ 
			\int_{\Sigma'} \cat{A} 	 \ar[rrru,"\PhiA(f.H)", swap] 
		\end{tikzcd}
	\end{equation}
	commutes up to canonical isomorphism of $\int_{\partial \Sigma \times \opint }\cat{A}$-module maps, i.e.\ up to an isomorphism in
	the category
	$\smc_\partial \left( \int_{\Sigma} \cat{A}, \cat{A}^{\otimes \Legs(T)}\right)$.
	This isomorphism,
	when combined with the isomorphism from Theorem~\ref{thmfhenv} and the homotopy fixed point structure of $\cat{O}_\Sigma^\cat{A}$,
	yield the isomorphism $\widehat{\cat{A}}\left(\widetilde f\right):\widehat{\cat{A}}(H)\to \widehat{\cat{A}}(f.H)$:
	\begin{equation}\label{propenvfh2exy0}
		\begin{tikzcd}
			&	\int_\Sigma \cat{A}	 \ar[rrrd,"\PhiA(H)"] \ar[dd,"f_*"]   \\ I \ar[ur,"\cat{O}_\Sigma^\cat{A}"] \ar[dr,swap,"\cat{O}_{\Sigma'}^\cat{A}"]  \ar[rrrr,bend left=100,"\widehat{\cat{A}}(H)"]   \ar[rrrr,bend left=-100,swap,"\widehat{\cat{A}}(f.H)"] 
			&&&& \cat{A}^{\otimes \Legs(T )} && = & \widehat{\cat{A}}\left(\widetilde f\right)  \\ 
			&		 \int_{\Sigma'} \cat{A} 	 \ar[rrru,"\PhiA(f.H)", swap] 
		\end{tikzcd}
	\end{equation}
\end{proposition}

\begin{proof}
	Let
	$\varphi : \sqcup_J \mathbb{D}^2 \to \Sigma$ be
	an oriented embedding, then $f\circ \varphi:\sqcup_J \mathbb{D}^2 \to \Sigma'$ is also an oriented embedding. 
	The diffeomorphism $\widetilde f : H \to f.H$ from Lemma~\ref{lemmaisofibprop}
	induces a diffeomorphism
	$H^\varphi \to (f.H)^{f \circ \varphi}$ giving us, after evaluation with the modular extension $\widehat{\cat{A}}$, an isomorphism
	$\widehat{\cat{A}}(H^\varphi)\to \widehat{\cat{A}}((f.H)^{f \circ \varphi})$. By the construction of $\PhiA$ these isomorphisms descend to factorization homology to yield the isomorphism~\eqref{propenvfh2ex}. It is an isomorphism of $\int_{\partial \Sigma \times \opint }\cat{A}$-module maps because $f$ respects the boundary parametrizations.
	Now~\eqref{propenvfh2exy0} can be verified directly using the definitions.
\end{proof}

A relevant special case of Proposition~\ref{propenvfh2x} is the following:

\begin{corollary}\label{corxi}
	Let $\cat{A}$ be a cyclic framed $E_2$-algebra in $\smc$, $T\in \Graphs$ a corolla and $g:H\to H'$ a diffeomorphisms between $H,H'\in\Hbdy(T)$.
	Set $f:= \partial g$.
	Then after the identification $\PhiA(f.H) \cong \PhiA(H')$ induced by $\xi_g$ from Remark~\ref{remisofibprop2}, we obtain an isomorphism 
	of $\int_{\partial \Sigma \times \opint }\cat{A}$-module maps
	filling the triangle
	\begin{equation}\label{propenvfh2exX}
		\begin{tikzcd}
			\int_\Sigma \cat{A}	 \ar[rrrd,"\PhiA(H)"] \ar[dd,swap,"f_*=\partial g_*"]  \\ 
			&&& \cat{A}^{\otimes \Legs(T )}  \\ 
			\int_{\Sigma'} \cat{A} 	 \ar[rrru,"\PhiA(H')", swap] 
		\end{tikzcd}
	\end{equation}
	such that
	\begin{equation}\label{propenvfh2exy0X}
		\begin{tikzcd}
			&	\int_\Sigma \cat{A}	 \ar[rrrd,"\PhiA(H)"] \ar[dd,"f_*=\partial g_*"]   \\ I \ar[ur,"\cat{O}_\Sigma^\cat{A}"] \ar[dr,swap,"\cat{O}_{\Sigma'}^\cat{A}"]  \ar[rrrr,bend left=100,"\widehat{\cat{A}}(H)"]   \ar[rrrr,bend left=-100,swap,"\widehat{\cat{A}}(H')"] 
			&&&& \cat{A}^{\otimes \Legs(T )} && = & \widehat{\cat{A}}\left(g\right)  \\ 
			&		 \int_{\Sigma'} \cat{A} 	 \ar[rrru,"\PhiA(H')", swap] 
		\end{tikzcd}
	\end{equation}
	In particular, if a diffeomorphism $g$ of $\Sigma$ lies in the diffeomorphism group of some handlebody $H$, there is a
	canonical isomorphism
	\begin{align}
		\label{eqnphicomp}\nu_g : 	\PhiA(H)g_*\cong \PhiA(H) \ . 
	\end{align}
	This endows $\PhiA(H)$ with the structure of a homotopy fixed  point for the $\Diff(H)$-action on
	the  category
	$\smc_\partial \left( \int_{\Sigma} \cat{A}, \cat{A}^{\otimes \Legs(T)}\right)$.
\end{corollary} 

\begin{remark}
This induces an action of $\Map(H)$ on $\Aut_\partial(\PhiA(H))$:  A mapping class $g:H\to H$ acts via the composition
\begin{align}
	\Aut_\partial(\PhiA(H)) \longrightarrow \Aut_\partial(\PhiA(H)g_*)\longrightarrow \Aut_\partial(\PhiA(H)) \ , 
\end{align}
where the first map is obtained by applying the functor $g_*$ and the second one is the pullback along the fixed point structure $\nu_g$. \end{remark}

 		\subsection{The connection to skein algebras and skein modules\label{secskeinalgmod}}	It was briefly mentioned in the introduction that the $\PhiA(H)$-maps can be understood
as generalized skein modules for handlebodies, formulated via the modular envelope and factorization homology. 
We will explain now why this is indeed just a different perspective on Theorem~\ref{thmfhenv}:
 			For a framed $E_2$-algebra $\cat{A}$ in $\cat{S}$ and 
 			a surface $\Sigma$, one defines
 			the \emph{skein algebra of $\cat{A}$ and $\Sigma$}
 			by
 			\begin{align}
 			\SkAlg_\cat{A}(\Sigma) := \End(\cat{O}_\Sigma^\cat{A}) \ , 
 			\end{align}
 			where the endomorphisms on the right hand side are taken in the hom category $\cat{S}\left(I,\int_\Sigma \cat{A}\right)$.
 The crucial fact that this definition of the skein algebra coincides with the `traditional' definition in the case where $\cat{A}$ is a semisimple ribbon category follows from \cite{cooke}.
 		
 		We will now see that this implies that
 		the value $\widehat{\cat{A}}(H)$ of the modular extension $\widehat{\cat{A}}$ of $\cat{A}$ on $H$
 		 becomes a module over the skein algebra of $\Sigma = \partial H$. 
 		Indeed, the
 		 action is the algebra map
 		\begin{align}
 		\Add: \SkAlg_\cat{A}(\Sigma) \to \End( \widehat{\cat{A}}(H) ) \        \label{eqnaddmap}  
 		\end{align}
		(the endomorphisms on the right hand side are taken in the hom category $\cat{S}\left(I, \cat{A}^{\otimes \Legs(T)}\right)$; the notation $\Add$ is borrowed from~\cite{masbaumroberts})
 		defined as the composition
 		\begin{align}\label{eqndefadd}
 			\SkAlg_\cat{A}(\Sigma)=\End(\cat{O}_\Sigma^\cat{A}) \ra{\PhiA(H)\circ -}
 			\End(\PhiA(H)\circ\cat{O}_\Sigma^\cat{A}) \stackrel{\text{Theorem~\ref{thmfhenv}}}{\cong}
 		\End(	\widehat{\cat{A}}(H) )\ . 
 		\end{align}
	
		In  case $\cat{A}$ is ribbon and semisimple, following~\cite{Walker} and \cite[Section~3.7]{skeinfin}, any handlebody $H$ with $n$ embedded disks in its boundary $\partial H=\Sigma$ also leads to a functor
		\begin{align}
		\Sk_\cat{A}(H):	\int_\Sigma \cat{A} \longrightarrow \cat{A}^{\otimes n}
		\end{align}
		whose components are given by relative skein modules. This functor tautologically also satisfies the universal property of Proposition~\ref{propenvfh1e} so that there is a canonical isomorphism $\Sk_\cat{A}(H)\cong \PhiA(H)$ in that case.

\section{Connected cyclic framed $E_2$-algebras and their associated extensions of the surface operad}
One of the main goals of this article is to identify general necessary and
sufficient conditions for a cyclic framed $E_2$-algebra  to extend to a modular functor.
In this subsection, we encounter the main technical condition that will be relevant in this context, namely the notion of \emph{connectedness} that will be defined in terms of handlebody skein modules. 
In this section,
 we will be occupied with studying this notion in detail;
we will see why the notion is relevant in~Section~\ref{secclassmf},
and we will discuss how to produce examples in Section~\ref{seccof}.

\subsection{The $\Omega_\cat{A}$-groupoids and the map $\SurfA\to\Surf$ of modular operads}\label{secsurfa}
	This subsection contains a technical construction that lies at the heart of the paper: For any cyclic framed $E_2$-algebra $\cat{A}$, we will construct a modular operad $\SurfA$ (depending on $\cat{A}$) and a canonical map $\SurfA \to \Surf$.
	The construction of a modular functor in Theorem~\ref{thmdefsurfa2} will then produce a $\SurfA$-algebra structure on $\cat{A}$.

	\begin{definition}\label{defomega}
		Let $\cat{A}$ be a cyclic framed $E_2$-algebra in $\smc$ and $T$ a corolla. For $\Sigma \in \Surf(T)$,
		we define 
		\begin{align}\label{eqndefomega} \Omega_\cat{A} (\Sigma) := \left\{ \begin{array}{c}
				\text{replete full subgroupoid of} \  \smc_\partial \left( \int_\Sigma \cat{A}, \cat{A}^{\otimes \Legs(T)}\right) \\ \text{spanned by all module maps} \
				\PhiA(H) : \int_\Sigma \cat{A} \to \cat{A}^{\otimes \Legs(T)} \\  \text{given in Proposition~\ref{propenvfh} for any handlebody $H$ with boundary $\Sigma$} \\ \text{(with the module map structure from Proposition~\ref{propmodulemap})}\end{array}\right\} \ .
		\end{align}
		Here \emph{replete} full (as opposed to just full) means
		that $\Omega_\cat{A} (\Sigma)$ is closed under isomorphisms: It does not only contain all the $\PhiA(H)$, but also all maps $\int_\Sigma \cat{A} \to \cat{A}^{\otimes \Legs(T)}$ isomorphic, as module map, to some $\PhiA(H)$. (Of course, passing from a full subcategory to its repletion does not change the subcategory up to equivalence.)
		The definition~\eqref{eqndefomega} is extended to the case of non-connected $T=\sqcup_{i=1}^n T^{(i)}$ by $\Omega_\cat{A}     (\sqcup_{i=1}^n \Sigma_i)=\prod_{i=1}^n \Omega_\cat{A}(\Sigma_i)$ for $\Sigma_i\in\Surf(T^{(i)})$ for $i=1,\dots,n$.
	\end{definition}

Roughly speaking, the modular operad $\SurfA$ will be constructed as some sort of homotopy quotient of this collection of groupoids by the actions of the mapping class groups: It will be made of pairs $(\Sigma, \Phi_\cA(H))$ with $\Sigma=\partial H$ with operadic composition induced by the gluing of surfaces, exploiting the fact that the functors $\Phi_\cA(H)$ are compatible with those.

In order to turn this into a precise definition, recall that
	for a functor $F:\cat{C}\to\Cat$, 
	i.e.\ a category-valued precosheaf, 
	the \emph{Grothendieck construction} 
	of $F$ is the category $\Gr F$
	formed by pairs $(c,x)$, 
	where $c\in\cat{C}$ and $x \in F(c)$, see \cite[Section~I.5]{maclanemoerdijk} for a textbook reference.
	A morphism $(c,x)\to (c',x')$ 
	is a morphism $f: c \to c'$ 
	in $\cat{C}$
	together with a morphism
	$\alpha : F(f)x \to x'$ in $F(c')$.
	The Grothendieck construction 
	comes with a canonical projection functor $\Gr F \to \cat{C}$.
	(Often, the Grothendieck construction of $F$ is denoted by $\int F$. In this article,
	we will not use the integral sign for the Grothendieck construction
	because of the notational conflict with factorization homology.)
	There is also a version of the Grothendieck construction for functors $F:\cat{C}^\op \to \Cat$, i.e.\ for category-valued presheaves. It is defined as the category $\Gr F$ 
	formed by pairs $(c,x)$, where $c\in\cat{C}$ and $x \in F(c)$.
	A morphism $(c,x)\to (c',x')$ is a morphism $f: c \to c'$ 
	in $\cat{C}$
	together with a morphism
	$\alpha : x \to F(f)x'$ in $F(c)$.
	It comes again with a projection functor $\Gr F \to \cat{C}$ (not to $\cat{C}^\op$).
	Later we will also need a Grothendieck construction for higher categories, see e.g.\ the treatment in 
	\cite[Chapter~3.2]{htt}.
	
	Recall from Remark~\ref{remSURF} the topological modular 
	operad $\SURF$ of surfaces and \emph{diffeomorphisms} as opposed to mapping classes.
	Clearly, we have for every corolla $T$
	a functor
	\begin{align}
		\smc_\partial \left( \int_{-} \cat{A}, \cat{A}^{\otimes \Legs(T)}\right) \   : \  \SURF^\op(T) \to \Grpd \ , \quad \Sigma \mapsto  \smc_\partial \left( \int_{\Sigma} \cat{A}, \cat{A}^{\otimes \Legs(T)}\right) \ . \label{eqnmodulemapsfunctor}
	\end{align}
	This gives us also a functor
	\begin{align} \Omega_\cat{A}:\SURF^\op(T) \to \Grpd \ , \quad \Sigma \mapsto \Omega_\cat{A}(\Sigma) \ ,  \label{eqnomegafunctor} \end{align}
	by restriction in range. To see that this is true, we only need to verify that the
	$\Omega_\cat{A}$-subgroupoids
	$\Omega_\cat{A}(\Sigma)\subset \smc_\partial(   \int_\Sigma \cat{A},\cat{A}^{\otimes \Legs(T)}   )$
	are stable under the precomposition with diffeomorphisms, but this follows from their definition and
	Proposition~\ref{propenvfh2x}.

	By taking the Grothendieck construction over the functors~\eqref{eqnmodulemapsfunctor}, we will define the following auxiliary operad:

	\begin{definition}\label{defauxga}
		We denote by
	$
	\cat{G}_\cat{A}$ the topological modular operad
	sending a corolla $T$ to the Grothendieck construction
	\begin{align}     \left(   \cat{G}_\cat{A}
		\right)(T):=    \Gr \left(   \SURF^\op(T) \ra{   \smc_\partial \left( \int_{-} \cat{A}, \cat{A}^{\otimes \Legs(T)}\right)   } \Grpd  \right) \ .  \label{eqnprescriptioncorolla}
	\end{align}
	of the functor~\eqref{eqnmodulemapsfunctor}.
	\end{definition}
	
	Let us specify the operadic structure:
	First of all, we extend the prescription~\eqref{eqnprescriptioncorolla} to non-connected objects in $\Graphs$ by sending disjoint unions of corollas
	to products of groupoids. 
	Next we need to extend  the definition of $ \cat{G}_\cat{A} $ to morphisms in $\Graphs$. Let $\Gamma : T \to T'$ be a morphism in $\Graphs$. Without loss of generality, we can assume that $T'$ is connected, i.e.\ a corolla. We write the decomposition of $T$ into corollas as $T=\sqcup_{i=1}^n T^{(i)}$. 
	The image $\left(  \cat{G}_\cat{A} \right) (\Gamma)$ of $\Gamma$ under $ \cat{G}_\cat{A}$, that we now want to construct, is by definition a functor
	
	\footnotesize
	\begin{align}
		\left(  \cat{G}_\cat{A} \right) (\Gamma) \ : \ 	\prod_{i=1}^n \Gr \left(   \Surf^\op(T) \ra{   \smc_\partial \left( \int_{-} \cat{A}, \cat{A}^{\otimes \Legs(T^{(i)})}\right)   } \Grpd  \right) \to  \Gr \left(   \Surf^\op(T) \ra{   \smc_\partial \left( \int_{-} \cat{A}, \cat{A}^{\otimes \Legs(T')}\right)   } \Grpd  \right) \ .  
	\end{align}
	\normalsize
	We define this functor as follows:
	An object on the left consists of a surface $\Sigma = \sqcup_{i=1}^n \Sigma_i \in \Surf(T)$
	and a family $\Psi = (\Psi_i)_{1\le i\le n}$ of maps $\Psi_i \in \smc_\partial \left( \int_{\Sigma_i} \cat{A}, \cat{A}^{\otimes \Legs(T^{(i)})}\right)$. 
	Since $\Surf$ is already a modular operad, we may evaluate $\Surf$ on $\Gamma$ and obtain a functor $\Gamma_* : \Surf(T)\to\Surf(T')$ sending $\Sigma$ to $\Gamma_* \Sigma$. 
	We will accomplish the definition of
	$	\left(  \cat{G}_\cat{A} \right) (\Gamma)$
	by assigning to $\Psi$ an object $\Gamma_*\Psi \in \smc_\partial \left( \int_{\Gamma_* \Sigma} \cat{A}, \cat{A}^{\otimes \Legs(T)}\right)$. 
	Then we can define \begin{align} 	\left(  \cat{G}_\cat{A} \right) (\Gamma)  (\Sigma,\Psi):= (\Gamma_*\Sigma,\Gamma_* \Psi) \ .     \label{eqnsetdef}    \end{align}

	The definition of $\Gamma_*\Psi$ is the hard part and done
	as follows:
	First we observe that $\Gamma$ induces a map
	\begin{align}
		\bigotimes_{i=1}^n \cat{A}^{\otimes \Legs(T^{(i)})} \to \cat{A}^{\otimes \Legs(T')} \ ;     \label{eqncontractmap}
	\end{align}	
	this is \emph{dual} to constructions used for the definition of the modular endomorphism operad in \cite[Section~2.3]{cyclic}. 
	Essentially, this map performs cyclic permutations
	and contracts copies of $\cat{A}$ via the pairing $\kappa :\cat{A}\otimes\cat{A}\to I$ --- just as prescribed by $\Gamma$. 
	As a consequence, we obtain the following map:
	
	\footnotesize
	\begin{align}
		\prod_{i=1}^n	\smc_\partial \left( \int_{\Sigma_i} \cat{A}, \cat{A}^{\otimes \Legs(T^{(i)})}\right)   \ra{\otimes} \smc_\partial \left( \bigotimes_{i=1}^n    \int_{\Sigma_i} \cat{A}, \bigotimes_{i=1}^n \cat{A}^{\otimes \Legs(T^{(i)})}\right) \ra{   \eqref{eqncontractmap}  }     \smc_\partial \left( \bigotimes_{i=1}^n    \int_{\Sigma_i} \cat{A}, \cat{A}^{\otimes \Legs(T')} \right)  \label{eqnfhdescend}
	\end{align}
	\normalsize
	The `$\partial$' reminds us of the restriction to module maps with respect to the boundary, which means precisely:
	\begin{itemize}\item  For the first
		step,
		these are maps of $\int_{\partial \Sigma_i \times [0,1]}\cat{A}$-modules in each factor.
		\item For the second step, these are maps of $\int_{\partial \Sigma \times \opint } \cat{A}$-modules.
		\item  For the third step, these are maps of $\int_{\partial \Gamma_*\Sigma \times \opint } \cat{A}$-modules,
		where
		$\bigotimes_{i=1}^n    \int_{\Sigma_i} \cat{A}$ is seen as $\int_{\partial \Gamma_* \Sigma \times \opint  }\cat{A}$-module by restriction along the algebra map \begin{align}
			\int_{\partial \Gamma_* \Sigma \times \opint  }\cat{A}  \to    \int_{\partial \Sigma \times \opint } \cat{A}  \end{align}induced by $\partial \Gamma_*\Sigma \subset \partial \Sigma$. 
	\end{itemize} 
	In order to obtain
	the desired map $\Gamma_*\Psi : \int_{\Gamma_* \Sigma} \cat{A} \to  \cat{A}^{\otimes \Legs(T')}$, we need to show that the image  of
	$\Psi$ under~\eqref{eqnfhdescend}
	that we will denote by $\Psi^\kappa$
	actually factors naturally through the map $\bigotimes_{i=1}^n    \int_{\Sigma_i} \cat{A} \to \int_{\Gamma_* \Sigma} \cat{A}$ induced by the gluing prescribed by $\Gamma$.  
	To this end, we will use the excision property of factorization homology which, in the case at hand, tells us the following: If $\Gamma$ prescribes a gluing of $r$ pairs of boundary components, then $\bigotimes_{i=1}^n    \int_{\Sigma_i}\cat{A}$ is a module over $\left(   \int_{\mathbb{S}^1\times\opint }   \cat{A} \right)^{\otimes r}$ in two ways; we denote the actions by $\alpha_\ell : \left(   \int_{\mathbb{S}^1\times\opint }   \cat{A} \right)^{\otimes r} \otimes \bigotimes_{i=1}^n    \int_{\Sigma_i} \cat{A} \to \bigotimes_{i=1}^n    \int_{\Sigma_i} \cat{A}$, $\ell=1,2$. 
	By the excision property of factorization homology, $\int_{\Gamma_* \Sigma} \cat{A}$ is the homotopy coequalizer of $\alpha_1$ and $\alpha_2$. 
	Hence, we need to show that $\Psi^\kappa: \bigotimes_{i=1}^n    \int_{\Sigma_i} \cat{A}\to \cat{A}^{\otimes \Legs(T')}$ coequalizes $\alpha_1$ and $\alpha_2$ up to isomorphism (this is structure and not just a property). 
	To see this, we first observe that $\bigotimes_{i=1}^n \cat{A}^{\otimes \Legs(T^{(i)})}$ also comes with two actions of $\left(   \int_{\mathbb{S}^1\times\opint }   \cat{A} \right)^{\otimes r}$, namely by action on the first and the second copy of $\cat{A}$ associated to a gluing pair.  We denote these two actions $\beta_1$ and $\beta_2$. 
	Now consider the diagram:
	\begin{equation}
		\begin{tikzcd}
			\left(   \int_{\mathbb{S}^1\times\opint }   \cat{A} \right)^{\otimes r} \otimes  \bigotimes_{i=1}^n    \int_{\Sigma_i} \cat{A}   \ar[rr,"\alpha_\ell"] \ar[dd," \id \otimes \bigotimes_{i=1}^n \Psi_i  ",swap] &&     \bigotimes_{i=1}^n    \int_{\Sigma_i} \cat{A}    \ar[dd,"\bigotimes_{i=1}^n \Psi_i"]  \ar[rr,"\Psi^\kappa"] && \cat{A}^{\otimes \Legs(T')} \\ \\ 
			\left(   \int_{\mathbb{S}^1\times\opint }   \cat{A} \right)^{\otimes r}   \otimes \bigotimes_{i=1}^n \cat{A}^{\otimes \Legs(T^{(i)})}   \ar[rr, swap," \beta_\ell"]  & &    \bigotimes_{i=1}^n \cat{A}^{\otimes \Legs(T^{(i)})}   \ar[uurr,swap,"  \eqref{eqncontractmap}   "]         \ . 
		\end{tikzcd}
	\end{equation}
	The left square commutes by a canonical isomorphism (this is part of data that makes the $\Psi_i$ module maps). The triangle on the right commutes by definition. 
	The needed isomorphism $\Psi^\kappa \alpha_1 \cong \Psi^\kappa  \alpha_2$ can now be 
	obtained by evaluating the above diagram at $\ell=1$ and $\ell=2$ and by combining these two cases. 
	This uses additionally the fact that~\eqref{eqncontractmap} homotopy coequalizes $\beta_1$ and $\beta_2$ thanks to the cyclicity of the pairing.	
	This completes the construction of $\Gamma_*\Psi$.
	
	Having completed the definition of $\Gamma_*\Psi$, we can define $\left(  \cat{G}_\cat{A} \right) (\Gamma)(\Sigma,\Psi)$ via~\eqref{eqnsetdef}. All of the constructions so far are functorial in module maps. Therefore, the definition extends to morphisms, thereby giving us the functor 	$	\left(  \cat{G}_\cat{A} \right) (\Gamma)$.
	This concludes the construction of $ \cat{G}_\cat{A}$ as modular operad.
	The canonical projection to the base of the Grothendieck construction yields a morphism $ \cat{G}_\cat{A}\to \SURF$
	of modular operads.

	\begin{lemma}\label{lemmasubopead}
	By	performing the Grothendieck construction~\eqref{eqnprescriptioncorolla}
	with the $\Omega_\cat{A}$-groupoids instead of 
		the groupoids
		$ \smc_\partial \left( \int_{-} \cat{A}, \cat{A}^{\otimes \Legs(T)}\right) $, the construction from Definition~\ref{defauxga} restricts in range and gives rise to a suboperad $\cat{G}_\cat{A}^\Omega \subset \cat{G}_\cat{A}$.
		\end{lemma}

\begin{proof} 
	This follows from the fact that the $\Omega_\cat{A}$-subgroupoids are stable 
	under compositions and cyclic permutations. The latter is obvious, 
	the former follows from the fact that the gluing of module 
	maps $\int_\Sigma \cat{A}\to\cat{A}^{\otimes \Legs(T)}$, 
	when applied to maps of the form $\PhiA(H)$, translates to the gluing of handlebodies.  
	More precisely, the functor
	$	\left(  \cat{G}_\cat{A} \right) (\Gamma)$
	associated to a morphism $\Gamma:T\to T'$ in $\Graphs$
	sends $\PhiA(H)$ for $H\in\Hbdy(T)$ to $\PhiA(\Gamma_* H)$.
	This is a consequence of the construction of the $\PhiA$-maps 
	in Proposition~\ref{propenvfh} and properties of the modular extension $\widehat{\cat{A}}$.
	\end{proof}

	So far we have maps $\cat{G}_\cat{A}^\Omega \to \cat{G}_\cat{A}\to\SURF$ of modular operads. 
	After applying the symmetric monoidal fundamental groupoid functor, we obtain a map $\Pi \cat{G}_\cat{A}^\Omega
	\to\Surf$ of groupoid-valued operads.
	
	\begin{definition}\label{DefsurfA}
	We define $\SurfA$ via \begin{align}\label{defsurfa} \SurfA:= \Pi \cat{G}_\cat{A}^\Omega \ , \end{align}
	thereby giving us a map $p_\cat{A}:\SurfA \to \Surf$ of modular operads. 
	\end{definition}

	\begin{theorem}\label{thmdefsurfa}
		For any cyclic framed $E_2$-algebra $\cat{A}$ in $\smc$ and $T \in \Graphs$, the modular operad $\SurfA$ is given by 
		\begin{align}
			\SurfA (T)   \simeq \Gr \left(     \Surf(T)^\op \to \Cat \ , \quad \Sigma \mapsto \Omega_\cat{A}(\Sigma) \right)   \label{eqnsurfA}
		\end{align} up to a canonical equivalence.
		Moreover, it comes with a map
		$s_\cat{A}:\Hbdy \to \SurfA$ of modular operads over $\Surf$ sending $H\in\Hbdy(T)$ to $\PhiA(H)$.
	\end{theorem}

		This exhibits $\SurfA(T)$ as the disjoint union of the homotopy quotients of $\Omega_\cat{A}(\Sigma)$ by $\Map(\Sigma)$, where $\Sigma$ runs over all operations in $\Surf(T)$.

	\begin{proof}[\slshape Proof of Theorem~\ref{thmdefsurfa}]
	We can write $\SurfA(T)$ again as a Grothendieck construction over the homotopy fibers of $p_\cat{A}:\SurfA \to \Surf$, see \cite{hollander}, but it is unclear that the homotopy fiber over $\Sigma \in \Surf(T)$ is actually the groupoid $\Omega_\cat{A}(\Sigma)$. 
			By construction the homotopy fiber of $\cat{G}_\cat{A}^\Omega \to \SURF$ over $\Sigma$ is $\Omega_\cat{A}(\Sigma)$, but the homotopy fiber will generally not commute with taking the fundamental groupoid. 
			To see that $\Omega_\cat{A}(\Sigma)$ is indeed the fiber of $p_\cat{A}$ over $\Sigma$, we look at the long exact sequence of the homotopy groups for the fibration $\cat{G}_\cat{A}^\Omega(T) \to \SURF(T)$. To this end, choose a base point in the fiber over $\Sigma$, i.e.\ some $\PhiA(H)$ with $\partial H=\Sigma$. 
			With the definition~\eqref{defsurfa}, we find the exact sequence
			\begin{center}
				\begin{equation}\centering
					\hspace*{-5cm}\begin{tikzcd}
						1\ar[rrr] &&&
						\pi_2(   \cat{G}_\cat{A}^\Omega(T),\PhiA(T)  )  	\ar[rrr]
						&&& \pi_1(\Diff_0(\Sigma))   \ar[out=355,in=175,dllllll]    \\
						\pi_1(\Omega_\cat{A}(\Sigma),\PhiA(H))	 \ar[rrr] &&&  \pi_1(\SurfA(T) , \PhiA(H)) \ar[rrr] &&& \Map(\Sigma) \ar[out=355,in=175,dllllll] \\ \pi_0(\Omega_\cat{A}(\Sigma))   \ar[rrr] &&&   \pi_0( \SurfA(T)  ) \ar[rrr] &&& \pi_0(\Surf(T)) \ , 
					\end{tikzcd} 
			\end{equation}\end{center}
			where $\Diff_0(\Sigma)$ is the unit component of the topological group $\Diff(\Sigma)$. 
			For $\Omega_\cat{A}(\Sigma)$ to be the homotopy fiber of $p_\cat{A}$ over $\Sigma$, we need $\pi_1(\Omega_\cat{A}(\Sigma),\PhiA(H)) \to \pi_1(\SurfA(T) , \PhiA(H))$ to be injective. By exactness this is the case if and only if $\pi_1(\Diff_0(\Sigma))=\pi_2(\SURF(T), \Sigma) \to \pi_1(\Omega_\cat{A}(\Sigma),\PhiA(H))$ is trivial. 
			This map is indeed trivial because the map $\Diff_0(H)\to\Diff_0(\Sigma)$ is an equivalence
			(this is a consequence of \cite[Theorem~2]{hatcher}, see also the comments in~\cite[Section~2]{mwdiff})
			and because $\PhiA(H)$ is a homotopy fixed point under the $\Diff(H)$-action by Corollary~\ref{eqnphicomp}.
			This tells us that $p_\cat{A}:\SurfA \to \Surf$ is indeed a map with homotopy fiber
			$\Omega_\cat{A}(\Sigma)$ over $\Sigma$, thereby giving us~\eqref{eqnsurfA}.
			
			Finally, the fact the map $s_\cat{A}:\Hbdy \to \SurfA$ is a map of modular operads over $\Surf$ follows from the compatibility of the $\PhiA$-maps with gluing, their equivariance behavior from Corollary~\ref{corxi} and $p_\cat{A} \PhiA(H)=\partial H$. 
	\end{proof}

\subsection{Connectedness of the $\Omega_\cat{A}$-groupoids: General connectedness and reduction to genus one\label{secconnectedgenusone}}
We will now investigate the connectedness of the $\Omega_\cat{A}$-groupoids. This will inform the definition of a connected cyclic framed $E_2$-algebra in the next subsection.
 
To this end, we need to establish some terminology: 
With the presentation of $\Map(\Sigma)$ given in 
\cite[Section~4.4.4]{farbmargalit} we can find two handlebodies $H$ and $H'$ such that the subgroups $\Map(H)$ and $\Map(H')$ \emph{together} generate $\Map(\Sigma)$. We call such a pair $(H,H')$ of handlebodies a \emph{generating pair of handlebodies with boundary $\Sigma$}.

\begin{lemma}\label{lemmaconectedcompl}
	For a cyclic framed $E_2$-algebra $\cat{A}$ in $\cat{S}$, the groupoid $\Omega_\cat{A}(\Sigma)$ is connected if and only if
	for some generating pair $(H,H')$ of handlebodies with boundary $\Sigma$ the $\int_{\partial \Sigma \times  \opint} \cat{A}$-module maps $\PhiA(H)$ and $\PhiA(H')$ are isomorphic as module maps.
	\end{lemma}

\begin{proof}
	If $\Omega_\cat{A}(\Sigma)$ is connected, then clearly $\PhiA(H)\cong\PhiA(H')$ for any two handlebodies with boundary $\Sigma$ (simply by definition).
	Let us prove the converse:
	After the choice of an isomorphism $u : H \to H'$, we have a group isomorphism
	\begin{align}\Map(H) \to \Map(H') \ , \quad a \mapsto uau^{-1} \ . 
	\end{align}
	With $s:= \partial u\in \Map(\Sigma)$, 
	we may now without loss of generality assume that $H'=s.H$, see Remark~\ref{remisofibprop2}.
	For arbitrary handlebody $\widetilde H$ with boundary $\Sigma$,
	we now need to show $\PhiA(\widetilde H)\cong \PhiA(H)$ as module maps. First note that $\widetilde H$ can be written as $\widetilde H = f.H$ for some $f\in \Map(\Sigma)$. Since $\Map(H)$ and $\Map(H')$ generate $\Map(\Sigma)$ and since $\Map(H)\cong \Map(H')$ via conjugation with $u$, we can write
	$f= a_1 ub_1 u^{-1} \dots a_n ub_nu^{-1} $ with $a_1,b_1,\dots,a_n,b_n \in \Map(H)$. 
	With Proposition~\ref{propenvfh2x}, we find
	\begin{align}
	\PhiA(\widetilde H)=\PhiA(f.H)\cong\PhiA(H) f_*^{-1} \cong \PhiA(H) s_* {\partial b_n^{-1}}_* s^{-1}_* {\partial a_n^{-1}}_* \dots s_* {\partial b	_1^{-1}}_* s^{-1}_* {\partial a_1^{-1}}_* \ , 
	\end{align}
	where $\cong$ is always an isomorphism of module maps.
	For this reason, it suffices to prove:
	\begin{align}
	\PhiA(H) s_* &\cong \PhiA(H) \ , \label{eqnconn1} \\
	\PhiA(H) s_*^{-1} &\cong \PhiA(H) \ , \label{eqnconn2} \\
	\PhiA(H) \partial a_* &\cong \PhiA(H) \quad \text{for all}\quad a \in \Map(H)  \label{eqnconn3}\ . 
	\end{align}
	First we use Proposition~\ref{propenvfh2x} to conclude $\PhiA(H')=\PhiA(s.H)=\PhiA(H)s_*^{-1}$. With $\PhiA(H')\cong\PhiA(H)$ --- which holds by assumption --- this yields $\PhiA(H)\cong \PhiA(H)s_*^{-1}$ and proves~\eqref{eqnconn2}. After applying $s_*$ from the right to both sides, we obtain~\eqref{eqnconn1}. Finally, \eqref{eqnconn3} is a consequence of the homotopy
	$\Diff(H)$-fixed point structure of $\PhiA(H)$, see 
	Corollary~\ref{corxi}.
	\end{proof}

\begin{proposition}\label{propconnected}
	For a cyclic framed $E_2$-algebra $\cat{A}$ in $\cat{S}$, the following conditions are equivalent:
	\begin{pnum}
		\item All groupoids $\Omega_\cat{A}(\Sigma)$, where $\Sigma$ runs over all surfaces, are connected.
		\label{propconnectedi}

		\item For some generating pair $(H,H')$
		of handlebodies whose  boundary  is the torus $\mathbb{T}^2_1$
		with one boundary component,
		the handlebody skein modules $\PhiA(H)$ and $\PhiA(H')$ are isomorphic as $\int_{\mathbb{S}^1 \times [0,1]} \cat{A}$-module maps.\label{propconnectediv}
		\end{pnum}
	\end{proposition}

\begin{proof}
	Clearly, \ref{propconnectedi} $\Rightarrow$ \ref{propconnectediv}.  Let us prove 
\ref{propconnectediv} $\Rightarrow$ \ref{propconnectedi}: 
By Lemma~\ref{lemmaconectedcompl}  the condition~\ref{propconnectediv} implies that $\Omega_\cat{A}(\mathbb{T}_1^2)$ is connected.
Therefore, it remains to prove that connectedness of $\Omega_\cat{A}(\mathbb{T}^2_1)$ implies connectedness of $\Omega_\cat{A}(\Sigma)$ for all surfaces $\Sigma$.  
	For this, we will make use of the Lego Teichmüller game from~\cite{bakifm}. For each surface $\Sigma$, one can define a groupoid $\catf{C}(\Sigma)$ of cut systems of $\Sigma$. Roughly, these are pair of pants decompositions; for details, we refer to \cite[Section~7.1-7.3]{bakifm}. 
	The morphisms are generated by the following two moves that can be applied to subsurfaces (the cuts are displayed in blue):
	\begin{equation}
	\begin{tikzpicture}[scale=0.5]
		\begin{pgfonlayer}{nodelayer}
			\node [style=none] (0) at (-10, 6) {};
			\node [style=none] (1) at (-10, 4) {};
			\node [style=none] (2) at (-3, 6) {};
			\node [style=none] (3) at (-3, 4) {};
			\node [style=none] (4) at (3, 6) {};
			\node [style=none] (5) at (3, 4) {};
			\node [style=none] (6) at (-6, 6) {};
			\node [style=none] (7) at (-6, 4) {};
			\node [style=none] (8) at (-10, -4) {};
			\node [style=none] (9) at (-10, -6) {};
			\node [style=none] (10) at (-3, -4) {};
			\node [style=none] (11) at (-3, -6) {};
			\node [style=none] (12) at (-6, -4) {};
			\node [style=none] (13) at (-6, -6) {};
			\node [style=none] (14) at (12, 5) {};
			\node [style=none] (15) at (6, 5.25) {};
			\node [style=none] (16) at (10, 5.25) {};
			\node [style=none] (17) at (7, 4.5) {};
			\node [style=none] (18) at (9, 4.5) {};
			\node [style=none] (19) at (8, 8) {};
			\node [style=none] (20) at (8, 2) {};
			\node [style=none] (21) at (8, 4.25) {};
			\node [style=none] (22) at (8, 2) {};
			\node [style=none] (23) at (3, -4) {};
			\node [style=none] (24) at (3, -6) {};
			\node [style=none] (25) at (12, -5) {};
			\node [style=none] (26) at (6, -4.75) {};
			\node [style=none] (27) at (10, -4.75) {};
			\node [style=none] (28) at (7, -5.5) {};
			\node [style=none] (29) at (9, -5.5) {};
			\node [style=none] (30) at (8, -2) {};
			\node [style=none] (31) at (8, -8) {};
			\node [style=none] (32) at (5, -5) {};
			\node [style=none] (33) at (11, -5) {};
			\node [style=none] (34) at (8, 1) {};
			\node [style=none] (35) at (8, -1) {};
			\node [style=none] (36) at (7, 0) {$\bar{\text{S}}$};
			\node [style=none] (37) at (-6, 3) {};
			\node [style=none] (38) at (-6, -3) {};
			\node [style=none] (39) at (-7, 0) {$\bar{\text{F}}$};
		\end{pgfonlayer}
		\begin{pgfonlayer}{edgelayer}
			\draw [bend right=90, looseness=1.25] (1.center) to (0.center);
			\draw [bend right=90, looseness=1.25] (0.center) to (1.center);
			\draw [bend right=90, looseness=1.25] (3.center) to (2.center);
			\draw (0.center) to (2.center);
			\draw (1.center) to (3.center);
			\draw [bend right=90, looseness=1.25] (5.center) to (4.center);
			\draw [bend right=90, looseness=1.25] (4.center) to (5.center);
			\draw [style=mydots, bend right=270, looseness=1.25] (3.center) to (2.center);
			\draw [bend right=90, looseness=1.25] (9.center) to (8.center);
			\draw [bend right=90, looseness=1.25] (8.center) to (9.center);
			\draw [bend right=90, looseness=1.25] (11.center) to (10.center);
			\draw (8.center) to (10.center);
			\draw (9.center) to (11.center);
			\draw [style=mydots, bend right=270, looseness=1.25] (11.center) to (10.center);
			\draw [style=blue, bend left=90, looseness=1.25] (6.center) to (7.center);
			\draw [style=bluedashed, bend right=90, looseness=1.25] (6.center) to (7.center);
			\draw [bend right=45, looseness=1.25] (15.center) to (16.center);
			\draw [bend left] (17.center) to (18.center);
			\draw [in=90, out=0] (19.center) to (14.center);
			\draw [in=0, out=-90] (14.center) to (20.center);
			\draw [in=0, out=180] (20.center) to (5.center);
			\draw [in=0, out=180] (19.center) to (4.center);
			\draw [style=blue, bend left=90, looseness=1.25] (21.center) to (22.center);
			\draw [style=bluedashed, bend right=90, looseness=1.25] (21.center) to (22.center);
			\draw [bend right=90, looseness=1.25] (24.center) to (23.center);
			\draw [bend right=90, looseness=1.25] (23.center) to (24.center);
			\draw [bend right=45, looseness=1.25] (26.center) to (27.center);
			\draw [bend left] (28.center) to (29.center);
			\draw [in=90, out=0] (30.center) to (25.center);
			\draw [in=0, out=-90] (25.center) to (31.center);
			\draw [in=0, out=180] (31.center) to (24.center);
			\draw [in=0, out=180] (30.center) to (23.center);
			\draw [style=end arrow] (34.center) to (35.center);
			\draw [style=end arrow] (37.center) to (38.center);
			\draw [style=blue, bend left=90, looseness=0.75] (32.center) to (33.center);
			\draw [style=blue, bend right=90] (32.center) to (33.center);
		\end{pgfonlayer}
	\end{tikzpicture}
	\label{moves}
	\end{equation}
	The $\bar{\text{F}}$-move deletes a cut in a cylindrical region, assuming of course that we are still left with a cut system afterwards. The $\bar{\text{S}}$-move replaces a cut on $\mathbb{T}^2_1$ with a transversal one.
	If we freely generate a groupoid with these moves, we obtain a groupoid $\widetilde{\catf{C}}(\Sigma)$.  The actual groupoid $\catf{C}(\Sigma)$ of cut system is then obtained by imposing further relations. These will however not be relevant for us. For us, it suffices to know that $\catf{C}(\Sigma)$ is connected and simply connected \cite[Theorem~7.9]{bakifm}. This implies that $\widetilde{\catf{C}}(\Sigma)$ is connected. 
	We will now construct a functor
\begin{align}T:	\widetilde{\catf{C}}(\Sigma) \to \Omega_\cat{A}(\Sigma) \ . \label{eqnthefunctorT}
\end{align}
For a cut system $C\in \widetilde{\catf{C}}(\Sigma)$, we take the genus zero surface $\Sigma_C$ obtained by cutting $\Sigma$ at $C$. We fill each component of $\Sigma_C$ with a handlebody; this choice is essentially unique because $\Hbdy \to \Surf$ is an equivalence in genus zero. 
This gives us a handlebody $H_C$ with $\partial H_C=\Sigma_C$.
The cut system describes a gluing operation from $\Sigma_C$ to $\Sigma$, i.e.\ $\Sigma = \cup \Sigma_C$, or more formally, $\Sigma = \Gamma_* \Sigma_C$ for a morphism $\Gamma$ in $\Graphs$ described by $C$.  The operadic composition of $\SurfA$ sends $\PhiA(H_C)$ to an object that we define to be $T(C)$. Thanks to 
 the proof of  Lemma~\ref{lemmasubopead}, we know \begin{align}
T(C)\cong \PhiA(\cup H_C) \ . \label{eqnTC}
\end{align}
Every handlebody with boundary $\Sigma$ can be obtained as $\cup H_C$ for some cut system $C$ for $\Sigma$. Thanks to~\ref{eqnTC}, this makes $T$, so far just as an assignment on object level, essentially surjective. 
 To an $\bar{\text{F}}$-move $C \to C'$, we associate an isomorphism $T(C)\cong T(C')$ that we obtain from~\eqref{eqnTC} and the fact that $\cup H_C=\cup H_{C'}$. 
 In order to define $T$ an  $\bar{\text{S}}$-move $C \to C'$, we can, thanks to the compatibility of the $\PhiA$-maps with gluing, see~\eqref{eqnTC}, assume that $\Sigma = \mathbb{T}_1^2$. In that case, we just assign to $C\to C'$ some morphism $T(C) \to T(C')$ existing by the assumption that $\Omega_\cat{A}(\mathbb{T}_1^2)$ is connected.
 In $\widetilde{\catf{C}}(\Sigma)$ we do not impose further relations on the moves. Therefore, the assignments provide us with the functor~\eqref{eqnthefunctorT}. Since $\widetilde{\catf{C}}(\Sigma)$ is connected and $T$, as already remarked, is essentially surjective, $\Omega_\cat{A}(\Sigma)$ must be connected.
	\end{proof}

\subsection{The notion of connectedness\label{secextensions}}
Having established Proposition~\ref{propconnected}, we can now finally define:

\begin{definition}\label{defconnected}
	We call
	a cyclic framed $E_2$-algebra $\cat{A}$ or, equivalently,
	a self-dual balanced braided 
	algebra in $\smc$
	\emph{connected}
	if any of the two equivalent conditions from Proposition~\ref{propconnected} hold.
\end{definition}

With the results and constructions of the previous two subsection, we can now
see easily that, for any connected cyclic framed $E_2$-algebra $\cat{A}$,
the map $\SurfA \to \Surf$ is an extension.
In fact, we will then see later in Theorem~\ref{thmuniversal} that these extensions are universal in a precise sense.

\begin{theorem}\label{thmdefsurfa2}
	For a cyclic framed $E_2$-algebra $\cat{A}$ in $\cat{S}$, the map
	$p_\cat{A}: \SurfA\to\Surf$
	is an extension if and only if $\cat{A}$ is  connected.
	In that case, the extension of $\Map(\Sigma)$ provided by $\SurfA$ takes the form of a
	short  exact sequence of groups	\begin{align}
		1 \to \Aut_\partial(\PhiA (H)  ) \to \MapA( \Sigma) \to \Map(\Sigma)\to 1 \quad  \text{with}\quad \MapA( \Sigma) = \pi_1(\SurfA(T),\Sigma) \   \label{eqnsesextension} ,
	\end{align}
	where
	$\Aut_\partial(\PhiA (H)  )$ denotes the group of
	 automorphisms of $\PhiA(H):\int_{\Sigma} \cat{A}\to \cat{A}^{\otimes \Legs(T)}$ as maps of $\int_{\partial \Sigma \times \opint }\cat{A}$-modules. 
\end{theorem}
\begin{proof}
		It follows from~\eqref{eqnsurfA}
	immediately that $p_\cat{A}$ is  
	an extension if and only if the groupoids $\Omega_\cat{A}(\Sigma)$ are connected. 
	In that case,	 the resulting extension of mapping class groups takes the form
	$	1 \to \Aut_\partial(\PhiA (H)  ) \to \MapA( \Sigma) \to \Map(\Sigma)\to 1$ (this just amounts to specializing the sequence~\eqref{eqnsesext} to the case at hand).

	\end{proof}

The following observation is a far-reaching generalization of \cite[Remark~2.3]{masbaumroberts}:
\begin{corollary}\label{corequivskein}
	Let $\cat{A}$ be a connected cyclic framed $E_2$-algebra.
	With the notation as in Theorem~\ref{thmdefsurfa2} and the skein algebra action 
	\begin{align} \Add:\SkAlg_\cat{A}(\Sigma)\to \End(\widehat{\cat{A}}(H))\cong \End(\FA(\Sigma))\end{align} from~\eqref{eqnaddmap},
	we have for
	 $f \in \Map(\Sigma)$ and $a \in \SkAlg_\cat{A}(\Sigma)$
	that 
	\begin{align}
		\Add(f.a)=\FA(f)\Add(a)\FA(f)^{-1} \ . 
	\end{align}
	Here $\FA(f)$ has to be understood as $\FA\left(\widetilde f\right)$ for any lift $\widetilde f$ of $f$ to $\MapA( \Sigma) $, but this is omitted from the notation because the right hand side does not depend on this lift.
	\end{corollary}

\begin{proof}
This follows from a calculation very similar to one from the
proof of Proposition~\ref{propenvfh2x}
	\end{proof}

In Definition~\ref{deflambdaequiv} we single out the admissible extensions of $\Surf$ that we will be interested in.
This leads to the obvious question whether the extension $\SurfA$ from Theorem~\ref{thmdefsurfa} fits into this framework. Let us first verify the insertion of vacua property:

	\begin{proposition}[Insertion of vacua]\label{propinsertionofvacualemma}
	Let $\cat{A}$ be a cyclic framed $E_2$-algebra in $\smc$, $\Sigma\in\Surf(T)$ for  $T\in \Graphs$ and
	$\varphi : \sqcup_J \mathbb{D}^2 \to \Sigma$
	an oriented embedding.
	Then
	\begin{align}
		\Omega_\cat{A} (  \Sigma \setminus \mathring{\operatorname{im}}\, \varphi )        \ra{ \id    \times \PhiA(\mathbb{B}^3_1)		  } \Omega_\cat{A} (  \Sigma \setminus \mathring{\operatorname{im}}\, \varphi )   \times \Omega_\cat{A}(\mathbb{D}^2)^{\times J} \ra{\text{gluing}}\Omega_\cat{A}(\Sigma)     \label{eqnfunctordisks}
	\end{align}
	is an equivalence. 
\end{proposition}

\begin{proof}
	Without loss of generality, we can assume that $T$ is a corolla. For any handlebody $H$ with boundary $\Sigma$, we can form the handlebody $H^\varphi$ which has additional embedded disks in its boundary coming from $\varphi$; we have used this notation on page~\pageref{Hvarphipage} already. 
	On the object level, \eqref{eqnfunctordisks} sends $\PhiA(H^\varphi) \mapsto \PhiA(H)$, thereby making the functor essentially surjective.
	In order to see that~\eqref{eqnfunctordisks} is also fully faithful, we observe that it is given by the restriction in domain and range of the equivalence
	\begin{align}
		\smc_\partial\left(   	\int_{\Sigma \setminus \mathring{\operatorname{im}}\,\varphi} \cat{A}, \cat{A}^{\otimes\Legs(T)}\otimes \cat{A}^{\otimes J} \right) &\simeq \smc_\partial \left(   	\int_{\Sigma \setminus \mathring{\operatorname{im}}\,\varphi} \cat{A} 
		\otimes_{ \left(  \int_{\mathbb{S}^1\times \opint } \cat{A} \right)^{\otimes J} }    \cat{A}^{\otimes J} , \cat{A}^{\otimes \Legs(T)}  \right) \\&\simeq \smc_\partial\left( \int_\Sigma \cat{A},\cat{A}^{\otimes \Legs(T)}\right)
	\end{align}
	which uses the pairing and excision.
\end{proof}

In combination with Theorem~\ref{thmdefsurfa},
 Proposition~\ref{propinsertionofvacualemma} gives us:

\begin{theorem}\label{thmtrivhbdy}
	Let $\cat{A}$ be a connected cyclic framed $E_2$-algebra.
	Then	$p_\cat{A}: \SurfA \to \Surf$ is 
	an extension of $\Surf$ relative to genus zero
	together with the needed section being the map $s_\cat{A}:\Hbdy\to\SurfA$ from Theorem~\ref{thmdefsurfa}. It additionally admits insertions of vacua
	in the sense of Definition~\ref{deflambdaequiv}, i.e.\
	we have $\SurfA \in \ExtSurf$.
\end{theorem}

	\section{The construction and classification of modular functors\label{secclassmf}}

The goal of this section is to give a general construction procedure for modular functors and to ultimately use this construction to classify modular functors. 

\subsection{The construction of modular functors\label{secmodularfunctorsconstruction}}
As a first result, we will construct,
for any cyclic framed $E_2$-algebra $\cat{A}$, a modular $\SurfA$-algebra structure on $\cat{A}$.
This modular $\SurfA$-algebra will be a modular functor if and only if $\cat{A}$ is connected.

\begin{definition}\label{defthefunctorsF}
For any cyclic framed $E_2$-algebra $\cat{A}$, a corolla $T\in\Graphs$ and $\Sigma \in \Surf(T)$, the distinguished object $\cat{O}_\Sigma^\cat{A}:I\to \int_\Sigma \cat{A}$ gives us by precomposition a functor
\begin{align}
	\mathfrak{F}_\cat{A}^{T,\Sigma} : \Omega_\cat{A}(\Sigma)     \ra{\text{inclusion}}    \smc_\partial\left(\int_\Sigma\cat{A},\cat{A}^{\otimes\Legs(T)}\right) \ra{-\circ \cat{O}_\Sigma^\cat{A}}  
	\smc\left(I,\cat{A}^{\otimes\Legs(T)}\right)\ . \label{eqnthefunctorsF}
	\end{align}
This is extended to the case of non-connected $T$.
\end{definition}

\begin{theorem}[Construction of modular functors]\label{thmfa2}
	For any cyclic framed $E_2$-algebra $\cat{A}$, 
	the functors from Definition~\ref{defthefunctorsF} yield the structure of a modular $\SurfA$-algebra $\FA$ whose restriction to $\Hbdy$ is canonically identified with the ansular functor $\widehat{\cat{A}}$ of $\cat{A}$.
	This modular $\SurfA$-algebra $\cat{A}$ is a modular functor
	if and only if
	$\cat{A}$ is connected.
\end{theorem}

\begin{proof}
	Let $T \in \Graphs$ be a corolla. Then for any diffeomorphism $f:\Sigma \to \Sigma'$,
	 the triangle 
	\begin{equation}\label{gdconeqn}
		\begin{tikzcd}
			\Omega_\cat{A}(\Sigma') \ar[rrd,"\text{$\mathfrak{F}_\cat{A}^{T,\Sigma'}$}"] \ar[dd,swap,"- \circ f_*"] \\ 
			&& \smc(I,\cat{A}^{\otimes \Legs(T)}) \\
			\Omega_\cat{A}(\Sigma) \ar[rru,swap,"\text{$\mathfrak{F}_\cat{A}^{T,\Sigma}$}"]
		\end{tikzcd}
	\end{equation} commutes by a canonical natural isomorphism coming from the homotopy fixed point structure $f_* \cat{O}_\Sigma^\cat{A} \cong \cat{O}_{\Sigma'}^\cat{A}$
 of the distinguished object $\cat{O}_\Sigma^\cat{A}$ under diffeomorphisms
 (this homotopy fixed point structure was recalled in Section~\ref{secprefh}).
	This entails that the functors $\mathfrak{F}_\cat{A}^{T,\Sigma}$ descend to the Grothendieck construction $\cat{G}_\cat{A}^\Omega(T) = \Gr \left(   \SURF(T)^\op \ra{\Omega_\cat{A}} \Grpd  \right)$ and yield functors $\mathfrak{F}_\cat{A} : \SurfA(T)\to  \smc(I,\cat{A}^{\otimes \Legs(T)})$
	because $\SurfA(T)$ is the fundamental groupoid of $\cat{G}_\cat{A}^\Omega(T)$, see \eqref{defsurfa}, and because $\smc(I,\cat{A}^{\otimes \Legs(T)})$ is just a 1-category.
 This extends in a straightforward way to the case of non-connected $T$. Now the functors $\mathfrak{F}_\cat{A}$ endow $\cat{A}$ with the structure of a modular $\SurfA$-algebra. 
	The restriction of $\mathfrak{F}_\cat{A}$ 
	along the map $\Hbdy \to \SurfA$ from Theorem~\ref{thmdefsurfa}
	sends a handlebody $H$ to $\PhiA(H)\cat{O}_{\partial H}^\cat{A}$, which can be canonically identified with  $\widehat{\cat{A}}(H)$ by Theorem~\ref{thmfhenv}. In other words, it agrees with the ansular functor associated to $\cat{A}$.
	By Theorem~\ref{thmdefsurfa} and~\ref{thmtrivhbdy} $\SurfA$ is an extension if and only if $\cat{A}$ is connected. Therefore, $(\SurfA,\mathfrak{F}_\cat{A})$ is 
	 a modular functor if and only if $\cat{A}$ is connected.
	\end{proof}

\begin{remark}
	The construction of Theorem~\ref{thmfa2} is suitably functorial in maps of cyclic framed $E_2$-algebras which by \cite[Proposition~2.18]{cyclic} are always invertible. More precisely, it is straightforward to verify the following:
	Let $\varphi : \cat{A}\ra{\simeq}\cat{B}$ be an equivalence of cyclic framed $E_2$-algebras.
	Then $\varphi$ gives rise to a modular operad $\SurfAB$ that comes with equivalences
	$\varepsilon_\cat{A}:\SurfAB\ra{\simeq}\SurfA$ and $\varepsilon_\cat{A}:\SurfAB \ra{\simeq}\SurfB$
	such that the map $\varphi$ induces an equivalence $\varepsilon_\cat{A}^*\mathfrak{F}_\cat{A}\ra{\simeq}\varepsilon_\cat{B}^* \mathfrak{F}_\cat{B}$
	of modular $\SurfAB$-algebras.
\end{remark}

\subsection{Universality of $\SurfA$ and $\mathfrak{F}_\cat{A}$}
We will ultimately use the construction $\cat{A}\mapsto \mathfrak{F}_\cat{A}$ from the previous subsection to classify modular functors.
To this end, the next result will be crucial: We will prove that, if $\cat{A}$ is connected, then an extension of $\cat{A}$ to a modular functor living over an extension $\cat{Q}$ of $\Surf$ amounts exactly to a map $\cat{Q}\to\SurfA$ of extensions.
In other words, $\SurfA$ is a `classifying space' for modular functors extending $\cat{A}$.

\begin{theorem}[Universality of $\SurfA$ and $\mathfrak{F}_\cat{A}$]\label{thmuniversal}
	For any connected cyclic framed $E_2$-algebra $\cat{A}$,
the $\SurfA$-algebra $\mathfrak{F}_\cat{A}$ from Theorem~\ref{thmfa2} is universal in the sense that, for $\cat{Q} \in \ExtSurf$, the map
	\begin{align}
		\Map_{\ExtSurf}(\cat{Q},\SurfA) \ra{\simeq} \Ext(\cat{A};\cat{Q}) \ , \quad \psi \mapsto \psi^* \mathfrak{F}_\cat{A} \ \label{eqnuniversal}
		\end{align}
	is an equivalence of groupoids.
	\end{theorem}

The fact that the morphism category in $\ExtSurf$
is indeed a groupoid is a consequence of Lemma~\ref{lemma2morphinv}.

\begin{proof} 
	The main part of the proof will construct
	for each $\cat{B}\in \Ext(\cat{A};\cat{Q})$ a map
	$\psi^\cat{B}:\cat{Q}\to\SurfA$ of extensions over $\Surf$. This will be accomplished in the  steps~\ref{proofpsistep1}-\ref{steppsiconstruct}.
By sending $\cat{B}\in \Ext(\cat{A};\cat{Q})$ to
$\psi^\cat{B}:\cat{Q}\to\SurfA$  we will be able to describe  a weak inverse to~\eqref{eqnuniversal} directly in step~\ref{stepdescribeinverse}. 
	
	\begin{pnum}
		\item\label{proofpsistep1}
	Let $\cat{B}\in \Ext(\cat{A};\cat{Q})$, moreover $T\in\Graphs$ a corolla and $\Sigma \in \Surf(T)$.
	In this first step, it is our goal to construct natural maps
	\begin{align}
		\Lambda_\cat{B}^-  :          \cat{Q}_\Sigma \to \cat{S}_\partial \left(   \int_\Sigma \cat{A},\cat{A}^{\otimes \Legs(T)}\right) \ , \quad q \mapsto \Lambda_\cat{B}^q \ . 
		\end{align}
	As usual $\cat{Q}_\Sigma$ is the homotopy fiber of the extension $\cat{Q}\to\Surf$ over $\Sigma$.
	
For any oriented embedding $\varphi : (\mathbb{D}^2)^{\sqcup J} \to \Sigma$, we have the equivalence
\begin{align}
		\lambda_\varphi : 	\cat{Q}_{\Sigma \setminus \mathring{\operatorname{im}}\, \varphi} \ra{\simeq} \cat{Q}_  \Sigma    \label{eqnlambdaequiv}
	\end{align}
from Definition~\ref{deflambdaequiv} since $\cat{Q}$ admits insertions of vacua by assumption. 
Now for $q \in \cat{Q}_  \Sigma$ and $r=(o,\alpha) \in \lambda_\varphi/q$, i.e.\ an object
$o\in \cat{Q}_{\Sigma\setminus \mathring{\operatorname{im}}\, \varphi}$ and a morphism
$\lambda_\varphi o \ra{\alpha} q$, the algebra $\cat{B}$ gives us a map
\begin{align}
 \cat{B}^{\otimes J} \ra{\cat{B}_{o}} \cat{B}^{\otimes \Legs(T)} \ . 
	\end{align}
These maps combine into a map
\begin{align}
\cat{B}_q^\varphi :	\hocolimsub{q' \in \lambda_\varphi / q } \cat{B}^{\otimes J} 
\to
  \cat{B}^{\otimes \Legs(T)}  \ . 
\end{align}
We will now show that
$\cat{B}_q^\varphi$ 
 is natural in $\varphi$ and hence gives us a map
\begin{align}
\label{eqnPhibq}	\Lambda_\cat{B}^q : \hocolimsub{\varphi} \hocolimsub{r \in \lambda_\varphi / q } \cat{B}^{\otimes J}\to \cat{B}^{\otimes \Legs(T)}  \ . 
	\end{align}
	To see this, we write
	 $J=\{1,\dots,n\}$
	and $J'=\{1,\dots,m\}$ and pick in addition to $\varphi$ a framed $E_2$-operation $\varphi' : \sqcup_{J'} \mathbb{D}^2 = (\mathbb{D}^2)^{\sqcup m} \to \mathbb{D}^2$. We will see $\varphi'$ not only as a $\framed$-operation, but also as genus zero surface.
	The functor
	\begin{align}
	c_i(\varphi,\varphi') : \cat{Q}_{\Sigma \setminus \mathring{\operatorname{im}}\,\varphi} \ra{\text{pointing of $\cat{Q}_{\varphi'}$}} \cat{Q}_{\Sigma \setminus \mathring{\operatorname{im}}\, \varphi}  \times \cat{Q}_{\varphi'} \ra{\text{gluing}} \cat{Q}_{\Sigma \setminus \mathring{\operatorname{im}}\,(\varphi\circ_i \varphi')}
	\end{align}
	makes the triangle
		\begin{equation}\label{propenvfh2eXY}
	\begin{tikzcd}
	& \cat{Q}_{\Sigma \setminus \mathring{\operatorname{im}}\,\varphi}	 \ar[rrrd,"\lambda_\varphi"]  \ar[dd,swap,"\text{$c_i(\varphi,\varphi')$}"] \\ &&&& \cat{Q}_\Sigma \ ,   \\ 
	&\cat{Q}_{\Sigma \setminus \mathring{\operatorname{im}}\,(\varphi\circ_i \varphi')} \ar[rrru,"\lambda_{\varphi \circ_i \varphi'}", swap] 
	\end{tikzcd}
	\end{equation}
	commute up to a canonical natural isomorphism. Since $\lambda_\varphi$ and $\lambda_{\varphi \circ_i \varphi'}$ are equivalences by assumption, so is 
	$c_i(\varphi,\varphi')$. As a result, we may identify \begin{align} \lambda_{\varphi\circ_i \varphi'} / q\simeq \lambda_\varphi / q \ . \label{eqnequivslices} \end{align}
	Consider now the following triangle:
	\begin{equation}\label{propenvfh2eX}
	\begin{tikzcd}
	&  \hocolimsub{r' \in \lambda_{\varphi\circ_i \varphi'} / q }	\cat{B}^{\otimes(j-1)} \otimes \cat{B}^{\otimes m} \otimes \cat{B}^{\otimes (n-j)}	 \ar[rrrd,"\cat{B}_q^{\varphi \circ_i \varphi'}"] \ar[dd,"(*)  ",swap]  \\ &&&& \cat{B}^{\otimes \Legs(T )} \ ,   \\ 
	&	\hocolimsub{r \in \lambda_\varphi / q }	\cat{B} 	^{\otimes n} \ar[rrru,"\cat{B}_q^\varphi", swap] 
	\end{tikzcd}
	\end{equation}
	The map $(*)$ is induced by $\cat{B}_{ s(\varphi')}:\cat{B}^{\otimes m }\to \cat{B}$
	(we are using here the section $s: \Hbdy \to \cat{Q}$ associated to $\cat{Q}$)
	 and~\eqref{eqnequivslices}.
	Since $\cat{B}$ is a modular algebra, this triangle commutes up to a canonical isomorphism.
	This tells us that we indeed obtain the map~\eqref{eqnPhibq}. 

The key point is that the slices $\lambda_\varphi/ -$ are contractible (see~\eqref{eqnlambdaequiv}). For this reason,
the double homotopy colimit on the left hand side of~\eqref{eqnPhibq}
 models the factorization homology of the restriction of $\cat{B}$ to
$\framed$. 
But as part the data, we have an equivalence
of cyclic framed $E_2$-algebras 
between this restriction of $\cat{B}$ and  $\cat{A}$. 
Therefore, we may identify the left hand side of~\eqref{eqnPhibq} canonically with $\int_\Sigma \cat{A}$.
In other words, we may see $\Lambda_\cat{B}^q$ as a map 
\begin{align}
\Lambda_\cat{B}^q : \int_\Sigma \cat{A} \to \cat{A}^{\otimes \Legs(T)}  \ . 
	\end{align}
By construction this comes with the structure of an $\int_{\partial \Sigma \times \opint } \cat{A}$-module map
(this is essentially the same argument as in the proof of Proposition~\ref{propmodulemap}). 
Moreover, the assignment $q \mapsto \Lambda_\cat{B}^q$ naturally extends to a functor out of $	\cat{Q}_\Sigma$, and it sends the morphisms in $\cat{Q}_\Sigma$ to morphisms of $\int_{\partial \Sigma \times \opint } \cat{A}$-module maps. Therefore, we obtain natural functors
\begin{align}
		\Lambda_\cat{B}^-  :          \cat{Q}_\Sigma \to \cat{S}_\partial \left(   \int_\Sigma \cat{A},\cat{A}^{\otimes \Legs(T)}\right) \ , 
		\end{align}
where $\smc_\partial \left( \int_\Sigma \cat{A}, \cat{A}^{\otimes \Legs(T)}\right)$ is the groupoid of $\int_{\partial \Sigma \times \opint } \cat{A}$-module maps $\int_\Sigma \cat{A}\to \cat{A}^{\otimes \Legs(T)}$ with module isomorphisms as morphisms.
Then
	\begin{align}
		\cat{Q}_\Sigma \ra{ \Lambda_\cat{B}^-  }    \smc_\partial\left( \int_\Sigma \cat{A}, \cat{A}^{\otimes \Legs(T)}\right) \ra{-\circ \cat{O}_\Sigma^\cat{A}} \smc(I,\cat{A}^{\otimes \Legs(T)})\label{eqnidentifyB}
		\end{align}
	can be canonically identified with $	\cat{Q}_\Sigma\ra{\cat{B}} \smc(I,\cat{A}^{\otimes \Legs(T)})$ under the identification of $\cat{A}$ and $\cat{B}$ as cyclic algebras (this is the same argument as in the proof of Theorem~\ref{thmfhenv}).
	
	\item 
	Next consider the following
	diagram (the dashed arrow is to be ignored for the moment):
		\begin{equation}\label{eqndiagramuniversal}
		\begin{tikzcd}
		\partial^{-1}(\Sigma) \ar[dd,swap,"s"] \ar[rrr,"s_\cat{A}"]	&&& \Omega_\cat{A}(\Sigma) \ar[rrrdd," \mathfrak{F}_\cat{A}^{T,\Sigma}  "] \ar[dd," \iota  "]  \\ \\ 
			\cat{Q}_\Sigma \ar[rrr,"    \Lambda_\cat{B}^-      "]  \ar[rrruu, dashed,"   \psi_\Sigma^\cat{B}     "] \ar[rrrrrrdd, swap,"   \cat{B}      "]  && &     \smc_\partial\left( \int_\Sigma \cat{A}, \cat{A}^{\otimes \Legs(T)}\right) \ar[rrr, swap,"  -\circ \cat{O}_\Sigma^\cat{A}    "] &&&   \smc(I,\cat{A}^{\otimes \Legs(T)}) \ar[dd,"\simeq"] \\ \\  &&& &&&\smc(I,\cat{B}^{\otimes \Legs(T)}) 
		\end{tikzcd}
	\end{equation}
Here $\iota$ is the full subgroupoid inclusion.
The lower surface commutes by a canonical isomorphism  because the composition~\eqref{eqnidentifyB}, as just discussed, can be identified with $\cat{B}$.
The right triangle commutes strictly by Definition~\ref{defthefunctorsF}.
By construction the square on the left commutes up to a canonical isomorphism (this is the identification of $\cat{B}$ with $\widehat{\cat{A}}$ over $\Hbdy$ that is part of data). 
The fact that by connectedness of $\cat{Q}_\Sigma$ the functor $s: \partial^{-1}(\Sigma)\to\cat{Q}_\Sigma$ is essentially surjective implies that $\Lambda_\cat{B}^-$ must factor through $\Omega_\cat{A}(\Sigma)$. In other words, there is a unique functor $\psi_\Sigma^\cat{B} : \cat{Q}_\Sigma\to \Omega_\cat{A}(\Sigma)$ with $\iota \psi_\Sigma^\cat{B} = \Lambda_\cat{B}^-$. 
Now $\iota \psi_\Sigma^\cat{B} s = \Lambda_\cat{B}^- s \cong \iota s_\cat{A}$ by a canonical isomorphism. 
Since $\iota$ is a full subgroupoid inclusion, 
this gives us a canonical isomorphism $\psi_\Sigma^\cat{B} s\cong s_\cat{A}$. 
We have now established that the entire diagram~\eqref{eqndiagramuniversal} commutes up to a canonical isomorphism.

\item
The map $\psi_\Sigma^\cat{B}$, being natural in $\Sigma$, gives us a map $\psi^\cat{B}:\cat{Q}\to\SurfA$ of modular operads, in fact, of modular operads over $\Surf$ (because it is constructed fiberwise anyway). Over $\Hbdy$, the map comes with a section thanks to $\psi^\cat{B} s\cong s_\cat{A}$. \label{steppsiconstruct}
\item With all these preparations, we can describe the desired inverse
$\Ext(\cat{A};\cat{Q}) \to \Map_{\ExtSurf}(\cat{Q},\SurfA)$ to~\eqref{eqnuniversal}.
It sends $\cat{B} \in  \Ext(\cat{A};\cat{Q})$ to the map $\psi^\cat{B} : \cat{Q} \to \SurfA$
from step~\ref{steppsiconstruct}.
This gives us indeed an inverse: For any map $\theta :  \cat{Q} \to \SurfA$ of extensions,
we find $\Lambda_{\theta^* \mathfrak{F}_\cat{A}}^- \simeq \Lambda_{\mathfrak{F}_\cat{A}}^{\theta(-)}$, and hence $\psi^{\theta^* \mathfrak{F}_\cat{A}}\simeq\psi^{\mathfrak{F}_\cat{A}} \circ \theta=\theta$ (this uses that $\psi^{\mathfrak{F}_\cat{A}}$ is the identity).
Similarly, for $\cat{B}\in  \Ext(\cat{A};\cat{Q})$, we find $(\psi^\cat{B})^* \mathfrak{F}_\cat{A}\simeq \cat{B}$ by the commutativity of the lower surface in the above diagram. 
\label{stepdescribeinverse}
\end{pnum}

	\end{proof}

\begin{lemma}\label{lemmaconnected1}
	Let $\cat{A}$ be a cyclic framed $E_2$-algebra. If the bicategory $\Ext(\cat{A})$ of extensions of $\cat{A}$ is not empty, then $\cat{A}$ is connected.
	\end{lemma}

	\begin{proof}
		Let us assume that we have some $\cat{B}\in \Ext(\cat{A};\cat{Q})$
		for some $\cat{Q} \in \ExtSurf$.
		Even though $\cat{A}$ is not connected, the proof 
		of Theorem~\ref{thmuniversal} still gives us a map $\psi^\cat{B} : \cat{Q} \to \SurfA$ of modular operads over $\Surf$
(the construction of this map in the proof did not need anywhere the $\Omega_\cat{A}$-groupoids are connected; without the assumption, we only have the problem that the map $\psi^\cat{B} : \cat{Q} \to \SurfA$ is not a map in $\ExtSurf$, simply because $\SurfA$ might have non-connected fibers and hence might
not qualify as an extension). 
We can verify $\psi^\cat{B} s \simeq s_\cat{A}$ for the canonical maps $s:\Hbdy \to \cat{Q}$ and $s_\cat{A}:\Hbdy \to \Surf$. 
This entails that the essentially surjective functor $\partial^{-1}(\Sigma) \ra{s_\cat{A}}\Omega_\cat{A}(\Sigma)$ factors as $\partial^{-1}(\Sigma) \ra{s} \cat{Q}_\Sigma \ra{\psi^\cat{B}} \Omega_\cat{A}(\Sigma)$. But  this is only possible if $\psi^\cat{B} : \cat{Q}_\Sigma \to  \Omega_\cat{A}(\Sigma)$ is essentially surjective. Since $\cat{Q}_\Sigma$ is connected by assumption, so is $\Omega_\cat{A}(\Sigma)$.
\end{proof}

\subsection{Uniqueness of extensions}
For a cyclic framed $E_2$-algebra $\cat{A}$, we defined the bicategory of extensions of $\cat{A}$ to a modular functor in Definition~\ref{defallext}.
The following result is mostly an application of the universality of $\SurfA$:

\begin{theorem}[Uniqueness of extensions]\label{uniquenessthm}
	For any cyclic framed $E_2$-algebra $\cat{A}$, the  realization $|B	\Ext(\cat{A})|$
	of the Duskin
	nerve of the bicategory $\Ext(\cat{A})$
	is empty or contractible. In other words,
	if there is an extension of $\cat{A}$ to a modular functor, this extension is essentially unique, i.e.\ unique up to a contractible choice.
\end{theorem}

In the case of
cyclic framed $E_2$-algebras whose underlying category is finitely semisimple and equipped with an involution, 
such a uniqueness result has been established by Andersen-Ueno~\cite{andersenueno} (with a different definition of modular functor). 
The above result applies to cyclic framed $E_2$-algebras regardless of finiteness and semisimplicity assumptions and makes a statement about the \emph{space of extensions} rather than extensions up to equivalence.

\begin{proof}[\slshape Proof of Theorem~\ref{uniquenessthm}] \label{pageproofunique}
	We need to show that the bicategory $\Ext(\cat{A})$
	 is either empty or has a contractible nerve.
	 We will do this by proving that $\Ext(\cat{A})$ is contractible under the assumption that it is not empty. 
	 Under this assumption, it is automatic that $\cat{A}$ is connected by Lemma~\ref{lemmaconnected1}.
	 
	 Let us now prove the contractibility of $\Ext(\cat{A})$ in the case that $\cat{A}$ is connected:
By Definition~\ref{defallext}
 $\Ext(\cat{A})$ is the Grothendieck construction of the bicategory-valued functor $\Ext(\cat{A};-):\ExtSurf^\op \to \catf{BiCat}$
(the bicategorical version of the Grothendieck construction can be defined in complete analogy to the categorical one).
By Theorem~\ref{thmuniversal} we have equivalences
$\Ext(\cat{A};\cat{Q})\simeq \Map_{\ExtSurf}(\cat{Q},\SurfA)$ natural in $\cat{Q}$, which implies the equivalence
	\begin{align}
\label{eqnnerveext}		\Ext(\cat{A}) \simeq  \Gr \left( \cat{Q} \mapsto \Map_{\ExtSurf}(\cat{Q},\SurfA) \right)  
\end{align}
of bicategories.
But the Grothendieck construction
$\Gr \left( \cat{Q} \mapsto \Map_{\ExtSurf}(\cat{Q},\SurfA) \right)$ is, after spelling out the definition, the slice bicategory $\ExtSurf / \SurfA$. 
The identity of $\SurfA$ is a terminal object in this slice bicategory.
This proves that  the nerve of $\Gr \left( \cat{Q} \mapsto \Map_{\ExtSurf}(\cat{Q},\SurfA) \right)$ is contractible.
By~\eqref{eqnnerveext} and homotopy invariance of the nerve under equivalences of bicategories we find that $\Ext(\cat{A})$ is contractible as well.
	\end{proof}

Theorem~\ref{uniquenessthm} proves under very general conditions that a modular functor is determined up to a contractible choice
by its genus zero data. But for this it is important to understand `genus zero data' precisely
 as the data of a \emph{cyclic} framed $E_2$-algebra. 
Prescribing just the non-cyclic genus zero data, i.e.\ a balanced braided monoidal structure, leads already to a different situation about which a different statement can be made:
Let $\cat{A}$ be a framed $E_2$-algebra in $\Rex$ with underlying finite category (that is a balanced braided category).
We denote by $Z_2^\catf{bal}(\cat{A})\subset Z_2(\cat{A})$ the subcategory of the Müger center
spanned by the objects with trivial balancing
and denote by
$\catf{PIC}(Z_2^\catf{bal}(\cat{A}))$ the
Picard groupoid  of
$Z_2^\catf{bal}(\cat{A})$, i.e.\
the groupoid 
of invertible objects in $Z_2^\catf{bal}(\cat{A})$. Then the following holds:

\begin{corollary}
	For any non-cyclic framed $E_2$-algebra $\cat{A}$ in $\Rex$ with underlying finite category,
	the space 
	$|B\Ext(\cat{A})|$ 
	of extensions of $\cat{A}$ to a modular 
functor (with maps between such extensions being relative to $\cat{A}$) is either empty or 
\begin{align}
|B\Ext(\cat{A})|\simeq	  | B\catf{PIC}(Z_2^\catf{bal}(\cat{A}))| \ .   
	\end{align}
If the braiding of $\cat{A}$ is non-degenerate and the unit of $\cat{A}$ is simple,
then the space on the right is the classifying space $B\k^\times$ of the group of units of the ground field $\k$.
\end{corollary}

\begin{proof}
	If $|B\Ext(\cat{A})|$ is not empty, then a cyclic structure exists on $\cat{A}$ and an extension to a modular functor is unique up to a contractible choice once a cyclic structure is chosen; this follows from
	Theorem~\ref{uniquenessthm}.
	Therefore, the remaining choice is exactly the ambiguity in choosing a cyclic structure.
	In $\Lexf$ and, equivalently, $\Rexf$ this ambiguity is described in \cite[Theorem~4.2]{mwcenter} through
	$\catf{PIC}(Z_2^\catf{bal}(\cat{A}))$ as groupoid or $| B\catf{PIC}(Z_2^\catf{bal}(\cat{A}))|$ as space. 
	\end{proof}

\subsection{The classification of modular functors}
In this subsection, we state and prove the main result of this paper, namely  an explicit description of the moduli space of modular functors.
This will however be relatively straightforward because most of the work was already done in the previous subsections.

\begin{theorem}[Classification of modular functors]\label{thmproofclassmf}
	The moduli space $\moduli$ of
	modular functors with values in $\smc$ is equivalent the 2-groupoid of connected cyclic framed $E_2$-algebras $\cat{A}$ in $\smc$. 
	The equivalence is afforded by restriction to genus zero. An inverse sends $\cat{A}$ to $\mathfrak{F}_\cat{A}$.
\end{theorem}

If we take into account that cyclic framed $E_2$-algebras are self-dual balanced braided algebras by the main result of \cite{mwansular}
(this was recalled as Theorem~\ref{thmmodext} in this article),  $\moduli$ 
is also equivalent to the 2-groupoid of connected self-dual balanced braided algebras in $\smc$.

The proof of Theorem~\ref{thmproofclassmf} will use a bicategorical version of  
Thomason's Theorem \cite[Theorem~1.2]{thomason}
that can be obtained by a straightforward modification of the version of this result given in~\cite{ccg,ceg}: 
Recall that for 
a functor $F:\cat{C}\to \catf{BiCat}$ 
from a 2-groupoid $\cat{C}$,
seen as tricategory, to the tricategory of bicategories
(by default we understand this to be a functor is the weak sense), 
the Grothendieck construction $\Gr F$ is the bicategory formed by pairs $(c,x)$, where $c\in \cat{C}$ and $x\in F(c)$, in analogy to the usual Grothendieck construction.
 Thomason's Theorem gives us a canonical homotopy equivalence
\begin{align}
B \Gr F \simeq \hocolimsub{c\in \cat{C}} BF(c) \label{thomasontheqnbicat}
\end{align}
of simplicial sets (but we can of course apply the geometric realization
$|-|$ to get a homotopy equivalence of spaces).

\begin{proof}[\slshape Proof of Theorem~\ref{thmproofclassmf}]
With Thomason's Theorem~\eqref{thomasontheqnbicat}, we can prove
the desired statement as follows: 
	We consider the genus zero restriction map
	$
	R:\MF  \to \catf{CycAlg}(\framed)  
	$
	and first observe that by Lemma~\ref{lemmaconnected1} it actually factors through
	the full sub-2-groupoid $\catf{CycAlg}^\Omega(\framed)\subset \catf{CycAlg}(\framed)$ of connected cyclic framed $E_2$-algebras.
	 Therefore, the restriction map $R:\MF  \to \catf{CycAlg}(\framed)$ restricts in range to a map
		\begin{align}
	\widetilde R:\MF  \to \catf{CycAlg}^\Omega(\framed)  \ .     \label{eqnwidetildeR}
	\end{align}
	The homotopy fiber of~\eqref{eqnwidetildeR} is $\Ext(\cat{A})$ by definition.
	In other words, we can describe $\MF$ as a Grothendieck construction
	$\MF= \Gr \left(\cat{A}\mapsto \Ext(\cat{A})\right)$; in fact, that is just a different way to express the definition of $\MF$.
	Now Thomason's Theorem~\eqref{thomasontheqnbicat}
	 gives us
	\begin{align}
	B \MF  \simeq
	\hocolimsub{\cat{A} \in \catf{CycAlg}^\Omega(\framed) }
	 B\Ext(\cat{A}) \ . 
	\end{align}
	But $B\Ext(\cat{A})$, for an arbitrary cyclic framed $E_2$-algebra is always empty or contractible
	 by Theorem~\ref{uniquenessthm}, and for those algebras in
	 $\catf{CycAlg}^\Omega(\framed)$,
	 the latter is the case by Lemma~\ref{lemmaconnected1}.
	 This implies that~\eqref{eqnwidetildeR} induces an equivalence after taking the nerve.
	 The statement about the inverse equivalence follows from the description~\eqref{eqnnerveext} of $\Ext(\cat{A})$. 
	\end{proof}
	
Since by Theorem~\ref{thmproofclassmf}
$\moduli$ forms a \emph{full} sub-2-groupoid of cyclic framed $E_2$-algebras, we immediately obtain:

\begin{corollary}\label{corautos}
	Let $\cat{A}$ be a cyclic framed $E_2$-algebra.
	Then the 2-group of automorphisms of $\cat{A}$ as cyclic framed $E_2$-algebra is equivalent to the 2-group of automorphisms of any modular functor extending $\cat{A}$.
\end{corollary}

\section{Sufficient conditions for connectedness\label{seccof}}
The connectedness condition that is needed for a cyclic framed $E_2$-algebra to uniquely extend to a modular functor is relatively abstract.
Even though one of the different equivalent characterizations of connectedness just boils down to a genus one condition, it will still be involved to check the condition in practice. In this section, we provide concrete algebraic conditions that are sufficient for connectedness. These will be used in the next section to discuss classes of examples of modular functors.

\subsection{Triviality of the mapping class group action on factorization homology}
We start with the following observation
that will still not produce a concrete criterion, but instead a valuable tool:

\begin{lemma}\label{lemmatrivcon}
	If for a cyclic framed $E_2$-algebra $\cat{A}$ and every surface $\Sigma$, 
	the 1-morphism $f_* : \int_\Sigma \cat{A}\to\int_{\Sigma}\cat{A}$ is trivializable for all $f\in\Map(\Sigma)$, i.e.\ isomorphic as a $\int_{\partial \Sigma \times [0,1]} \cat{A}$-module map to the identity, then $\cat{A}$ is connected, and the resulting $\Map(\Sigma)$-action on $\Aut_\partial(\PhiA (H)  )$ in~\eqref{eqnsesextension} is by inner automorphisms.
	\end{lemma}

\begin{proof}
	For handlebodies $H$ and $H'$ with $\partial H = \partial H'=\Sigma$, pick an isomorphism $g :H\to H'$.
	Then $\PhiA(H')f_*\cong \PhiA(H)$ with $f:=\partial g$ by Proposition~\ref{propenvfh2x}.
	Together with Theorem~\ref{thmdefsurfa2}, this implies the assertion.
	\end{proof}
	
	\begin{remark}
The Lemma assumes that we can trivialize each $f_*$ separately. It is not implied that we can trivialize \emph{coherently}.
	\end{remark}
	\begin{remark}
		The condition that the action is inner implies that the extension we obtain in that case, contains a central extension of $\Map(\Sigma)$ by the center of $\Aut_{\partial}(\PhiA(H))$.
	\end{remark}

	The triviality of the mapping class group action on factorization homology can be checked locally:
	
	\begin{lemma}\label{lemmaactfh}
		Let $\cat{A}$ be a cyclic framed $E_2$-algebra in $\smc$.
		The action $f_* : \int_{\Sigma }\cat{A}\to\int_{\Sigma }\cat{A}$
		of every diffeomorphism $f:\Sigma\to\Sigma$
		of every surface $\Sigma$ is isomorphic, as 
		$\int_{\partial \Sigma \times [0,1]}\cat{A}$-module
		map, to the identity
		if and only if
		this is the case for the action 
		$d_* : \int_{\mathbb{S}^1 \times\opint } \cat{A} \to \int_{\mathbb{S}^1 \times\opint } \cat{A}$
		of the Dehn twist 
		on the factorization homology of the cylinder.
	\end{lemma}
	
	\begin{proof}
		First recall from Remark~\ref{remactiononfh}
		that the isomorphism class
		of $f_* : \int_{\Sigma }\cat{A}\to\int_{\Sigma }\cat{A}$,
		as  
		$\int_{\partial \Sigma \times [0,1]}\cat{A}$-module
		map, depends only on the mapping class of $f$.
		Now the statement follows from the fact
		that the pure mapping class group of any surface is generated by Dehn twists (see~\cite[Theorem~4.11]{farbmargalit} for a textbook reference)
		and the excision property of factorization homology.
	\end{proof}

	\subsection{Cofactorizability}
	Using the tools from the previous subsection, we will now prove that  \emph{cofactorizability}
of $E_2$-algebras~\cite[Section~3.3]{bjss} is sufficient for connectedness. 
Let us review the notion:
If $\smc$ has an internal hom $[-,-]$, then for any algebra $\cat{A}$ in $\smc$ the multiplication $\cat{A}\otimes\cat{A}\to\cat{A}$ induces a map 
$\cat{A}\to [\cat{A},\cat{A}]$. If $\cat{A}$ is an $E_2$-algebra, this map descends to the cocenter $\int_{\mathbb{S}^1 \times \opint } \cat{A}$ and yields an algebra  map
\begin{align}
	\int_{\mathbb{S}^1 \times \opint } \cat{A} \to [\cat{A}, \cat{A}] \ . \label{eqncofmap}
\end{align} 

\begin{definition}
	Let $\smc$ be a symmetric monoidal $(2,1)$-category with internal hom.
	An $E_2$-algebra $\cat{A}$
	is called \emph{cofactorizable} if~\eqref{eqncofmap} is an equivalence.
\end{definition}

\begin{remark}\label{remcof}
	Let $\cat{A}$ be a cyclic framed $E_2$-algebra in a symmetric monoidal $(2,1)$-category with internal hom.
	Then by direct inspection 
	the map
	\begin{align} \PhiA( \mathbb{D}^2 \times [0,1]   ):\int_{\mathbb{S}^1 \times \opint } \cat{A} \to \cat{A}\otimes \cat{A}   \label{eqnphimapcof}
	\end{align} agrees with~\eqref{eqncofmap} after the identification $[\cat{A}, \cat{A}]\simeq \cat{A}\otimes\cat{A}$ via the pairing. This implies that
	$\cat{A}$ is cofactorizable if and only if~\eqref{eqnphimapcof} is an equivalence. 
	Consequently, we will by definition refer to a cyclic framed $E_2$-algebra $\cat{A}$
	 as cofactorizable if~\eqref{eqnphimapcof} is an equivalence \emph{regardless of whether $\smc$ has internal homs}.
	 Since cofactorizability
	 is just the invertibility of
	  $\PhiA( \mathbb{D}^2 \times [0,1]   )$, it is a genus zero condition.
\end{remark}

\begin{proposition}\label{propstatic}
	For any cofactorizable
	cyclic framed $E_2$-algebra $\cat{A}$ in $\smc$,  any diffeomorphism $f:\Sigma \to \Sigma$
	of any surface $\Sigma$ acts on $\int_{\Sigma}\cat{A}$ by a 1-morphism
	that is, as $\int_{\partial \Sigma \times [0,1]}\cat{A}$-module
	map, isomorphic to the identity.
	In particular, cofactorizability implies connectedness.
\end{proposition}

\begin{proof}
	For the Dehn twist $d$ of the cylinder,
	Corollary~\ref{corxi} gives us an isomorphism
	\begin{align}
		\PhiA( \mathbb{D}^2 \times [0,1]   ) \circ d_* \cong \PhiA( \mathbb{D}^2 \times [0,1]   )
	\end{align}
	of maps
	of modules over $\left(\int_{ \mathbb{S}^1\times\opint     } \cat{A}\right)^{\otimes 2}$. 
	Since $	\PhiA( \mathbb{D}^2 \times [0,1]   )$ is assumed to be invertible, this implies that $d_*$ is isomorphic, as $\left(\int_{ \mathbb{S}^1\times\opint     } \cat{A}\right)^{\otimes 2}$-module map, to the identity.
	By Lemma~\ref{lemmaactfh} this implies that every mapping class acts trivially on factorization homology.
	Thanks to Lemma~\ref{lemmatrivcon}, this implies connectedness.
\end{proof}

\begin{theorem}\label{thmcofactorizable}
	For any cofactorizable cyclic framed $E_2$-algebra, there is an essentially unique extension to a modular functor. More precisely, we have
	an embedding of 2-groupoids
	\begin{align}
		\{\text{cofactorizable self-dual balanced braided algebras in $\smc$}\}	\ha{ \ \ \ }   \moduli\ , \quad \cat{A}\to\mathfrak{F}_\cat{A} \ . 
	\end{align}
\end{theorem}

\begin{proof}Proposition~\ref{propstatic}
	says that cofactorizable cyclic framed $E_2$-algebras are connected. The rest is a consequence of  Theorem~\ref{thmproofclassmf}.
	\end{proof}

\begin{remark}\label{remcofnotesssurj}
	This embedding is not an equivalence in general:
	If we take $\smc$ to be a symmetric monoidal 1-category, seen as a symmetric monoidal $(2,1)$-category with only identity 2-morphisms,
	then cyclic framed $E_2$-algebras in $\smc$ are just commutative Frobenius algebras, so that any of those gives rise to a two-dimensional topological field theory (hence to a modular functor). This is a consequence of $\pi_0(\Hbdy)\cong \pi_0(\Surf)$ and recovers the well-known classification of ordinary two-dimensional topological field theories~\cite{abrams,kocktft}. 
	But non-trivial Frobenius algebras are in general not cofactorizable.
\end{remark}

\subsection{Simplification of $\SurfA$\label{secsimplifysurfa}}
In this subsection we will calculate the fibers of the extensions appearing
in Theorem~\ref{thmdefsurfa2}
in relevant special cases.

The following notion is a natural strengthening of cofactorizability:

\begin{definition}\label{defphiinv}
	A cyclic framed $E_2$-algebra $\cat{A}$ in $\smc$ is called \emph{$\Phi$-invertible on a connected surface $\Sigma \in \Surf(T)$}
	if for some	  handlebody $H$ with boundary $\partial H=\Sigma$
	the 1-morphism $\PhiA(H):\int_{\partial H}\cat{A}\to \cat{A}^{\otimes\Legs(T)}$ is an equivalence.
	A cyclic framed $E_2$-algebra is called \emph{$\Phi$-invertible} if it is $\Phi$-invertible on all connected surfaces.
\end{definition}

\begin{remark}\label{remallH}
	If $\PhiA(H)$ is an equivalence for \emph{some} handlebody with boundary $\Sigma$, then this is true for \emph{all} handlebodies with boundary $\Sigma$.
	This follows from 
	Proposition~\ref{propenvfh2x}.
\end{remark}

\begin{remark}\label{remconnectedness}
	In detecting whether $\Omega_\cat{A}(\Sigma)$ is connected, the following implications are helpful:
	\begin{align}
		\text{$\cat{A}$ is $\Phi$-invertible at $\Sigma$} \quad\Longrightarrow\quad  \begin{array}{c} \text{$f_* \cong \id$ for all $f\in \Map(\Sigma)$}\\ \text{as module map}\end{array} \quad \Longrightarrow\quad 
		\text{$\Omega_\cat{A}(\Sigma)$ connected} \ . 
	\end{align}
	The proof of the first implication repeats essentially the argument of Proposition~\ref{propstatic}
	while the second one is Lemma~\ref{lemmatrivcon}.
\end{remark}

\begin{proposition}\label{propfullyfaithful} 
	Let $\cat{A}$ be a cyclic framed $E_2$-algebra.
	If $\cat{A}$ is \label{lemmafullyfaithfuli}
	$\Phi$-invertible on a connected surface $\Sigma\in\Surf(T)$
	(as defined in Definition~\ref{defphiinv}), then the extension
	of $\Map(\Sigma)$ defined by $p_\cat{A}:\SurfA \to \Surf$ is 
	of the form
	\begin{align}	1 \to \Aut_\partial(\id_{\int_\Sigma \cat{A}}) \to \MapA( \Sigma) \to \Map(\Sigma)\to 1 \ , \end{align}
	i.e.\ it is a central extension
	of $\Map(\Sigma)$ by the abelian group $\Aut_\partial(\id_{\int_\Sigma \cat{A}})$ of automorphisms of $\id_{\int_\Sigma \cat{A}}$ as $\int_{\partial\Sigma \times\opint } \cat{A}$-module map.
\end{proposition}

\begin{proof}
	We first note that Theorem~\ref{thmdefsurfa} combined with Remark~\ref{remconnectedness} tells us that we get extensions
	of the form
	\begin{align} 1 \to \Aut_\partial(\PhiA (H)  ) \to \MapA( \Sigma) \to \Map(\Sigma)\to 1\end{align} for $H$ with $\partial H=\Sigma$.
	Additionally, the module equivalence $\PhiA(H):\int_\Sigma \cat{A}\to\cat{A}^{\otimes \Legs(T)}$
	provides 
	us with an isomorphism 	$
	\Aut_\partial(\id_{\int_\Sigma \cat{A}})    \cong           \Aut_\partial(\PhiA(H)) 
	$ of groups. 
	Finally, we observe that $\Aut_\partial(\id_{\int_\Sigma \cat{A}})$ is abelian because it is a subgroup  of $ \Aut(\id_{\int_\Sigma \cat{A}})$ which by the bicategorical Eckmann-Hilton argument is abelian. By Lemma~\ref{lemmatrivcon} the extension is central.
\end{proof}

In the particular case that $\cat{A}$ is cofactorizable, one obtains a modular functor featuring 
extensions of mapping class groups
that are of a particularly nice form.
\begin{proposition}\label{propcof}
	Let $\cat{A}$ be a cofactorizable cyclic framed $E_2$-algebra. Then  the resulting
	extensions of mapping class groups of connected surfaces are of the form
	\begin{align}
		\label{extsesfactorizable}	1 \to \Aut(   \PhiA(\mathbb{B}^3)  ) \to \MapA( \Sigma) \to \Map(\Sigma)\to 1 \ . 
	\end{align}
\end{proposition}

\begin{proof} By 
	Proposition~\ref{propstatic}
	 $\SurfA$ is connected.
	It remains to prove that the resulting extensions of mapping class groups are of the form~\eqref{extsesfactorizable}:
	Let $\Sigma$ be a connected surface with $n+2$ boundary components and $\Sigma'$ the surface obtained by gluing two of these boundary components together. 
	With the 
	map
	$\PhiA( \mathbb{D}^2 \times \opint   ):\int_{\mathbb{S}^1 \times \opint } \cat{A} \to \cat{A}\otimes \cat{A}$
	that is in fact
	a map of
	$\left( \int_{\mathbb{S}^1\times\opint } \cat{A} \right)^{\otimes 2}$-modules by Proposition~\ref{propmodulemap},
	we obtain a square
	
	\footnotesize
	\begin{equation}\label{eqncofsc}
		\begin{tikzcd}
			\smc_\partial\left(  \int_\Sigma\cat{A},\cat{A}^{\otimes (n+2)}   \right) \ar[rrr,"\text{gluing, see proof of Theorem~\ref{thmdefsurfa}}"]\ar[dd,swap,"\simeq \ \text{via pairing $\kappa$}"]	&&&  	\smc_\partial\left(  \int_{\Sigma'}\cat{A},\cat{A}^{\otimes n}   \right) \ar[dd,"\simeq \ \text{via excision}"] \\ \\ 
			\smc_\partial\left(  \int_\Sigma\cat{A}    
			\otimes_{   \left( \int_{\mathbb{S}^1\times\opint } \cat{A} \right)^{\otimes 2} } 
			\cat{A}^{\otimes 2}   ,\cat{A}^{\otimes n}   \right)	           \ar[rrr, swap," \text{$ -\circ \PhiA( \mathbb{D}^2 \times \opint  ) $}       "]  &&&  \smc_\partial\left(  \int_\Sigma\cat{A}    
			\otimes_{   \left( \int_{\mathbb{S}^1\times\opint } \cat{A} \right)^{\otimes 2} } 
			\int_{\mathbb{S}^1 \times \opint } \cat{A}   ,\cat{A}^{\otimes n}   \right)	    \ . 
		\end{tikzcd}
	\end{equation}
	\normalsize
	It is straightforward to observe that it commutes up to natural isomorphism. 
	Since $\cat{A}$ is assumed to be cofactorizable, $\PhiA( \mathbb{D}^2 \times \opint   )$  is
	by definition an equivalence.  
	This implies that the lower horizontal map in the square is an equivalence. But then the upper horizontal map must be an equivalence as well.
	The gluing functor $\Omega_\cat{A}(\Sigma) \to \Omega_{\cat{A}}(\Sigma')$
	is by the construction in the proof of Theorem~\ref{thmdefsurfa} the restriction the 
	upper horizontal map in \eqref{eqncofsc}
	in domain and range. Therefore, it must be fully faithful. Moreover, it is  essentially surjective because both groupoids are connected. This makes
	$\Omega_\cat{A}(\Sigma) \to \Omega_{\cat{A}}(\Sigma')$
	an equivalence
	and proves after iterated application
	$\Omega_\cat{A}(\Sigma_{0,n+2g})\simeq \Omega_\cat{A}(\Sigma_{g,n}) $, where $\Sigma_{g,n}$ denotes the surface with genus $g$ and $n$ boundary components. 
	Thanks to the	insertion of vacua (Proposition~\ref{propinsertionofvacualemma}), we have additionally
	$\Omega_{\cat{A}}(\Sigma_{g,n})\simeq \Omega_\cat{A}(\Sigma_{g,0})$. This gives us
	$	\Omega_\cat{A}(\Sigma)\simeq \Omega_\cat{A}(\mathbb{S}^2) $ for every connected surface $\Sigma$. 
	Thanks to $\Aut(   \PhiA(\mathbb{B}^3)  )=\pi_1(\Omega_\cat{A}(\mathbb{S}^2))$, this proves that the extensions resulting from $\cat{A}$ are of the form~\eqref{extsesfactorizable}. 
\end{proof}



Let $\cat{A}$ be any cyclic framed $E_2$-algebra. The factorization homology for the sphere $\mathbb{S}^2$ 
is equivalent to $\cat{A}\otimes_{\int_{\mathbb{S}^1 \times \opint }\cat{A}} \cat{A}$ thanks to excision.
Under this identification, the map $\PhiA(\mathbb{B}^3):\int_{\mathbb{S}^2}\cat{A}\to I$
is the map   \begin{align} \cat{A}\otimes_{\int_{\mathbb{S}^1 \times \opint }\cat{A}} \cat{A} \label{PhiS2map} \to I\end{align} induced by the pairing $\kappa :\cat{A}\otimes\cat{A}\to I$.
This leads to the following notion that is analogous to cofactorizability:

\begin{definition}
	We call a cyclic framed $E_2$-algebra $\cat{A}$ \emph{co-non-degenerate}
	if it is $\Phi$-invertible on the sphere, i.e.\ if
	the map $\cat{A}\otimes_{\int_{\mathbb{S}^1 \times \opint }\cat{A}} \cat{A} 
	\to I$
	induced by the pairing of $\cat{A}$ is an equivalence.
\end{definition}

\begin{remark}\label{remktimes} 
	If $\cat{A}$ is co-non-degenerate, then $\Aut(   \PhiA(\mathbb{B}^3)  )\cong \Aut(\id_I)$, which for $\cat{S}=\Rex$ is just the group $\k^\times$ of units in the ground field $\k$.
	If the ambient symmetric monoidal $(2,1)$-category $\smc$ has internal homs, then application of $[-,I]$ to~\eqref{PhiS2map} yields the canonical map
	$I\to Z_2(\cat{A})$ from the monoidal unit of $\smc$ to the Müger center of $\cat{A}$  (we recalled the definition on page~\pageref{defmueger}). 
	This follows from the description of the Müger center in \cite[Definition~2.6]{bjss}
	as the internal endomorphisms of $\cat{A}$ as $\int_{\mathbb{S}^1 \times \opint }\cat{A}$-module. 
	This implies: If $\cat{A}$ is co-non-degenerate, then $\cat{A}$ is non-degenerate (i.e.\ $Z_2(\cat{A})$ is trivial).
	The converse holds if $\int_{\mathbb{S}^2}\cat{A}$ is dualizable in $\smc$, which applies for instance to a finite ribbon category. 
	This has an important consequence: A finite ribbon category, seen as cyclic framed $E_2$-algebra in $\Rex$, is co-non-degenerate
	if and only if its braiding is non-degenerate, i.e.\ if its Müger center $Z_2(\cat{A})$ is trivial.
	In this case, we find $\Aut(   \PhiA(\mathbb{B}^3)  )\cong \k^\times$. 
	Hence, the extensions appearing in~\eqref{extsesfactorizable} will really be central extensions by $\k^\times$. 
\end{remark}


\section{Applications and examples\label{secapp}}
This section is devoted to applications of the main result in the case that our ambient symmetric monoidal $(2,1)$-category $\smc$ is given by $\Rex$; this special case was covered in Section~\ref{exseccyclicframede2}.

\subsection{The spaces of conformal blocks of a $\Rex$-valued modular functor\label{subsecconfblock}}
Let $\cat{A}$ be a cyclic framed $E_2$-algebra in $\Rex$. In fact, we will assume that the underlying linear category of $\cat{A}$ is a finite linear category, as defined in Section~\ref{exseccyclicframede2}. This certainly restricts the generality, but for all of the examples that we want to discuss, it will be sufficient. Under these assumptions, $\cat{A}$ is in fact a Grothendieck-Verdier category in $\Rexf$.

\begin{corollary}\label{corconfblocks}
	Let $\cat{A}$ be a ribbon Grothendieck-Verdier in $\Rexf$.
	Then $\cat{A}$ extends in at most one way to a modular functor.
	In that case, the
	space of conformal blocks for the surface of genus $g$ and $n$ boundary components labeled with $X_1,\dots,X_n \in \cat{A}$
	is isomorphic to \begin{align} \label{eqnconfblock} \cat{A}(X_1 \otimes \dots \otimes X_n \otimes \mathbb{A}^{\otimes g},K)^*\ , 
	\end{align} where $K\in\cat{A}$ is the dualizing object and $\mathbb{A}\in\cat{A}$ is the result of applying the monoidal product to the end $\int_{X\in\cat{A}} X \boxtimes DX\in\cat{A}\boxtimes\cat{A}$. 

\begin{pnum}
	\item A sufficient condition\label{corconfblocksi}
	for the existence of the modular functor 
	extension is cofactorizability of $\cat{A}$. 
	\item 
The spaces of conformal blocks are always finite-dimensional.\label{corconfblocksii}
\end{pnum}

	\end{corollary}

\begin{proof}
	The uniqueness statement is Theorem~\ref{uniquenessthm}, and the existence in the cofactorizable case, i.e.\ statement~\ref{corconfblocksi}, is~Theorem~\ref{thmcofactorizable}.
The formula  for the vector space~\eqref{eqnconfblock}
 follows immediately from the description of the underlying ansular functor, i.e.\ 
the $\Hbdy$-algebra, see \cite[Corollary~6.3]{mwansular} (note that those results refer to $\Lexf$; we need to dualize them to $\Rexf$).
Statement~\ref{corconfblocksii} follows from~\eqref{eqnconfblock}.\end{proof}

\begin{remark}
We should warn the reader
 that the identification of the space of conformal blocks with the vector space~\eqref{eqnconfblock},
while being very explicit, uses non-canonical choices. This problem is often ignored in the classical literature on modular functors, see \cite[Remark~7.10]{cyclic} for a more detailed comment. 
\end{remark}

\subsection{Modular categories}
Let us now treat the case of modular categories, as defined in Example~\ref{exfiniteribboncategories}:

\begin{corollary}\label{cormodular}
	For any modular category $\cat{A}$, seen as non-cyclic framed $E_2$-algebra in $\Rex$, there is an essentially unique extension to a modular functor, and this modular functor is equivalent to Lyubashenko's modular functor.
	The 2-group of ribbon autoequivalences of $\cat{A}$ acts, up to coherent isomorphism, on this modular functor. 
	\end{corollary}

This proves that for a given modular category the Lyubashenko construction is the essentially unique possible modular functor construction
 in a precise sense. To the best of our knowledge, this was previously not known.

\begin{proof}
Since $\cat{A}$ is a finite ribbon category, it is a cyclic framed $E_2$-algebra in $\Rex$ by Example~\ref{exfiniteribboncategories}.
Moreover, this is the unique cyclic structure by \cite[Corollary~4.4]{mwcenter} because the braiding of $\cat{A}$ is non-degenerate.

	Thanks to \cite[Theorem~1.1]{shimizumodular}, $\cat{A}$ is factorizable and then also cofactorizable by \cite[Theorem~1.6]{bjss}. 
	Now Theorem~\ref{thmcofactorizable} gives us an essentially unique extension to a modular functor.
	The  extensions of mapping class groups are central extensions by the abelian group $\k^\times$ by Remark~\ref{remktimes}.
	By Theorem~\ref{uniquenessthm} this modular functor must agree, up to equivalence with Lyubashenko's modular functor, provided that Lyubashenko's modular functor in genus zero actually agrees with $\cat{A}$ as cyclic framed $E_2$-algebra. This is indeed the case by \cite[Proposition~7.14]{cyclic}.
	
	For the additional statement on the action of the
	2-group $\catf{AUT}_{\framed}(\cat{A})$
	of ribbon autoequivalences of $\cat{A}$
	on the modular functor,
	we observe that ribbon autoequivalences are automatically compatible with the duality, so that they come automatically 
	 with the structure of a cyclic automorphism; in other words, the forgetful map
	$\catf{cAUT}_{\framed}(\cat{A})\to \catf{AUT}_{\framed}(\cat{A})$
	 from cyclic $\framed$-automorphisms
	to non-cyclic ones has a section
	$s:\catf{AUT}_{\framed}(\cat{A})\to \catf{cAUT}_{\framed}(\cat{A})$, see also~\cite[Corollary~4.6]{mwcenter}. 
	Now the action of 
	$\catf{AUT}_{\framed}(\cat{A})$ on the modular functor $\mathfrak{F}_\cat{A}$ is given by
	\begin{align}
		\catf{AUT}_{\framed}(\cat{A})\ra{s}\catf{cAUT}_{\framed}(\cat{A}) \stackrel{\text{Corollary~\ref{corautos}}}{\simeq} \catf{AUT}(\mathfrak{F}_\cat{A}) \ . 
		\end{align}
	\end{proof}

Since, in the semisimple case, the spaces of conformal blocks of this modular functor are exactly the state spaces of the Reshetikhin-Turaev construction, 
Corollary~\ref{cormodular} is in line with~\cite{akz}, where the Reshetikhin-Turaev state spaces of an anomaly-free 
 modular \emph{fusion} category (i.e.\ a semisimple modular category)	 ---
as vector spaces, not as mapping class group representations ---
are produced from factorization homology.

\begin{example}[The Hopf algebra case]
	If a modular category $\cat{A}$ arises as finite-dimensional modules over a ribbon factorizable Hopf algebra $H$, the spaces of conformal blocks $\cat{A}(X_1 \otimes \dots \otimes X_n \otimes \mathbb{A}^{\otimes g},K)^*$ 
	 reduce to the familiar ones $\Hom_H(X_1\otimes\dots\otimes X_n\otimes H_\text{ad}^{\otimes g},k)^*$, where $H_\text{ad}$ is the adjoint representation of $H$ and $k$ is the ground field with $H$-action via the counit. This is dual to the statement in~\cite[Corollary~7.13]{cyclic}.
	\end{example}

	\subsection{Vertex operator algebras}
Using the classification of modular functors we can give a universal construction for spaces of conformal blocks 
from a suitable vertex operator algebra. We refer to~\cite{frenkelbourbaki} for one possible introduction to the topic of vertex operator algebras.
Let $V$ be a $C_2$-cofinite vertex operator algebra
and fix a notion of $V$-module that meets the requirements listed in~\cite[Theorem~2.12]{alsw}, so that the category of $V$-modules becomes a ribbon Grothendieck-Verdier category in $\Rexf$ that we denote by $\Vmod$. 
We will not list these conditions here and refer to
\cite{alsw} for the details, but these conditions are roughly speaking the weakest known conditions that make $\Vmod$ braided monoidal and 
 closed under taking the contragredient representation $X \mapsto X^*$. 

\begin{corollary}[Universal construction of spaces of conformal blocks for a vertex operator algebra]\label{corvoax}
	Let $V$ be a $C_2$-cofinite vertex operator 
	together with a notion of $V$-modules, subject to the conditions just mentioned.
	Then we have the following universal construction procedure for spaces of conformal blocks:
	We define spaces of
	conformal blocks in genus zero  by
	sending an $n$-holed sphere whose boundary 
	components are labeled by $V$-modules $X_1 , \dots , X_n$ 
	to the 
	dual hom space \begin{align}
	\Hom_V (X_1\otimes \dots \otimes X_n , V^*)^*\label{genuszeroblockseqn}\end{align} while the mapping class groups in genus zero 
	(ribbon braid groups)
	act through the ribbon Grothendieck-Verdier structure present on $\Vmod$.
	Then the following statements hold:
	
	\begin{pnum}
		
		\item This definition of spaces of genus zero conformal blocks (understood as cyclic $\Surf_0$-structure on $\Vmod$)
		is the unique possible one that extends the balanced braided structure of $\Vmod$ if
		the Müger center of $\Vmod$ is trivial (for example if $\Vmod$ is modular).\label{voanum1}
		
		\item There is, up to equivalence, at most one extension of the genus zero blocks~\eqref{genuszeroblockseqn}
		to a modular functor. The higher genus blocks are then given by the formula~\eqref{eqnconfblock}.\label{voanum2}
		
		\item The unique extension to a modular functor mentioned under~\ref{voanum2} exists if and only $\Vmod$ is connected. A sufficient condition is cofactorizability 
		(which is satisfied for instance if $\Vmod$ is modular). \label{voanum3}
		
		\item If $\Vmod$ fails to be connected, then the higher genus blocks still exist and are compatible
		with gluing
		 (sometimes called `factorization'). They will still carry a representation of handlebody groups instead of mapping class groups of surfaces.\label{voanum4}
		\end{pnum}
	\end{corollary}

\begin{proof}
The hypotheses guarantee that $V\text{-mod}$ is a ribbon Grothendieck-Verdier category in $\Rexf$
by \cite[Theorem~2.12]{alsw}
 and hence a cyclic framed $E_2$-algebra in $\Rexf$ and also in $\Rex$
 by the results from~\cite{cyclic} recalled in   Section~\ref{exseccyclicframede2}.
 The genus zero blocks are then given by~\eqref{genuszeroblockseqn}, as follows from Corollary~\ref{corconfblocks}.
 
 If the Müger center of $V\text{-mod}$ is trivial, then so is the balanced Müger center. Therefore, this cyclic structure is the only one by \cite[Corollary~4.4]{mwcenter}. This proves~\ref{voanum1}. 
 
 Moreover, Theorem~\ref{uniquenessthm} and 
 Corollary~\ref{corconfblocks} give us \ref{voanum2} while the Theorems~\ref{thmproofclassmf} and~\ref{thmcofactorizable} prove~\ref{voanum3}.
 Finally, \cite[Theorem~5.9]{mwansular} implies~\ref{voanum4}.
	\end{proof}

\begin{corollary}
	The ansular functor from Corollary~\ref{corvoax}~\ref{voanum4} built from the vertex operator algebra $V$
	extends always, regardless of connectedness of the category of $V$-modules, 
	to a modular $\Surf_{\Vmod}$-algebra that is the best possible approximation to an extension of the cyclic framed $E_2$-algebra
	$\Vmod$
	to a modular functor in the sense that any possibly existing extension of $\Vmod$ to a modular functor  factors essentially uniquely through this modular  $\Surf_{\Vmod}$-algebra structure.
	\end{corollary}
	
	\begin{proof}
		This is a consequence of 
		Theorems~\ref{thmfa2} and~\ref{thmuniversal}.\end{proof}

\begin{remark}\label{comparisonvoa} 
	For a given vertex operator algebra $V$, one can attempt to build spaces of conformal blocks/a modular functor
	in at least two ways:
	\begin{pnum}
		\item By forming a suitable category of modules over $V$ (this usually involves making some choices) to obtain a nice representation category (which in good cases will be a modular category)
		that can be inserted into the available modular functor constructions \cite{turaev,baki,lyubacmp}.  \label{mfcon1}
		
		\item By performing constructions directly using the vertex operator algebra, see the monograph of Ben-Zvi-Frenkel \cite{frenkelbenzvi} or, as a summary, \cite{frenkelbourbaki}, and additionally the construction of Damiolini-Gibney-Tarasca~\cite{dgt-voa}. \label{bzf2point}
		
	\end{pnum}
	The comparison between these approaches is a major open problem in conformal field theory, see e.g.\ the review article~\cite[Section~3.2 \& 3.3]{frs25}.
	We will not attempt here to describe the construction~\ref{bzf2point} within the framework of this article
	(it is not at all clear that this is in fact possible), but 
	Corollary~\ref{corvoax} provides a potentially powerful comparison tool by introducing
	in a sense a third construction which is doubtlessly closer to~\ref{mfcon1}, but is ultimately just characterized by a universal property and not tied to certain specific construction or even a certain `ansatz' of what the spaces of conformal blocks at a certain genus should look like. 
	If one then exhibits arbitrary constructions of a modular functor, along the lines of~\ref{mfcon1}, \ref{bzf2point} or even something else, such that they all agree in genus zero --- as cyclic framed $E_2$-algebras --- then all these constructions agree with the universal one up to equivalence. 
	 However, let us emphasize again that it is by no means obvious that a construction following~\ref{bzf2point} actually produces a modular functor in the sense of this article. 
	 In fact, even for nice classes of rational vertex operator algebras $V$, the gluing property (`factorization') for the the blocks constructed directly from $V$ is extremely non-trivial and a major and recent result~\cite{damgibtar}.
	 In contrast to that, the blocks appearing in Corollary~\ref{corvoax} always glue correctly.  
\end{remark}

The cofactorizability of $V\catf{-mod}$ is an algebraic condition that one could hope to verify
 in relevant examples, but an in-depth investigation of this point is beyond the scope of this article.
 Nonetheless, we want to give at least one example that goes beyond the situation of
  Corollary~\ref{cormodular}:

\begin{example}[Feigin-Fuchs boson]\label{exfeiginfuchs}
	Let $\Psi=(\mathfrak{h},\spr{-,-},\Lambda,\xi)$ be bosonic lattice data in the sense of \cite[Definition~3.1]{alsw} as recalled in Example~\ref{exvoa}.
	Denote by $\catf{VM}(\Psi)$ the associated ribbon Grothendieck-Verdier category, i.e.\ the category associated to the \emph{Feigin-Fuchs boson}.
	The braiding of this category is non-degenerate. 
	This is a consequence of the fact that the bilinear form $\spr{-,-}$ is assumed to be non-degenerate.
	The underlying braided monoidal category of 
	$\catf{VM}(\Psi)$ is rigid
	(but careful: the rigid duality is generally not the ribbon Grothendieck-Verdier duality of $\catf{VM}(\Psi)$).
	By the \cite[Theorem~1.1]{shimizumodular} and \cite[Theorem~1.6]{bjss}
	this implies
	that $\catf{VM}(\Psi)$ is cofactorizable (recall that this is a property of the underlying $E_2$-algebra).
		Now Corollary~\ref{corvoax}
		 tells us that the ribbon Grothendieck-Verdier 
	category describing the Feigin-Fuchs boson admits a unique
	 extension to a modular functor. We should remark that this category is not modular unless $\xi=0$.
	The ribbon Grothendieck-Verdier category $\catf{VM}(\Psi)$ has a group-cohomological description by \cite[Theorem~3.12]{alsw} (we have recalled this in Example~\ref{exvoa}): 
	$\catf{VM}(\Psi) \simeq \vect_{\Lambda^* / \Lambda}^{\omega(\Psi),2\xi}$.
	The underlying $\Hbdy$-algebra for group-cohomological ribbon Grothendieck-Verdier categories  is given in detail \cite[Example~7.12]{cyclic}. 
	In particular, this allows us to conclude that the space of
	conformal blocks
	 associated to a closed surface of genus $g$ has dimension
	$|\Lambda^* / \Lambda|^g$ if $2(g-1)\xi=0$. Otherwise, it is zero.
	In \cite{alsw} a suggestion of a space of conformal blocks for the Feigin-Fuchs boson
	is already made in the case of the torus, and the question is raised whether a full modular functor can be built. The considerations of this example prove that this is indeed the case. 
\end{example}

\subsection{Drinfeld centers of possibly non-spherical pivotal finite tensor categories}
We conclude with an example built from categories which are neither modular nor necessarily semisimple:
\begin{corollary}\label{corzentrum}
	Let $\cat{A}$ be a pivotal finite tensor category.
	Then the Drinfeld center $Z(\cat{A})$ comes naturally with a cyclic framed $E_2$-structure that admits an essentially  unique extension to a modular functor.
	This is, up to equivalence, the only modular functor whose underlying balanced braided monoidal category is $Z(\cat{A})$ with its usual balanced braided structure.
	\end{corollary}

\begin{proof}
	By \cite[Theorem~2.12]{mwcenter} we have a cyclic framed $E_2$-structure (equivalently, ribbon Grothendieck-Verdier structure)
	 on $Z(\cat{A})$,
	 namely the one with 
	the distinguished invertible object of $\cat{A}$~\cite{eno-d} as dualizing object of $Z(\cat{A})$. This cyclic framed $E_2$-structure is the only one whose underlying non-cyclic framed $E_2$-structure is the usual balanced braided structure of $Z(\cat{A})$.
	
	Now once again, we use that $Z(\cat{A})$ is not only factorizable by \cite[Proposition~4.4]{eno-d}, but in fact cofactorizable \cite[Theorem~1.6]{bjss} and apply Theorem~\ref{thmcofactorizable}.
	\end{proof}

For the cyclic framed $E_2$-structure on $Z(\cat{A})$ from Corollary~\ref{corzentrum}, we have
\begin{align}
Z(\cat{A}) \ \text{is modular} 
 \quad \Longleftrightarrow \quad \cat{A} \ \text{is spherical} \ , 
\end{align} by \cite[Corollary~2.13]{mwcenter}. If $Z(\cat{A})$ is modular, then the associated 
  modular functor  is 
   a special case of the modular functor associated to a modular category, but Corollary~\ref{corzentrum} does \emph{not} assume that $\cat{A}$ is spherical. 
	We refer to \cite[Section~3]{mwcenter} for a description of the underlying ansular functor and the discussion of non-spherical examples. 
	
	The modular functor for $Z(\cat{A})$, with $\cat{A}$ being a pivotal finite tensor category, should be thought of as a generalization of a modular functor of Turaev-Viro type. Indeed if $\cat{A}$ is a spherical fusion category, it coincides with the Reshetikhin-Turaev modular functor of the modular category $Z(\cat{A})$ by Corollary~\ref{cormodular}, and therefore with the Turaev-Viro type	 modular functor of $\cat{A}$~\cite{tuvi,balsam}. In that case, there is a string-net description~\cite{kirillovsn,fsy}. For non-spherical fusion categories, one would expect a relation to the non-spherical string-nets in \cite{Non-s-string-nets}.
	For a pivotal finite tensor category  $\cat{A}$ 
	that is not necessarily fusion, it is conjectured in \cite[Conjecture~3.5.11]{dsps} and the comments following afterwards
	that $\cat{A}$ gives rise to a three-dimensional combed local field theory
	(a combing on a three-dimensional manifold is a choice of a non-vanishing vector field). 
	Corollary~\ref{corzentrum} provides the underlying modular functor. If one drops the pivotal structure, then at least a framed modular functor is guaranteed through
\cite[Theorem~3.2.2 and~Corollary~3.2.3]{dsps} via the cobordism hypothesis. 
A state sum construction of this framed modular functor is given in \cite{fssstatesum}. 
The modular functor from Corollary~\ref{corzentrum} should admit a description that is an oriented extension of this state sum construction, see also \cite[Remark~5.28]{fssstatesum}.

	\small	

\newcommand{\etalchar}[1]{$^{#1}$}

\vspace*{0.3cm} \noindent \textsc{Institut de Mathématiques de Jussieu-Paris Rive Gauche, UMR 7586, Université Paris Cité, Sorbonne Université,
	CNRS, F-75013 Paris, France} \\[2ex] \noindent \textsc{Université Bourgogne Europe, CNRS, IMB UMR 5584, F-21000 Dijon, France}

\end{document}